\documentclass[11pt,a4paper,oneside]{amsart}

\usepackage{hyperref}

\usepackage[foot]{amsaddr}
\usepackage{amsmath}
\usepackage{amssymb}
\usepackage{amsthm}
\usepackage{cleveref}
\usepackage{comment}
\usepackage[usenames]{color}
\usepackage{enumitem}
\usepackage{geometry}
\usepackage{mathrsfs}
\usepackage{mathtools}
\usepackage[active]{srcltx}
\usepackage[normalem]{ulem}
\usepackage{aligned-overset}

\usepackage[
    type={CC},
    modifier={by},
    version={4.0},
    imagewidth={3.6em},
    imagemodifier={-80{x}15},
]{doclicense}

\theoremstyle{plain}
\newtheorem{lemma}{Lemma}[section]
\newtheorem{theorem}[lemma]{Theorem}
\newtheorem{proposition}[lemma]{Proposition}
\newtheorem{corollary}[lemma]{Corollary}

\theoremstyle{definition}

\newtheorem{definition}[lemma]{Definition}
\newtheorem{remark}[lemma]{Remark}
\newtheorem{example}[lemma]{Example}

\numberwithin{equation}{section}

\let\temp\phi%
\let\phi\varphi%
\let\temp\epsilon%
\let\epsilon\varepsilon%
\let\varepsilon\temp%

\newcommand{\R}{\mathbb{R}}
\newcommand{\N}{\mathbb{N}}

\DeclareMathOperator{\Dom}{Dom}
\DeclareMathOperator{\Ent}{Ent}
\DeclareMathOperator{\U}{U}

\DeclareMathOperator{\Graph}{Graph}

\DeclareMathOperator{\OptCaus}{OptCaus}
\DeclareMathOperator{\OptGeo}{OptGeo}
\DeclareMathOperator{\Caus}{Caus}

\DeclareMathOperator{\Ric}{Ric}

\DeclareMathOperator{\dist}{dist}

\DeclareMathOperator{\dom}{Dom}

\DeclareMathOperator{\id}{Id}

\DeclareMathOperator{\supp}{supp}

\DeclareMathOperator{\vol}{Vol}

\newcommand{\de}{{\mathrm{d}}}

\newcommand{\CD}{\mathsf{CD}}

\newcommand{\MCP}{\mathsf{MCP}}
\newcommand{\NC}{\mathsf{NC}}

\newcommand{\ve}{\varepsilon}
\newcommand{\weak}{\rightharpoonup}

\newcommand{\Prob}{\mathcal{P}}

\newcommand{\M}{\mathcal{M}}

\newcommand{\norm}[1]{\left\Vert#1\right\Vert}

\renewcommand{\L}{\mathcal{L}}

\newcommand{\q}{\mathfrak{q}}

\newcommand{\TMCP}{\mathsf{TMCP}}
\newcommand{\TCD}{\mathsf{TCD}}

\newcommand{\QQ}{\mathfrak{Q}}

\newcommand{\gflow}{\Psi}

\newcommand{\mm}{\mathfrak{m}}

\newcommand{\qq}{\mathfrak{q}}
\newcommand{\relation}{\mathcal{R}}

\newcommand{\sfd}{\mathsf{d}}

\newcommand{\vfield}{\mathfrak{X}}
\newcommand{\normal}{\mathscr{N}}

\newcommand{\staticWithoutH}{\mathfrak{s}}
\newcommand{\staticWithH}{\staticWithoutH(H)}
\newcommand{\static}{\staticWithoutH}

\newcommand{\initialWithoutH}{\mathfrak{a}}
\newcommand{\initialWithH}{\initialWithoutH(H)}
\newcommand{\initial}{\initialWithoutH}

\newcommand{\finalWithoutH}{\mathfrak{b}}
\newcommand{\finalWithH}{\finalWithoutH(H)}
\newcommand{\final}{\finalWithoutH}

\newcommand{\NEC}{Null Energy Condition}

\newcommand{\ascar}{tight}

\setlist[enumerate]{leftmargin=1cm}
\setlist[itemize]{leftmargin=1cm}

\newcommand{\noending}[1]{{#1}^\circ}

\title{On the geometry of synthetic null hypersurfaces}

\author{Fabio Cavalletti}
\email{fabio.cavalletti@unimi.it}
\author{Davide Manini}
\email{dmanini@campus.technion.ac.il}
\author{Andrea Mondino}
\email{andrea.mondino@maths.ox.ac.uk}

\begin{document}

\begin{abstract}
This paper develops a synthetic framework for the geometric and analytic study of null (lightlike) hypersurfaces in non-smooth spacetimes. Drawing from optimal transport and recent advances in Lorentzian geometry and causality theory, we define a \emph{synthetic null hypersurface} as a triple \((H, G, \mathfrak{m})\): \(H\) is a closed achronal set in a topological causal space, \(G\) is a gauge function encoding affine parametrizations along null generators, and \(\mathfrak{m}\) is a Radon measure serving as a synthetic analog of the rigged measure. This generalizes classical differential geometric structures to potentially singular spacetimes.

The central object is the \emph{synthetic null energy condition} (\(\mathsf{NC}^e(N)\)), defined via the concavity of an entropy power functional along optimal transport, with parametrization given by the gauge $G$. This condition is invariant under changes of gauge and measure within natural equivalence classes. It agrees with the classical Null Energy Condition in the smooth setting and it applies to low-regularity spacetimes. A key property of the \(\mathsf{NC}^e(N)\) condition is the stability under convergence of synthetic null hypersurfaces.

The \(\mathsf{NC}^e(N)\) condition is also remarkably fruitful for applications. First, it provides a framework for a synthetic version of Hawking's area theorem. Second, the celebrated Penrose's singularity theorem is proved for continuous spacetimes, and the existence of trapped regions is settled in the general setting of topological causal spaces satisfying the  \(\mathsf{NC}^e(N)\).
\end{abstract}

\subjclass{53C50; 83C75, 49Q22, 53C23}
\keywords{null energy condition, null hypersurface,
  optimal transport, degenerate cost, Hawking area theorem,
  Penrose singularity theorem}

\maketitle
\vspace{-1.3 cm}
\tableofcontents

\section{Introduction}

Null (or lightlike) hypersurfaces are fundamental objects in
Lorentzian geometry and mathematical relativity, playing a central
role in the causal structure of spacetimes and in the formulation of
physically significant results such as  Hawking's
area
theorem~\cite{Haw71} and Penrose's singularity
theorem~\cite{Penrose65}. Traditionally, the analysis of such
hypersurfaces---defined as co-dimension one submanifolds whose normal
vector field is null---relies on the smooth differential-geometric
setting. However, the increasing interest in low-regularity
spacetimes, including those with continuous metrics or metric-measure
structures, necessitates the development of a robust synthetic
framework capable of encoding the geometric and causal features of
null hypersurfaces beyond smooth regimes.

The  goal of this paper is to introduce and develop a synthetic theory of null hypersurfaces and the \NEC\ (NEC) in non-smooth spacetimes, within the general setting of topological causal spaces~\cite{CausalSpace, KS}. The main outcomes are:
\begin{itemize}
\item The \emph{well-posedness} and \emph{covariance/invariance} of
  the synthetic notions;
\item The \emph{compatibility} of the synthetic constructions of the present paper with the classical smooth notions;
\item The \emph{stability} of the synthetic NEC under a suitable
  extension of the measured Gromov{--}Hausdorff
  convergence to synthetic null hypersurfaces;
\item  \emph{Extensions} of  certain classical results,
  namely
  Hawking's area theorem and Penrose's singularity theorem to settings
  of extremely low regularity (the former to topological causal spaces, the latter to continuous spacetimes).
\end{itemize}

\subsection{Motivations}
\subsubsection*{Null hypersurfaces}
 Often, the null hypersurfaces appearing in applications (e.g., as Cauchy horizons)  lack smoothness. This motivates the development of a theory for null hypersurfaces under minimal regularity assumptions.
 Recall that a hypersurface \( H \subset (M,g) \), where \( (M,g) \)
 is a Lorentzian manifold, is said to be \emph{null} if the
restriction of the Lorentzian metric \( g \) to \( H \) is degenerate i.e., it has a one-dimensional kernel at every point.
Equivalently, the normal vector field to \(H \) is also tangent to \(H \). Null hypersurfaces naturally arise in general relativity: they encode the propagation of light-like signals and gravitational radiation, they serve as geometric models for event horizons, Cauchy horizons, and null infinity. 
A deeper understanding of their geometry may shed new light on fundamental problems in general relativity, for instance see the progress on the classification of Cauchy horizons~\cite{MoncIsenCMP1983, BustReiGRG21, GuMinCMP2022}. 
Furthermore, null hypersurfaces feature prominently in theoretical frameworks such as holography and the AdS/CFT correspondence, where boundary structures often possess null character.

\subsubsection*{\NEC}
The \emph{\NEC\ (NEC)}, in its geometric form, requires that the Ricci tensor  satisfies
\begin{equation}\label{eq:NECIntro}
  \Ric(v,v) \geq 0, \quad
  \text{for all null\;
    (i.e., light-like)
    vectors \( v \).}
\end{equation}
 This condition encapsulates the classical expectation that the
 effective energy density measured along lightlike directions is
 nonnegative.
 The NEC arises as a central assumption in several foundational
 results in Lorentzian geometry and general relativity. In the
 classical (non-quantum) regime, it is expected to hold for all
 physically reasonable matter models (cf.~\cite[Ch.~4.6]{Carroll}).
 Its mathematical significance is underscored by the following applications:

\begin{enumerate}
 \item \emph{Black Hole Area Theorem}: The NEC underpins the proof of Hawking's area theorem~\cite{Haw71}, which asserts the non-decreasing nature of the event horizon area of a black hole in classical general relativity. This result serves as a geometric analog of the second law of thermodynamics. 
    \item \emph{Penrose's Singularity Theorem}: The NEC is a key
      hypothesis in Penrose's singularity theorem~\cite{Penrose65},
      which guarantees geodesic incompleteness under conditions of
      gravitational collapse and the existence of trapped
      surfaces. This theorem represents a foundational result in the
      mathematical theory of black holes, recognized by the
      2020 Nobel Prize in Physics.
\end{enumerate}

\subsubsection*{Singular spacetimes}
As already anticipated in the foundational work of Lichnerowicz in the 1950s~\cite{Lichn}, many physically realistic models predict spacetimes with matter fields that yield energy-momentum tensors of low regularity. Through the Einstein equations, this leads to Lorentzian metrics of regularity below \( C^2 \). Notable examples include matched spacetimes modeling stellar interiors and exteriors~\cite{MaSe}, self-gravitating compressible fluids~\cite{BuLe}, and shock wave solutions. Certain models require even lower regularity, such as impulsive gravitational waves (see Penrose's original construction~\cite{PenGW}, and the more recent~\cite[Ch.~20]{GrPo}), and cyclic cosmologies~\cite{LLV}.

Ultimately, a long-term motivation for considering Lorentzian geometries with weak regularity stems from the expectation that at quantum gravitational scales---such as near black-hole singularities or the origin of the universe---the classical manifold model of spacetime may break down. In these regimes, spacetime may exhibit extreme irregularity, potentially eluding approximation by smooth structures.

In the setting of Lorentzian metrics of low regularity, curvature is
typically defined in the sense of distributions, leveraging the smooth
background manifold. This approach has been developed, for instance,
by Geroch and Traschen~\cite{GeTr}, allowing the definition of
distributional curvature tensors for \( W^{1,2}_{\text{loc}} \)
Lorentzian metrics satisfying an appropriate non-degeneracy
condition---automatically satisfied,  e.g., for $C^1$-metrics (cf.~\cite{Graf}). 

A key objective in this paper is to address curvature---not only in the sense of generalized tensors on a smooth underlying manifold, but in settings where \emph{the spacetime itself is singular}, and neither the metric tensor nor the ambient differentiable structure is assumed to be smooth.

\subsubsection*{Why extending singularity theorems to non-smooth spacetimes}

The singularity theorems constitute a cornerstone of Lorentzian geometry and general relativity, rigorously demonstrating that spacetime singularities---understood as causal geodesic incompleteness---arise generically under physically reasonable conditions. Classic results such as those of Penrose~\cite{Penrose65} and Hawking~\cite{Haw:67} establish incompleteness in scenarios involving complete gravitational collapse or cosmological expansion, respectively. These foundational theorems, developed in the late 1960s, were originally formulated under the assumption that the spacetime metric \( g \) is smooth.

However, already in their seminal monograph~\cite{HawEll}, Hawking and Ellis raised concerns about the regularity assumptions inherent in these theorems. In particular, they emphasized that a lack of low-regularity generalizations would limit the physical robustness of the conclusions: if singularity theorems cease to hold for metrics below the \( C^2 \)-threshold, then spacetime incompleteness---hence singularity in the standard sense (cf.~\cite[p.~10]{Cla:98},~\cite[Sec.~8.1]{HawEll})---could potentially be avoided by a mere drop in regularity. For instance, a \( C^{1,1} \)-metric would imply, via the Einstein equations, at most a finite discontinuity in the matter fields, which is hardly pathological from a physical standpoint.

While regularity issues surrounding the singularity theorems have been noted for decades (cf.~\cite[Sec.~6.2]{Sen:98}), substantial progress has only recently been achieved. On the one hand, significant developments in Lorentzian causality theory have extended the framework to encompass continuous metrics~\cite{Chr-Grant, SaC0}, and even more abstract settings~\cite{KS, Min}. On the other hand, advances in analytic techniques---especially convolution-based regularization---have led to rigorous proofs of Penrose's singularity theorem for $C^1$ metrics~\cite{Graf} and of Hawking's singularity theorem for \( C^{0,1} \) metrics~\cite{CGHKS-25}.

The present work pushes this frontier decisively further by establishing a version of Penrose's singularity theorem for continuous metrics, thereby encompassing a broad spectrum of physically relevant, non-smooth spacetimes.

Let us mention that Hawking's singularity theorem was extended by the first and third author~\cite{CaMo:20} to the setting of Lorentzian length-spaces satisfying timelike Ricci lower bounds in a synthetic sense; see also the more recent extension by Braun{--}McCann~\cite{Braun-McCann} to synthetic variable lower bounds on the timelike Ricci. However, as detailed below, the null setting of Penrose's singularity theorem  poses considerable additional challenges, that are addressed in the present work.

\subsection{Previous contributions on NEC and optimal trasport}

A first synthetic characterization of the NEC for smooth spacetimes was obtained by McCann~\cite{McCann-NEC}, as a limiting procedure of timelike Ricci lower bounds. This approach builds on  top of the optimal transport characterization of timelike Ricci lower bounds obtained for smooth spacetimes in~\cite{McCann, MoSu} and the synthetic $\mathsf{TCD}(K,N)$ condition  studied in~\cite{CaMo:20} (see also~\cite{Braun, CM24-IsopLor, Braun-McCann, Octet} for more recent developments) in the framework of Lorentzian length spaces introduced by Kunzinger{--}S\"amann~\cite{KS}.  The advantage of the characterization of the NEC obtained in~\cite{McCann-NEC} is that it extends to non-smooth Lorentzian length spaces however, as already remarked in~\cite{McCann-NEC},  it is not stable under convergence. 

A distinct approach was later proposed by Ketterer~\cite{Ket24}, who first studied optimal transport inside null hypersurfaces and considered displacement convexity of entropy for \emph{singular} measures supported on co-dimension two spacelike submanifolds embedded in null hypersurfaces.

After~\cite{Ket24}, the authors~\cite{CMM24a} established
that~\eqref{eq:NECIntro} can be equivalently reformulated in terms
of entropy convexity for \emph{diffused} probability measures
supported on null hypersurfaces.
More precisely,~\cite{CMM24a} provided a systematic study, in smooth spacetimes,  of the interplay between Ricci curvature in null directions, rigged measures on null hypersurfaces, and optimal transport of probability measures on null hypersurfaces which are absolutely continuous with respect to a rigged measure.

\subsection{A novel theory for synthetic null
    hypersurfaces}

In the present work, we use the smooth characterization of the NEC in terms of optimal transport provided in~\cite{CMM24a} as a starting point to develop a synthetic theory of  null hypersurfaces and of the \NEC\ in non-smooth spacetimes.

\subsubsection*{Synthetic null hypersurfaces}

The building blocks of the novel synthetic framework for null hypersurfaces are:
\begin{itemize}
\item a closed achronal set \( H \) within a topological causal space \( (X, \ll, \leq, \mathcal{T}) \),
\item a gauge function \( G : H^\circ \to \mathbb{R} \) encoding affine parametrization along causal curves (where $ H^\circ\subset H$ denotes the subset of intermediate points of the null generators of $H$, see Definition~\ref{D:maxmin}), 
\item a reference measure \( \mathfrak{m} \in \mathcal{M}_+(H) \) satisfying \( \mathfrak{m}(\static(H)) = 0 \), where \( \static(H) \) denotes the static (non-causally connected) part of \( H \).
\end{itemize}
The triple \( (H, G, \mathfrak{m}) \) defines what we call a \emph{synthetic null hypersurface}. This notion is shown to be compatible with classical null hypersurfaces in smooth Lorentzian manifolds, as well as with achronal boundaries arising as \( \partial I^+(A) \) for a compact, achronal, space-like, $C^2$-submanifold  \( A \subset M \) of codimension $k\geq 2$ (see Section~\ref{SS:CompatSyntNullHypers}).

The lack of a variational characterization for
\emph{null} geodesics  (contrary to \emph{timelike} geodesics) poses serious challenges in developing a synthetic theory of null hypersurfaces and of the \NEC\ in a non-smooth setting. Indeed,  in smooth spacetimes, null geodesics are uniquely determined
(up to affine reparametrizations) by the solution of the geodesic
equation.
In contrast, in the non-smooth setting, causal curves may admit wild
reparametrizations---such
as those involving non absolutely continuous functions like the Vitali function---that preserve causal behavior without reflecting genuine geodesicity. To overcome this, we define \emph{\( G \)-causal curves} as those for which the composition with the gauge function \( G \) is affine, a choice that restores the synthetic analog of affine parametrization and allows for a meaningful definition of null geodesics.

A similar ambiguity appears for the reference measure: in the smooth
setting, the rigged measure $\mm_L$
concentrated
on a null hypersurface is not unique and it depends on the choice of a
null  geodesic vector field $L$, whose definition requires the differential
structure of the manifold.
From a more synthetic point of view,  different rigged measures can be obtained via  multiplication by a transverse
function. The synthetic counterpart of the rigged measure $\mm_L$ is given by the measure \( \mathfrak{m} \in \mathcal{M}_+(H) \) which is part of the data defining a synthetic null hypersurface \( (H, G, \mathfrak{m}) \).
In order to show well-posedness of the theory,  we establish its covariance under multiplication of the reference measure $\mm$ by  transverse
functions (see \Cref{Sec:CovarianceGaugeMeas}).

\subsubsection*{Synthetic \NEC}

The cornerstone of our theory is the synthetic \NEC, denoted by \(\NC^{e}(N)\), where \( N > 0 \) plays the role of a synthetic upper bound on the dimension of the spacetime. This condition is formulated as the concavity of the functional
\begin{equation}\label{eq:defUNmIntro}
  \U_{N-1}(\mu|\mm)
  : =
  \exp\left(-\frac{\Ent(\mu|\mm)}{N-1} \right), 
\end{equation}
 where \( \operatorname{Ent}(\mu|\mathfrak{m}) \) denotes the Boltzmann–Shannon entropy of \( \mu \) relative to \( \mathfrak{m} \).
 The functional $\U_{N}$ is well-known
in information theory as the \emph{Shannon entropy power} (see, e.g.,~\cite{DCT-1991}; for instance, the entropy
power on $\R^N$ is the functional $\U_{N/2}$), 
and it was studied in connection to Ricci curvature bounds in~\cite{EKS}. 

The synthetic  \(\NC^{e}(N)\) condition requires that for any pair of probability measures \( \mu_0, \mu_1 \in \mathcal{P}(H) \) admitting a causal coupling, there exists a \( G \)-causal dynamical plan \( \nu \in \operatorname{OptGeo}_G(\mu_0, \mu_1) \) such that \( \U_{N-1}(\mu_t|\mathfrak{m}) \) is concave in \( t \), where \( \mu_t := (e_t)_\sharp \nu \) is the Wasserstein geodesic connecting $\mu_0$ to $\mu_1$ via the law $\nu$. 

This condition is inspired by the convexity-based characterizations of lower Ricci bounds for Riemannian manifolds obtained in~\cite{McCann, OttoVillani, CEMS, vRS} which led to the celebrated $\CD(K,N)$ spaces of Lott{--}Sturm{--}Villani~\cite{sturm:I, sturm:II, lottvillani}.

The synthetic framework developed in the present work
exhibits important stability and compatibility properties. Notably, we
show that the \(\NC^e(N)\) condition is independent of the specific
choice of gauge and reference measure within appropriate equivalence
classes, see \Cref{T:invariance-NC-locally-compact},
 thereby extending to the
synthetic setting  a property previously established in the smooth realm 
(see~\cite{CMM24a}, in particular Prop.~7.1.).
Furthermore, under additional geometric assumptions such as null
non-branching or tightness of the gauge, we prove that the sets of initial
and final points \(\initial(H) \), \( \final(H) \) of the null
generators of $H$ are negligible with respect to \( \mathfrak{m} \), a
property crucial for the applications, see \Cref{T:initialpoints} and
\Cref{P:extrimal-negligible}, respectively.

In addition to foundational results, we demonstrate that the synthetic
theory recovers the classical NEC in smooth
settings. Theorems~\ref{thm:compatibilityNEC-NCe}
and~\ref{T:compatibility-penrose} establish that if the classical
NEC~\eqref{eq:NECIntro} holds, then the associated synthetic null
hypersurface satisfies the synthetic \(\NC^e(N)\) condition.
This compatibility is shown both in the presence of global
cross-sections and for achronal boundaries of the form \( \partial I^+(A) \), even in the absence of smooth global coordinates.


In positive signature, a key property of the Lott{--}Sturm{--}Villani
$\CD(K,N)$ condition is the stability under measured Gromov{--}Hausdorff
convergence of metric measure spaces~\cite{sturm:I, sturm:II,
  lottvillani} (see also~\cite{villani:oldandnew}).
In \Cref{Sec:StabNCe} we establish  two stability theorems for the \NEC\ under a null counterpart of the  measured Gromov{--}Hausdorff  convergence for synthetic null hypersurfaces, see \Cref{thm:StabNCProb} and \Cref{thm:StabNCePointed} for the precise statements.

\subsubsection*{Optimal transport along synthetic null hypersurfaces}
In \Cref{Sec:OT-Loc}, we establish existence and uniqueness results for the optimal transport problem along a synthetic null hypersurface  $H$. These results enable us to localize the synthetic Ricci curvature lower bounds to the one-dimensional null generators of $H$  (see \Cref{T:localization}).

Such a locatization paradigm has a long history, with roots in convex
geometry~\cite{PW:1960,GM:1987,KLS:1995}.
It has been later revisited with tools of
Optimal Transport in spaces with lower bounds on the Ricci curvature,
both smooth~\cite{klartag,Ohta:2018} and non-smooth~\cite{CM1}.
More recently, it has been developed
for Lorentzian pre-length spaces satisfying synthetic \emph{timelike} Ricci curvature lower bounds expressed by the  $\TMCP(K,N)$ condition~\cite{CaMo:20} or the $\TCD(K,N)$ condition~\cite{CM24-IsopLor}, as well as \emph{variable} synthetic  timelike Ricci curvature lower bounds~\cite{Braun-McCann}.

\Cref{T:localization} extends such a powerful toolkit to the synthetic
\emph{null} setting.
These technical tools play a central role in formulating and proving synthetic analogues of Hawking’s area theorem and Penrose’s singularity theorem (see below).

\subsection{Applications}
\subsubsection*{A synthetic version of the Hawking's Area Theorem}

Hawking proved the area theorem in 1971~\cite{Haw71}. In its original formulation, it states that in a smooth spacetime satisfying the \NEC, the area of a smooth black hole horizon can never decrease.

In classical general relativity, this result played a fundamental role in the development of black hole thermodynamics. The area of a black hole horizon is interpreted as a measure of its entropy, and according to the second law of thermodynamics, entropy cannot decrease over time.

The theorem was later revisited by Chru{\'s}ciel, Delay, Galloway, and Howard~\cite{CDGH-AHP-2001}, as well as Minguzzi~\cite{MinguzziCMP15}, who extended it by relaxing the regularity assumptions on the null hypersurface while maintaining the assumption of a smooth ambient spacetime.

More recently, Ketterer~\cite{Ket24} and---using a different approach more closely aligned with the present work---the authors~\cite{CMM24a}, extended Hawking's Area Theorem to smooth weighted spacetimes satisfying the \NEC, employing tools from optimal transport theory.

\Cref{eq:HawSynt} generalizes Hawking's area theorem to the non-smooth setting of a synthetic null hypersurface $(H,G,\mm)$ contained in a topological causal space, satisfying the $\NC^e(N)$ condition.

\subsubsection*{Penrose's Singularity Theorem for continuous spacetimes}
In \Cref{S:Penrose}, we propose a way to extend Penrose's singularity
theorem to the setting of a continuous spacetime (i.e., when the
Lorentzian metric is continuous).
Recall that Penrose's singularity theorem~\cite{Penrose65} states that,
if a spacetime satisfying the \NEC\ admits a non-compact Cauchy
hypersurface and it contains a future trapped closed surface, then it is null
geodesic incomplete; i.e., there exists a null geodesic $\gamma$ (in
the sense that it solves the geodesics-{ODE}:
$\nabla_{\dot{\gamma}}\dot{\gamma}$) whose maximal domain of
definition is strictly contained in $\R$.

Besides the proof, a first challenge in generalizing such a result for
spacetimes of low regularity
(i.e., when the Lorentzian metric is of regularity lower than $C^2$)
is to make sense of the very statement.
Penrose's singularity theorem was extended by
Kunzinger{--}Steinbauer{--}Vickers~\cite{KSV-CQG} to $C^{1,1}$ spacetimes
and by Graf~\cite{Graf} to $C^{1}$ spacetimes (see
also~\cite{Aazami_2016, Minguzzi-CQG-2015, Lu-Minguzzi-Ohta_2019} for
extensions to Finsler spacetimes).
In such a regularity, the \NEC\ is understood in distributional sense
and all the other ingredients in the statement (i.e., existence of a
non-compact Cauchy hypersurface, null geodesic completeness, and existence
of a trapped surface) can be phrased, with proper care, as in the smooth
setting; the proof is then performed via a clever approximation
argument by smooth Lorentian metrics.

In the setting of  $C^0$
spacetimes 
(and, more generally, of closed cone structures),
Minguzzi~\cite[Th.~2.67]{Min} was able to extend part of the proof
of Penrose's incompleteness theorem.
Namely, he proved that there are no future trapped sets (recall the classical terminology that a compact achronal set $S$ is future trapped if $\partial I^+(S)$ is compact) 
in a spacetime admitting a stable non-compact Cauchy surface (see
also~\cite[Th.~4.9]{Ling}, where $C^0$ spacetimes are considered and
the stability assumption is dropped).
Having in mind the proof of Penrose's singularity theorem,
a natural question is then to find conditions ensuring the
existence of a trapped set in low-regularity
spacetimes.
This is where, the NEC, the geodesic
completeness, and the future trapped set play a key role.

When the spacetime is merely continuous, the \NEC\ cannot be phrased in a distributional sense (this would require the Lorentzian metric $g$ to be in the Geroch{--}Traschen class, i.e., $g\in L^{\infty}_{loc}\cap W^{1,2}_{loc}, \, g^{-1}\in L^{\infty}_{loc}$). We will thus employ the synthetic $\NC^e(N)$ condition discussed above.  Also the notions of trapped surface and of geodesic completeness present challenges, as both involve the Christoffel symbols and thus first derivatives of the Lorentzian metric. Thus, ignoring for the moment the technical difficulties involved in the proof, a basic question is how to even phrase the statement of Penrose's singularity theorem for a $C^0$-spacetime.  
We propose an
answer,  building on the synthetic optimal transport tools developed in the present work.
\begin{itemize}
\item \emph{A synthetic notion of null completeness}.
In the smooth setting, it was pointed out that the incomplete 
geodesic produced by Penrose's singularity theorem is
achronal.
Thus, negating the existence of such a geodesic means to 
require that all maximally-defined geodesics
are either complete or they cease to be achronal
(or, equivalently, they cease to be maximizing). 
In \Cref{SS:WNC} we will take this observation as definition 
of weak null completeness (see \Cref{defn:wgnc}).
  A null-geodesically complete manifold
  trivially satisfies
  such a property;  \Cref{Prop:ExNotComplete} exhibits an example of a
  weakly null complete Lorentzian manifold that is not null complete,
  justifying the terminology.
  The advantage of the weak null completeness is that it can be rephrased in terms of gauges. Indeed, \Cref{prop:wgnc} shows that for a strongly causal Lorentzian manifold $(M,g)$, the weak null completeness is equivalent to the existence of a (natural) \emph{proper} gauge. Thus, the properness of the gauge shall be seen as a synthetic counterpart of the weak null completeness.

\item  \emph{A synthetic notion of future convergence}.
  The classical  definition of future convergence
  for an achronal 2-codimensional submanifold $S$
  states that the mean
  curvature of $S$ on
  both (incoming and outgoing)
  sides of $\partial
  I^+(S)$ to be strictly negative.
  This notion is meaningful in the
    smooth setting.
  \Cref{def:future-converging} provides a synthetic formulation,
  called \emph{$(G,\mm)$-future convergence}.
  The building blocks of such a definition are the causal relation, the gauge $G$,
  and the measure $\mm$.
  The underlying idea is that the mean curvature is the first variation of area, and the area is in turn a first variation of volume.
\end{itemize}
With the above notions, we prove:

\begin{itemize}
\item \emph{A Penrose's theorem in a synthetic setting}, \Cref{T:penrose}. Roughly, it states that if an achronal set $S$ is $(G,\mm)$-future converging and $H=\partial I^+(S)$ is a synthetic null 
hypersurface satisfying the $\NC^e(N)$ condition for some $N> 2$, then  $H$ is compact; i.e., $S$ is future trapped.
\item \emph{A Penrose's singularity theorem in $C^0$-spacetimes},
\Cref{Cor:PenroseC0}.\\ Roughly, it states that if $(M,g)$ is a continuous spacetime  admitting a non-compact Cauchy surface  and containing a compact achronal set $S$ which is $(G,\mm)$-future converging and such that $H=\partial I^+(S)$ satisfies the
$\NC^e(N)$ condition, for some $N>2$, then $(M,g)$ is not weakly null
complete.
\end{itemize}

{\footnotesize
\par
\noindent
\textbf{Acknowledgments.}
The authors wish to thank E.\ Minguzzi, for useful 
comments on a preliminary version of this manuscript.

D.\,M.\;acknowledges support from the European Research Council (ERC)
under the European Union's Horizon 2020 research and innovation
programme, Grant Agreement No.~101001677 ``ISOPERIMETRY''.

A.\,M.\;acknowledges support from the European Research Council (ERC) under the European Union's Horizon 2020 research and innovation programme, Grant Agreement No.\;802689 ``CURVATURE''.

Part of this research was carried out at the Hausdorff Institute of Mathematics in Bonn, during the trimester program  ``Metric Analysis''. The authors wish to express their appreciation to the  institution for the stimulating atmosphere and the excellent working conditions. 

\par
\noindent
\textbf{License.}
\doclicenseText%
}

\section{Basics on causal spaces}\label{Subsec:BasicsLorentzianSpaces}

In this section, we briefly recall some basic notions and results from
the theory of causal spaces.
We give the definition of causal space,  following the convention of~\cite{KS}.

\begin{definition}[{Causal space  $(X,\ll,\leq)$~(\cite[Def.~2.11]{KS})}]
A \emph{causal space}  $(X,\ll,\leq)$ is a set $X$ endowed with a preorder $\leq$ and a transitive relation $\ll$ contained in $\leq$.
\end{definition}

Let us point out that this definition differs from the one present in
the paper by Kronheimer and Penrose~\cite{CausalSpace}, where the
notion of causal space was first introduced.
For the reader convenience, we compare in the table below the
two definitions. 

\begin{table}[ht]
  \newlength{\columnone}
  \setlength{\columnone}{1.65cm}
  \newlength{\columntwo}
  \setlength{\columntwo}{3.45cm}
  \newlength{\lastcolumn}
  \setlength{\lastcolumn}{\the\textwidth-7cm}
  \setlength{\lastcolumn}{\the\textwidth-\the\columnone-\the\columntwo-1.1cm}
  \centering
  \begin{tabular}{|p{\the\columnone}|p{\the\columntwo}|p{\the\lastcolumn}|}
    \hline
    &
      Causality\par\noindent (antisymmetry of $\leq$)
    &
      Push-up property (see~\Cref{def:push-up} for the push-up property)
    \\
    \hline
    K.{--}P.~\cite{CausalSpace}
    &
      Assumed
    &
      Assumed
    \\
    \hline
    K.{--}S.~\cite{KS}
    &
      Not assumed
    &
      Not assumed.
      However, in the setting Lorentzian pre-length spaces (the
      topic of~\cite{KS}), the push-up
      property is always available.
    \\
    \hline
  \end{tabular}
\end{table}

\begin{remark}
 The framework of the present paper will be the one of causal spaces, recalled above. Another possible approach could have been to take as starting point the null-distance of Sormani{--}Vega~\cite{Sormani-Vega}. It would be interesting to develop such a parallel theory, and investigate the relations with the present work. 
\end{remark}

We shall write $x<y$ when $x\leq y, x\neq y$. We say that $x$ and $y$ are \emph{timelike} (resp.\ \emph{causally}) related if $x\ll y$  (resp.\ $x\leq y$).  Let $A\subset X$ be an arbitrary subset of $X$. We define the \emph{chronological} (resp.\ \emph{causal}) future of $A$ the set
\begin{align*}
I^{+}(A)&:=\{y\in X\,:\, \exists x\in A,\, x\ll y\}, \\
J^{+}(A)&:=\{y\in X\,:\, \exists x\in A,\, x\leq y\},
\end{align*}
respectively. Analogously, we define $I^{-}(A)$ (resp.\ $J^{-}(A)$) the   \emph{chronological} (resp.\ \emph{causal}) past of $A$. In case $A=\{x\}$ is a singleton, with a slight abuse of notation, we will write $I^{\pm}(x)$ (resp.\ $J^{\pm}(x)$) instead of  $I^{\pm}(\{x\})$ (resp.\ $J^{\pm}(\{x\})$).

\begin{definition}[Topological causal space]\label{def:topological-causal-space}
  We say that $(X,\ll,\leq,\mathfrak{T})$ is a \emph{topological causal space}, provided:
  \begin{itemize}
\item  $(X,\ll,\leq)$ is a causal space;
\item  $\mathfrak{T}$ is a Polish 
  and proper topology on $X$;
  \item In the product $X\times X$, the relation $\leq$ is closed and
    the relation $\ll$ is open.
    \end{itemize}
\end{definition}

Recall that a Polish topology is \emph{proper} if there exists a proper and separable metric 
inducing the topology; recall that a metric is proper if closed balls are compact. The class of topological causal spaces forms a subclass of closed ordered spaces in the sense of Nachbin~\cite{Nachbin}; see also~\cite{Minguzzi-Review2019} for a survey of more recent developments.

Throughout the paper, $I\subset \R$ will denote an arbitrary interval. 

\begin{definition}\label{D:causalcurve}[Causal/timelike curves]
A non-constant curve $\gamma:I\to X$ is called (future-directed) \emph{causal}  (resp.\ \emph{timelike}) if $\gamma$ is continuous and if for all $t_{1}, t_{2}\in I$, with $t_{1}<t_{2}$, it holds $\gamma_{t_{1}}\leq \gamma_{t_{2}}$ (resp.\ $\gamma_{t_{1}}\ll \gamma_{t_{2}}$). We say that $\gamma$ is a \emph{null} curve if, in addition to being causal, no two points on $\gamma(I)$ are related with respect to $\ll$. 
\end{definition}

We recall the following  

\begin{definition}
    Let $(X,\ll,\leq,\mathfrak{T})$ be a topological causal space. 
    We say that a set $A\subset X$ is weakly convex, if for all $x,y\in A$, such that $x\leq y$, there 
    exists a causal curve $\gamma$ contained in $A$ 
    connecting $x$ to $y$.
\end{definition}

\begin{remark}
  We point out that \Cref{D:causalcurve} is more general than the
  one given in the seminal paper~\cite{KS} by Kunziger and S\"amann,
  for we do not require that a causal curve is (locally) Lipschitz.
  In that paper, the authors proved that if $(M,g)$ is a  strongly
  causal Lorentzian manifold (i.e., any point admits arbitrarily small neighborhoods which are weakly convex)  with  $g$ continuous, then a causal Lipschitz curve is casual in the classical
  sense (i.e., its tangent vector is causal almost everywhere).

  Next Lemma recovers Kunziger's and S\"amann's result, by proving, in
  the same setting, that any causal curve admits a Lipschitz
  reparametrization.
  \end{remark}

\begin{lemma}\label{lem:lipschitz-reparametrization}
  Let $(M,g)$ be a strongly causal Lorentzian manifold with $g$ continuous.
  Let $\gamma:I\to M$ be a causal curve.
  Then there exists a reparametrization of $\gamma$ which is locally
  Lipschitz-continuous.
\end{lemma}
\begin{proof}
  We can assume without loss of generality that $\gamma$ is injective
  and that $0\in I$.
  Since the statement is local, it suffices to show that $\gamma$
  admits a Lipschitz reparametrization in a neighborhood of $0$.
  Let $U$ a neighborhood of $\gamma_0$.
  We may assume that $U$ is coordinated, i.e., that $U\subset \R^n$.
  Up to further restricting $U$, we can also assume that $g|_{U}$ is
  close to the standard Minkowski product, in the sense that, for all
  $x\in U$, the causal cone in $T_{x}U=\R^n$ of $g_x$ is contained in the
  causal cone of $-2\de x_n^2 +\frac{1}{2}\sum_{i=1}^{n-1} \de x_i^2$,
  i.e., if $X$ is a $g$-causal vector, then
  \begin{equation}
    \sum_{i=1}^{n-1} |X_i|^2
    \leq
    4|X_n|^2.
  \end{equation}
  By strong causality, we can assume $U$ to be
  weakly convex.
  Let $[a,b]$, such that $\gamma([a,b])\subset U$,
  and let $[\alpha,\beta]=\gamma_n([a,b])$.
  
  We claim that, for all $a\leq t_1\leq t_2\leq b$, it holds that
  \begin{equation}
    \label{E:causal-to-lipschitz}
    \sum_{i=1}^{n-1}
    |\gamma^{i}_{t_1}-\gamma^{i}_{t_2}|^2
    \leq
    4
    |\gamma^{n}_{t_2}-\gamma^{n}_{t_2}|^2
    .
  \end{equation}
  Indeed, fix $t_1\leq t_2$
  and let $\delta:[0,1]\to U$ be a causal (in the classical, i.e., differential sense) curve connecting
  $\gamma_{t_1}$ to $\gamma_{t_2}$.
  A direct computation gives
  \begin{align*}
    \sqrt{
    \sum_{i=1}^{n-1}
    |\gamma^{i}_{t_1}-\gamma^{i}_{t_2}|^2
    }
    &
      =
    \sqrt{
    \sum_{i=1}^{n-1}
    |\delta^{i}_0-\delta^{i}_1|^2
      }
      =
    \sqrt{
    \sum_{i=1}^{n-1}
    \left|\int_0^1 \dot\delta^i\right|^2
      }
      \leq
      \int_0^1
    \sqrt{
    \sum_{i=1}^{n-1}
      \left|
      \dot\delta^i
      \right|^2
      }
    \\
    &
      \leq
      2
      \int_0^1
      |
      \dot\delta^n
      |
      =
      2(\delta^n_1
      -
      \delta^n_0)
      =
      2|\gamma^n_t-\gamma^n_r|
      .
  \end{align*}

  It follows that, for all $\tau\in[\alpha,\beta]$, there exists a
  unique $t\in[a,b]$, such that $\gamma_n(t)=\tau$ (if it were not
  unique, then~\eqref{E:causal-to-lipschitz} would contradict the injectivity of
  $\gamma$).
  We can therefore define the reparametrization $\eta_\tau=\gamma_t$,
  so that $\tau=\gamma^n_t$.
  This reparametrization is Lipschitz-continuous, because
  $|\eta^n_{\tau_1}-\eta^n_{\tau_2}|=|\tau_1-\tau_2|$, and
  \begin{equation}
    |\eta^i_{\tau_1}-\eta^i_{\tau_2}|
    =
    |\gamma^{i}_{t_1}-\gamma^{i}_{t_2}|
    \leq
    2|\gamma^{n}_{t_1}-\gamma^{n}_{t_2}|
    =
    2|\tau_1-\tau_2|
    .
    \qedhere
  \end{equation}
\end{proof}

\section{Null synthetic setting}\label{Ss:mesures-on-H}

\subsection{Synthetic definitions}

In the general setting of topological causal spaces, the role of smooth 
null-hypersurfaces will be played by 
closed achronal sets where the causal relation does not trivialise in a sense that will be clarified later.

Among geometrically meaningful closed achronal sets, one can find, e.g.,
achronal boundaries and Cauchy horizons.
We recall below the definition of achronal boundary 
that is classical in the smooth literature, see for instance~\cite{Penrose-DiffTopGR, HawEll}.

\begin{definition}[Achronal boundary]
Let $(X,\ll,\leq,\mathfrak{T})$ be a topological causal space.
A set $H\subset X$ is an \emph{achronal boundary}  if 
there exists a set $S \subset X$ 
such that $H=\partial I^{+}(S)$.
\end{definition}

The next lemma shows that achronal boundaries are actually achronal, so the terminology is consistent.

\begin{lemma}\label{P:achronal-boundary}
  Let $(X,\ll,\leq,\mathfrak{T})$ be a topological causal space.
  Then for any subset $S\subset X$, $\partial I^{+}(S)$ is achronal.
\end{lemma}
\begin{proof}
  Assume on the contrary that there exist $x,y\in \partial I^{+}(S)$,
  such that  $x\in I^-(y)$.
  Since $I^-(y)$ is open, there exists an open neighborhood   $U\subset I^{-}(y)$ of $x$.
  Since $x\in \partial I^{+}(S)$, then there exists
  $w\in U\cap I^{+}(S)$.
  It follows that $w\ll y$ and thus, by transitivity of $\ll$, we get that  $y\in I^{+}(S)$, a contradiction.
\end{proof}

It is worth underlining that 
every maximal achronal set is an achronal boundary, see for instance~\cite[Theorem 2.97]{Minguzzi-Review2019}. Since 
achronal sets are subsets of their own achronal boundary, they 
provides the natural framework 
for a synthetic treatment of null geometry that  will take into account both the local and the global picture.

We  next define 
 the maximal and minimal elements of a closed achronal set $H$ with respect to
 the pre-order $\leq$, denoted below as $\finalWithH$ and $\initialWithH$ respectively. Such objects are well-known to be useful in the study of achronal boundaries  within  smooth spacetimes, see e.g.\;\cite[Section 6.3]{HawEll}, and in the analysis of the  $L^1$ optimal transport problem within metric spaces.

\begin{definition}\label{D:maxmin}
Given a closed achronal set
$H$, the set of \emph{final} elements of $H$ is given by 
\[
  \finalWithH : =
  \{ m \in H \colon \ \nexists x \in H, \ m \leq x, \ x \neq m \}, 
\]
and the set of \emph{initial} elements of $H$ is given by 
\[
\initialWithH
: =
\{ m \in H \colon \ \nexists x \in H, \ x \leq m, \ x \neq m \}. 
\]
Then the achronal set without end-points is defined by 
\[
\noending{H} : = H \setminus  \left( \initialWithH \cup 
\finalWithH \right).
\]
Finally, the set of \emph{static} elements  of $H$ is given by
\[
\staticWithH : = \initialWithH \cap \finalWithH
.
\]
\end{definition}

In the sequel, whenever no confusion arises, we shall write
$\initial=\initialWithH$, and $\final$ and
$\static$ with the analogous meaning.

We briefly discuss 
the measurability of $\initial$ and $\final$.

\begin{lemma}[]\label{L:measurable}
In a topological causal space $X$,
the sets $\initialWithH, \finalWithH, \staticWithH$,
and $H^\circ$ are Borel measurable.
\end{lemma}

\begin{proof}
Since the topology on $X$ 
is induced by a proper metric
and $J^{+}$ is closed, 
denoting by $\Delta = \{(x,x) \colon x \in X \}$, 
the set $J^{+} \setminus \Delta \cap H \times H$ is $\sigma$-compact, 
i.e., is countable union of compact sets.  
As projections of $\sigma$-compact set is $\sigma$-compact, 
writing
\[
\initial = H \setminus P_{2}\left( (J^{+} \setminus \Delta) \cap H \times H \right),
\]
we obtain  that $\initial$ is Borel measurable and the same conclusion is valid for $\final$ proving the  claims.
\end{proof}

\begin{remark}\label{R:intervalH}
If $\gamma : I \to H$ is any causal curve then 
$\gamma_I \cap 
\initial$ contains at most a single element:   
suppose by contradiction there are two of them, then one has to be in the causal future of the other giving a contradiction. The same is true for $\gamma_I \cap 
\final$.

Hence $\gamma^{-1}(\initial)$ and 
$\gamma^{-1}(\final)$
are two closed intervals in $I$ 
on each of which $\gamma$ is constant. 
Therefore, without loss of generality we can assume that if $\gamma$ 
is a non-constant causal curve and if $I = [0,1]$, then $\gamma_{(0,1)} \subset \noending{H}$.
\end{remark}

We introduce the notion of gauge
function.
\begin{definition}\label{defn:gauge}
  Let $H$ be a closed  achronal set and let $G:\noending{H}\to\R$ be a Borel map. 
  We say that $G$ is a \emph{gauge} function for $H$, if it satisfies the following properties:
  \begin{enumerate}
  \item for any injective causal curve $\gamma$ in 
  $\noending{H}$,
    $G\circ \gamma$ is strictly increasing and continuous;
  \item for any causal curve $\gamma:[0,1]\to{H}$,
    such that $\gamma_{(0,1)}\subset \noending{H}$, then
    $\sup_{t\in (0,1)} G (\gamma_t)$ and 
    $\inf_{t\in (0,1)} G (\gamma_t)$ are both finite.
  \end{enumerate}
Finally a gauge function G is called \emph{proper}
if the pre-image of every compact set is precompact. 
\end{definition}

    Assumption (1) implies that 
    for any causal (possibly non injective) curve $\gamma$ in $\noending{H}$,
    $G\circ \gamma$ is non-decreasing
    and continuous.
    This fact follows by reparametrizing $\gamma$. 
    Assumption (2) above is needed in order to avoid possible
  blow-ups of the gauge. 
The concept of properness of a gauge will be used in \Cref{S:Penrose} to suitably reformulate the null completeness property.

In the smooth setting, a causal curve $\gamma:[0,1]\to H$  solves the geodesics equation if and only
if $G\circ\gamma$ is affine 
for a natural gauge $G$ (see \Cref{R:compatibility} below).
We take this last observation as the definition of a
causal null geodesic in the null synthetic setting.

\begin{definition}[$G$-Causal curves]\label{D:Gcausalcurve}
Let $H$ be a closed achronal set and let 
$G:\noending{H}\to \R$ be a gauge.
We say that a curve $\gamma:[0,1]\to H$ is \emph{$G$-causal}, if and only
if it is causal,  $\gamma_{(0,1)}\subset\noending{H}$,
and $G\circ \gamma|_{(0,1)}$ is an affine function.

The set of $G$-causal curves 
will be denoted by $\Caus_G\subset C([0,1]; H)$.  
\end{definition}

Sometimes, with a slight abuse of notation, we will consider
$G$-causal curves whose domain is not the interval $[0,1]$.

We are now in position to give the definition of synthetic null hypersurface.
We denote by $\M^{+}(H)$
the set of non-negative Radon measures on $H$.

\begin{definition}[Synthetic Null Hypersurface]\label{def:synthetic-null-hypersurface}
Let $(X,\ll,\leq,\mathfrak{T})$ be a topological causal space.
We say that the triple $(H,G, \mm)$ is a \emph{synthetic null hypersurface},
if
\begin{itemize}
    \item 
$H$ is
a closed achronal set, 
\item 
$G:\noending{H}\to\R$ is a gauge for $H$,
\item $\mm\in\M^{+}(H)$ does not give positive mass to the static elements of $H$, i.e., $\mm(\static)=0$.
\end{itemize}
\end{definition}

\begin{remark}
The request of $\mm(\static)=0$ is necessary to avoid examples of closed achronal sets which do not have the geometry of a null hypersurface. For instance,  if $H$ is an acasual set (e.g., a spacelike hypersurface), then the set of static points will coincide with $H$ itself and therefore 
the only admissible measure 
to consider over $H$ in order to form a synthetic null hypersurface is $\mm\equiv 0$.
\end{remark}

\subsection{Compatibility with the smooth setting}\label{SS:CompatSyntNullHypers}

We next discuss the compatibility of the above synthetic constructions with the smooth
theory, by associating to a classical null hypersurface a synthetic
one.
This is the content of the following remarks.

\begin{remark}[Compatibility with the smooth setting: the global case]\label{R:compatibility}
We proceed by first discussing the global case, i.e., when a global smooth cross-section is available.
Let $H$ be a causal and achronal, smooth null hypersurface, admitting
 $S\subset H$ a global smooth cross-section for $H$.
 Using the cross-section $S$, we can build a global null geodesic vector field
 $L$ over $H$ and the corresponding flow map
$\Psi_{L} : S \times \R \to H$.
A natural gauge function $G_{L,S}:H\to \R$
is then defined as the function such that $z=\gflow_L(p,G_{L,S}(z))$, for some
(unique) $p\in S$.  
To check that $G_{L,S}$ is indeed a gauge function, observe that
if $\gamma$ is a causal curve in $H$, then it lies in the generator of $H$ passing through $\gamma_0$ and, 
in particular, it must be of the form
$\gamma_t=\gflow(\gamma_0, s(t))$, 
for some non-decreasing function $s$.
Therefore $G_{L,S}(\gamma_t)=G_{L,S}(\gamma_0)+s(t)$, which is
non-decreasing.
If $\gamma$ is also injective, then $s$ must be strictly increasing.
Moreover, we point out that \emph{a causal curve is $G_{L,S}$-causal if and only 
if it is a null geodesic in $H$}.
Finally, as reference measure over $H$, we choose the rigged measure $\vol_L$ (see~\cite{CMM24a}).
A synthetic null hypersurface associated to $H$ is $(H, G_{L,S}, \vol_L)$.
\end{remark}

\begin{remark}[Compatibility with the smooth setting: the local case]\label{R:compatibility-local}
In several situations of interest (e.g., in the presence of recurring geodesics like the case of compact Cauchy horizons~\cite{MoncIsenCMP1983, BustReiGRG21, GuMinCMP2022})   a global cross-section for $H$ does not exist.
In these cases, it is possible to proceed via a local argument as follows. 

Let $(M,g)$ be a Lorentzian manifold and $U\subset M$ is any open set. For  $x,y\in U$, we set by definition $x \ll_U y$,  provided that there exists a $C^1$ curve 
$\gamma:[0,1]\to U$, connecting $x$ to $y$, such that $\dot \gamma$ 
is a future-directed timelike vector.
Notice that $\ll_U$ is (in general) strictly contained  in $\ll\cap\; (U\times U)$;
in a similar way we define $\leq_U$ and
denote by $\mathcal{T}_U$ the topology on $U$ induced by $\mathcal{T}$. 
We will write $X_U$ as a shorthand notation for the 
topological causal space $(U,\leq_U,\ll_U,\mathcal{T}_U)$; 
in particular $X_U$ is the topological causal space 
induced by the Lorentzian manifold $(U,g|_U)$.

Now if $H$ is a smooth null hypersurface, 
then for any $x\in H$, we can take a coordinated neighborhood $U$ for $x$.
Up to restricing $U$, we have that $H \cap U$ is causal, achronal and
closed and that
a cross-section $S$ is given by the intersection of $H$ and the
$0$-level set of the time coordinate.
We can thus apply the construction described in \Cref{R:compatibility}.
\end{remark}

\begin{remark}[Compatibility with the smooth setting: achronal boundaries]\label{R:compatibility-penrose}
  Let $(M,g)$ to be a causally simple (or, equivalently, strongly causal and
  causally closed),
    $C^2$ Lorentzian manifold.
  Let $A\subset M$ be a compact, 
 achronal,     space-like,
  $C^2$ sub-manifold of codimension
  $k\geq 2$. Then 
  \[H :=\partial I^{+}(A)\]
  is a locally-Lipschitz topological hypersurface 
  (see, e.g.,~\cite[Lemma 3.17]{Penrose-DiffTopGR} and~\cite[Proposition 6.3.1]{HawEll}).
  If $A$ is a singleton, $H$ is a future cone; if the codimension of $A$ is $2$, then $H$ is the central object in the proof of Penrose's singularity theorem.
  To include achronal boundaries of this type in our setting, 
  we proceed as follows.

 Under the above assumptions, it is a well-known fact that
 \begin{equation}
   H
   =
   \partial I^{+}(A)
   =
   J^{+}(A)\setminus I^{+}(A)   .
 \end{equation}
 Indeed, the inclusion ``$\subset$'' is a consequence of the the causal closedness of $(M,g)$ and the compactness of $A$. Conversely, the inclusion ``$\supset$'' follows by the $C^2$-regularity of $g$, for a proof see~\cite[Prop.\;2.13]{Min}.

 Firstly, we claim that $\initial\subset A$.
 Indeed, if $x\in \initial\subset J^{+}(A)$, there exists
 $y\in A\cap J^{-}(x)$; it holds that $y\notin I^{+}(A)$, otherwise the push-up
 property would imply that $x\in I^{+}(A)$, a contradiction.
 Hence $y\in H$, therefore $y=x$, otherwise $x$ would not be an
 initial point (i.e., minimal for $\leq$).

 We next claim that $J^{+}(A)=J^{+}(\initial)$.
 Indeed, fix $x\in J^{+}(A)$.
 By compactness, there exists $y$, a minimal element (w.r.t.\ $\leq$)
 of $A\cap J^{-}(y)\subset J^{+}(A)$; $y\in H$, otherwise,
 $y\in I^{+}(A)$, a contradiction with its minimality.
 Let $z\in H\cap J^{-}(y)$; minimality of $y$ implies that $z=y$,
 therefore $y\in\initial$, thus $x\in J^{+}(\initial)$.

 As a consequence, $I^{+}(A)=I^{+}(\initial)$ and
 \begin{equation}
   \label{eq:achronal-boundary-is-horismos}
   H
   =
    J^{+}(A)
    \setminus
    I^{+}(A)
    =
    J^{+}(\initial)
    \setminus
    I^{+}(\initial)    
    .
  \end{equation}

 Let $M^{-}:H\setminus\final\to A$ be the map defined by
  \begin{align*}
    M^{-}(x)
    :=
    \inf J^{-}(x)\cap A
    ,
  \end{align*}
  where the infimum is taken w.r.t.\ the causal order $\leq$.
  Since $x$ is not a final point, then $J^{-}(x)\cap A$ is totally
  ordered (i.e., we can always compare elements), hence there is at
  most one minimal point.
  Compactness of $A$ guarantees that the infimum is attained.
  Checking the measurability of $M^{-}$ is standard.
  Notice that $M^{-}$ takes values in $\initial$.

  We claim that the set $A\setminus\initial$ is compact.
  Indeed, let $x_n$ be a sequence in $A\setminus\initial$; by
  the compactness of $A$, up to a subsequence, $x_n\to x\in A$.
  Let $y_n\in A$, be such that $y_n< x_n$; up to a subsequence
  $y_n\to y\in A$.
  Causal closedness yields $y\leq x$.
  If $y\neq x$, then $y< x$, proving that $x\in A\setminus\initial$.
  If on the contrary $x=y$, then, since $x_{n}$ and $y_{n}$ converge to
  the same point $x$ and are causally related, then $T_{x}A$ is not
  space-like, a contradiction. Thus, $A\setminus\initial$ is compact.
  
  It follows that $\initial$ is open in $A$, thus a
  sub-manifold. Notice also that $\initial$ is acausal.

  Consider $\normal \initial$, the normal bundle to
  $\initial$ and endow it with the restriction of the ambient
  Lorentzian metric $g$; notice that each fiber of
  $\normal \initial$ is a $k$-dimensional Minkowski spacetime.
  Let $W\in\vfield(TM)$ be a vector field inducing the time
  orientation;  since $A\subset M$ is a smooth spacelike hypersurface, we can assume $W|_{A}\in\vfield(\normal A)$.
 Let
  \begin{equation}\label{eq:defhatH}
    \hat H:=\{
    v\in\normal \initial
    \colon
    g(v,v)=0
    \text{ and }
    g(v,W)< 0
    \}
    .
  \end{equation}
  We claim that
  $H\setminus \initial\subset \exp(\hat H)$.
  Indeed, Equation~\eqref{eq:achronal-boundary-is-horismos} implies
  that if $x\in H\setminus\initial$, there exists a null curve
  $\gamma$ connecting $x$ with $\initial$; by
  \Cref{P:non-branching-curves}, we can assume that $\gamma$ is a
  geodesic.
  It is immediate to see that $\dot\gamma_0$ a future-directed null
  vector and that it is orthogonal to $\initial$, thus
  $\dot\gamma_0\in\normal\initial$.

 We now define a gauge for $H$: 
  for $x\in \noending{H}$, $x=\exp(v)$ for some (unique)
  $v\in \hat H\cap \normal_{M^-(x)}\initial$; define
  \begin{equation}\label{eq:defGhatH}
    G(x):=- g(v,W).
  \end{equation}
  Notice that $G$ can be continuously extended to $0$ on the set
  $\initial$.

Next, we endow the manifold $\hat H\subset \normal \initial$ with a  tangent vector
  field $\hat L\in \vfield(T\hat H)$.
  In order to do so, notice that at each point $v\in
  \normal_{x}\initial$, the tangent space splits  as the direct sum (in the sense of linear algebra)
  \[
    T_{v}\normal\initial
    \cong
    \normal_{x}\initial\oplus T_{x}\initial
    =
    T_{x}M
    .
  \]
  We denote by $\Sigma: T\normal\initial\to TM|_{\initial}$,
  $\Sigma_{v}: T_{v}\normal\initial \to T_{x}M$, such a morphism of fiber
  bundles.
  Let $\hat g:= \Sigma^*g$ be the pull back on $T\normal\initial$ of the Lorentzian metric $g$; 
  since on each fiber $\Sigma_{v}$ is an isomorphism,  we can also
  define $\hat W\in \vfield(T\normal \initial)$ as
  $\hat{W}_{v}:=(\Sigma_{v})^{-1}(W)$, giving a time-orientation to
  $\normal \initial$.
Observing that  $\hat H$ is a smooth null hypersurface in the smooth Lorentzian manifold $(\normal \initial,\hat g)$, we can set
  $\hat L$ as the null vector field
  $\hat L\in\vfield(\normal \hat H)$, such that
  $\hat g(\hat L, \hat W)=-1$.

  Denote by $U\subset \hat H$ the open set where the exponential map is
  defined and its differential has maximal rank.
  For any $v\in U$, there exists a neighborhood $U_{v}\subset U$, such
  that $\exp|_{U_{v}}$ is a diffeomorphism on its image.
  In particular  
  \begin{equation}\label{eq:defHv}
  H_{v}:=\exp(U_{v})
  \end{equation}
  is a smooth null hypersurface.
  We endow $H_{v}$ with the null geodesic vector field
  $L^v:=\de\exp|_{U_{v}}[\hat L]$ and consider the rigged volume
  $\vol_{L^v}$.
  By pulling back the rigged measure $\vol_{L^{v}}$ associated to $L^v$, we obtain the measure
  \begin{equation}\label{eq:defhatmmloc}
  \hat \mm_{v}:=(\exp|_{U_{v}})^{-1}_{\sharp}\vol_{L^{v}}\in\M^{+}(U_{v}).
  \end{equation}
  Given $v_1,v_2 \in U$, it is immediate to see that the
  measures $\hat \mm_{v_1}$ and $\hat \mm_{v_2}$ coincide on
  $U_{v_1}\cap U_{v_2}$.
  We can therefore glue these local measures to define a global measure
  $\hat \mm\in\M^{+}(U)$.
  Let
  \[
    B:=
    \{v\in U
    \colon
    \exists \lambda>1,
    \;
    \exp(\lambda v)\notin I^{+}(A)
    \}
    .
  \]
  On one hand, it clearly holds that $\exp(B)\subset \noending{H}$.
  On the other hand, if $x\in \noending{H}$, then $x$ is not a focal point
  for the geodesic connecting $\initial$ to $x$
  (see~\cite[Th.~6.16~(a)]{Minguzzi-Review2019}); thus it is not a
  conjugate point, meaning that the differential of the exponential map
  has maximal rank at $\exp^{-1}(x)$, and therefore  $\exp^{-1}(x)\in B$.
  Thus
  \[
  \exp(B)= \noending{H},
  \]
  and we can endow $\noending{H}$ with $\mm:=\exp_{\sharp}(\hat\mm\llcorner_{B})$.
  A synthetic null hypersurface associated to $H$ is then
  $(H, G, \mm)$.
  \end{remark}

\subsection{Optimal transport and the synthetic \NEC\ \texorpdfstring{$\NC^{e}(N)$}{NCe{(N)}}}

A concept analogous to the one of optimal dynamical plan can be introduced in the
synthetic setting by means of the gauge function and the $G$-causal curves 
as follows. 
Firstly, given any two probability measures $\mu_0,\mu_1$, we introduce 
\[
\OptCaus^H(\mu_0,\mu_1) : = \{ \nu \in \mathcal{P}(C([0,1]; H)) \colon 
(e_i)_\sharp \nu = \mu_i, i=1,2, \ \nu(\Caus) = 1 \},
\]
where $\Caus \subset C([0,1]; H)$ is the closed subset of  causal curves and 
$e_{i} : C([0,1]; H) \to H$ are the evalutation map at $t = i$, with $i=0,1$.
We also recall the definition 
of the set of causal couplings
\[
\Pi_{\leq}(\mu_0,\mu_1)
= \{\pi \in \mathcal{P}(X\times X) \colon 
(e_{i})_\sharp \pi 
= \mu_{i}, 
\ 
i = 0,1
\}.
\]
In particular, 
$\OptCaus^H(\mu_0,\mu_1) \neq \emptyset$ only if $\Pi_\leq(\mu_0,\mu_1)\neq \emptyset$.
Notice that since $H$ is closed, as topological space $H$ is Polish and therefore $C([0,1]; H)$ is Polish with the topology of the uniform convergence.
In particular, every element of 
$\mathcal{P}(C([0,1]; H))$ is inner regular with respect to compact sets of 
$C([0,1]; H)$.

\begin{definition}
  Let $(H,G,\mm)$ be a synthetic null hypersurface and let $\mu_0,\mu_1\in \Prob(H)$ be
  two probability measures.
  Any $\nu\in\OptCaus^H(\mu_0,\mu_1)$ 
  is said to be \emph{$G$-causal} if
  and only if $\nu(\Caus_{G})=1$. The set of such $G$-causal plans $\nu$ is denoted by  $\OptGeo^G(\mu_0,\mu_1)$.
\end{definition}

Given a synthetic null hypersurface $(H,G,\mm)$,
for $\mu \in \Prob({H})$, the Boltz\-mann{--}Shannon  
entropy $\Ent(\mu|\mm)$ of $\mu$ with respect to $\mm$ is defined by
\[
 \Ent(\mu|\mm) : = \int_{X} \rho \log(\rho) \, \mm
 \in [-\infty,+\infty),
\]
if $\mu = \rho \, \mm$ and $(\rho\log(\rho))_{+}$ is $\mm$-integrable;  
otherwise we set $\Ent(\mu|\mm) = +\infty$. 
We denote by
\[
\Dom(\Ent(\cdot|\mm)):=\{\mu\in \Prob(X): \Ent(\mu|\mm)\neq+\infty\}.
\]

Let us
consider the following dimensional variant of the Boltzmann{--}Shannon entropy: given $N>0$, 
for $\mu \in \Dom(\Ent(\cdot|\mm))$, define
\begin{equation}\label{eq:defSN}
\U_{N}(\mu|\mm) : = \exp\left(-\frac{\Ent(\mu|\mm)}{N} \right), 
\end{equation}
and set $\U_{N}(\mu|\mm) = 0$ otherwise. 
Such a functional $\U_{N}$ is well-known
in information theory as the Shannon entropy power (see, e.g.,~\cite{DCT-1991}; for instance, the entropy
power on $\R^N$ is the functional $\U_{N/2}$), 
and it was studied in connection to Ricci curvature bounds in~\cite{EKS}.

\begin{definition}[{$\NC^{e}(N)$ for synthetic null hypersurfaces}]\label{D:syntheticnull}
Let $N>0$. Let $(X,\ll, \leq, \mathfrak{T})$ be a topological causal
space.

The synthetic null hypersurface $(H,G,\mm)$ satisfies the \NEC\ $\NC^{e}(N)$ if and only if the following holds:
for all measures $\mu_0,\mu_1\in \Prob(H)$
such that 
$\Pi_{\leq}(\mu_0,\mu_1)\neq \emptyset$, 
there exists $\nu\in\OptGeo^G(\mu_0,\mu_1)$ 
such that 
\begin{equation}\label{E:concavity}
  \U_{N-1}(\mu_{t}|\mm)
  \geq
  (1-t)
  \U_{N-1}(\mu_0|\mm)
  +
  t
  \U_{N-1}(\mu_1|\mm)
  , \quad \text{ for all }t\in [0,1],
\end{equation}
where $\mu_{t}:=(e_{t})_{\sharp}\nu$.
\end{definition}

We now collect two geometric consequences of \Cref{D:syntheticnull}. 
The first one is contained in the following

\begin{remark}[$\NC^e$ implies weak convexity]\label{R:absolutelycontinuous}
Note that if both 
$\mu_{0},\mu_{1}$ are not contained in $\Dom(\Ent(\cdot|\mm))$, then~\eqref{E:concavity} is trivially satisfied. 
On the other hand, if either $\mu_{0}$ or $\mu_{1}$ belongs to
$\Dom(\Ent(\cdot|\mm))$, then~\eqref{E:concavity} implies that,
for all $t \in (0,1)$,
$\mu_{t}$
is absolutely continuous with respect to $\mm$.

Notice also that if $x\leq y$ in $H$, then taking $\mu_0=\delta_{x}$, 
$\mu_1=\delta_{y}$, $\Pi_\leq(\mu_0,\mu_1)=\{\delta_{(x,y)}\}$.
\Cref{D:syntheticnull} implies that there exists $\nu\in\OptGeo^{G}(\mu_0,\mu_1)$, 
producing a $G$-causal curve from $x$ to $y$.
In other words, $H$ is weakly convex.
\end{remark}

We now prove that the sets of initial and 
final points can be neglected provided 
the set of $G$-causal curves satisfy the following compactness property.

\begin{definition}[Tightness of the gauge]
    A gauge function $G$ for $H$ is said to be \emph{\ascar} if for every $x\in H$, there exists a compact
  neighborhood $U$,
such that
\begin{equation}
  \label{eq:ascoli-arzela}
  \{\gamma \in \Caus_{G} \colon \gamma_0 , \gamma_1 \in U \}
  \qquad
  \text{ is compact in } C([0,1]; H).
\end{equation}
\end{definition}

The tightness assumption is independent from the non-branching condition,
that is usually assumed in the literature  when that initial and final points have measure zero. 

\begin{proposition}\label{P:extrimal-negligible}
Let $(H,G,\mm)$ be a 
synthetic null hypersurface 
satisfying the \NEC\ 
$\NC^{e}(N)$, for some $N> 0$.
Assume $G$ to be \ascar.
Then 
\[
\mm(\initial)=\mm(\final)
  =
  0.
\]
\end{proposition}
\begin{proof}
It will suffice to prove that $\mm(\initial\setminus\final)
= 0$:
by assumption $\mm(\static)=0$ (see
  \Cref{def:synthetic-null-hypersurface})
  and therefore it will follow that
  $\mm(\initial)=0$.
  The remaining part is symmetric.

Suppose by contradiction
\begin{equation}
\mm(\initial\setminus\final) > 0.
\end{equation}
Define the function $A:\noending{H}\times \noending{H}\cap J\to \R$ as 
$A(x,y)=G(y)-G(x)\geq 0$.
Fix $\epsilon>0$ and consider the set 
\begin{equation}
    W_\epsilon
    :=
    \{
    (x,z,w)\in \initial\setminus\final
    \times \noending{H}
    \times \noending{H}
    \colon
    (x,z),(z,w)\in J,
    A(z,w)>\epsilon
    \}
    .
\end{equation}
The set $W_\epsilon$ is Borel.
It holds that 
\begin{equation}
    \initial\setminus\final
    =
    \bigcup_{\epsilon>0}
    P_1(W_\epsilon).
\end{equation}
The (contradiction) assumption gives  $\mm(P_1(W_\epsilon))>0$, for some $\epsilon>0$.
By inner regularity of $\mm$, there exists  a compact set  $K\subset P_1(W_\epsilon)$ with $\mm(K)>0$.
We define the following analytic set: 
\[
\Lambda: = P_{1,3}(W_\epsilon).
\]
By Von Neumann's selection Theorem (see for instance~\cite[Th.~5.5.2]{Srivastava}),
we  deduce that $\Lambda$ contains the graph of a measurable map $T$. 
Define the absolutely continuous measure 
$\mu_{0} : =  \mm\llcorner_{K}/ \mm(K)  \in \mathcal{P}(H)$ and 
$\mu_{1} : = T_{\sharp} \mu_{0}$ 
so that 
$({\rm Id},T)_\sharp \mu_0 \in \Pi_\leq(\mu_0,\mu_1)$.

Up to neglecting a set of small $\mu_0$-measure, we can also assume that 
$T$ is continuous and therefore  also $\mu_1$ 
can be assumed with compact support: 
hence we have deduced the existence of a compact set $K \subset H$ such that 
$\mu_0 (K) = \mu_1(K) = 1$.

\Cref{D:syntheticnull} gives then a measure
$\nu\in\OptGeo^G(\mu_0,\mu_1)$ such that
the function $t\mapsto \U_{N}((e_{t})_{\sharp}\nu)$ is concave.
By \Cref{R:absolutelycontinuous} $(e_{t})_{\sharp}\nu \ll \mm$. 

Fix any $x \in K$ and an open neighborhood $U$ of $x$ such that $\mu_0(U)>0$;
then we deduce the existence of $\bar t>0$ such that 
\[
\nu (\{\gamma \in \Caus_{G} \colon \gamma_0 \in U, \ \gamma_{t} \in U \} ) > 0.
\]
Indeed by continuity of causal curves, 
\[
\nu \otimes \mathcal{L}^1
\left(\left\{ (\gamma,t) \in \Caus_{G} \times (0,1) \colon \gamma_0 \in U, \gamma_{t} \in U \right\} \right) >0,
\]
then Fubini's Theorem implies the claim. 
Then, since $\nu$ is inner regular, 
there exists a compact set $\bar C$ such that
\[
\bar C \subset  \{ \gamma \in \Caus_{G} \colon
\gamma_0 \in U, \gamma_{\bar t} \in U\}, \quad
\text{with } \quad \nu(\bar C ) >0.
\]
Denote by $\bar \nu$ the renormalized restriction of $\nu$ to $\bar C$ and 
define the probability measures
\[
\bar \mu_0 : = (e_0)_\sharp \bar \nu,
\qquad \bar \mu_1 : = (e_{\bar t})_\sharp \bar \nu.
\]
By construction, they are causally related 
and both concentrated on a compact set inside $U$. 
Since $U$ was arbitrary, we can take as $U$ the one additional
verifying the tightness hypothesis of the gauge $G$. 
By \Cref{D:syntheticnull}, we deduce the existence of a measure $\eta \in\OptGeo^G(\bar \mu_0,\bar \mu_1)$ such that
the function $t\mapsto \U_{N}((e_{t})_{\sharp}\bar \nu)$ is concave and   $\eta(C) = 1$, 
for some compact set $C$.

For each $t\in [0,1]$, denote now by $C_{t} : = e_{t}(C)$. Notice that each 
$C_{t}$ is a compact set, $C_{t}$ converges to $C_0$ in Hausdroff topology 
as $t \to 0$ and  
$C_0 \subset \initial\setminus\final$.
Finally, by the entropy inequality~\eqref{E:concavity} and a double Jensen's inequality we have
\[
\mm(C_{t})^\frac{1}{N-1} 
\geq \U_{N-1}(\bar \mu_{t}|\mm)
\geq (1-t) \U_{N-1}(\bar \mu_0|\mm)
\]
Then we reach a contradiction as follows. 
Since $C_0$ is a subset of initial points, necessarily $C_{t} \cap C_0 = \emptyset$; 
 by Hausdorff convergence for each $\ve >0$ 
 for $t$ sufficiently close to $0$ we have 
\[
\mm((C_0)^\ve)
\geq 
\mm(C_0 \cup C_{t})= 
\mm(C_0) +\mm( C_{t}) 
\geq \mm(C_0) + (1-t)^{N-1}U_{N-1}(\bar \mu_0 |\mm)^{N-1} .
\]
Hence letting $t \to 0$ and then $\ve \to 0$ 
we reached a contradiction and proved the claim.
\end{proof}

Following the discussion on the compatibility of the notion of synthetic null hypersurfaces  with the smooth one
(see \Cref{R:compatibility}), we also report the
compatibility of the classical NEC with the synthetic $\NC^e(N)$ condition (see \Cref{D:syntheticnull}),
established in the authors' previous work~\cite{CMM24a}.

\begin{theorem}[Compatibility of $\NC^e$ with the smooth NEC]\label{thm:compatibilityNEC-NCe}
  Let $(M,g)$ be a Lorentzian manifold of
  dimension $n$.
  Let $H\subset M$ be a $C^2$ null hypersurface. 
  Assume that the classical NEC is valid along $H$,
        i.e., $\Ric(w,w)\geq 0$, for any normal vector $w$ to $H$.
    Then the following two properties hold: 
     \begin{enumerate}
     \item\label{Item:NCGlob} 
  For any open set  $U\subset M$, if $H\cap U$ is closed in $U$
 and admits a global cross-section $S$ with respect to $\leq _{U}$, then
 the synthetic null hypersurface $(H\cap U,\vol_{L}, G_{L,S})$ 
 in $X_{U}$ defined in \Cref{R:compatibility} satisfies 
 $\NC^e(n)$.

 \item\label{Item:NCLoc} 
 For any $x\in H$, there exists a neighborhood $U$ 
 of $x$ in $M$, such that
 the synthetic null hypersurface $(H\cap U,\vol_{L}, G_{L,S})$ in $X_{U}$
 defined in \Cref{R:compatibility-local}
  satisfies 
 $\NC^e(n)$.
 \end{enumerate}
 
\end{theorem}
\begin{proof}
Part~\ref{Item:NCGlob} follows by~\cite[Theorem~1.6]{CMM24a}. The implication~\ref{Item:NCGlob} $\implies$~\ref{Item:NCLoc}:
is the content of \Cref{R:compatibility-local}.
\end{proof}

We now complement \Cref{thm:compatibilityNEC-NCe} by showing that the classical 
NEC is also compatible with the synthetic $\NC^e$ along achronal boundaries, 
without any further regularity assumption.  Recall that achronal 
boundaries  fall within the class of synthetic null hypersurfaces, see \Cref{R:compatibility-penrose}.

\begin{theorem}\label{T:compatibility-penrose}
  Let $(M,g)$ be a Lorentzian manifold of
  dimension $n$.
  Assume $M$ to be causally closed and strongly causal.
  Let $A\subset M$ be a $C^2$ achronal, space-like sub-manifold of
  codimension $k\geq 2$.
  Let $H:=\partial I^{+}(A)$.
    Assume that the classical NEC is valid on $H$, i.e.,
  $\Ric(w,w)\geq 0$, for any
  light-like vector $w\in TH$ at the differentiability points of $H$.

  Then the synthetic null hypersurface $(H,G,\mm)$, as defined in
  \Cref{R:compatibility-penrose}, satisfies $\NC^e(n)$.
\end{theorem}
\begin{proof}
The claim is a consequence of \Cref{T:localization} and the constructions performed in \Cref{R:compatibility-penrose}, to which we refer for the notation. For the sake of brevity, we outline the argument, omitting some straightforward technical details. 

Recall the definition~\eqref{eq:defhatH} of $\hat{H}$ as a $C^2$-submanifold of $TM$ endowed with the measure $\hat\mm$ locally given by~\eqref{eq:defhatmmloc}. By its very definition, $\hat{H}$ is a ruled submanifold, generated by the future pointing light-like vectors $v\in \normal\initial$. It is convenient to parametrize such generators on a suitable set $Q$ of indices, that can be identified with a suitable measurable subset of $\hat H$.
  We disintegrate the measure $\hat\mm$ on $\hat{H}$ along such generators of $\hat H$, as
  follows
  \[
    \hat\mm
    =
    \int_{Q}
    \hat \mm_\alpha
    \,
    \qq(\de\alpha),
  \]
 where $\qq$ is a probability measure on $Q$, and each measure $\hat \mm_\alpha$ is concentrated on the future pointing half-line generated by $v_\alpha, \; \alpha \in Q$.
  Since the exponential map does not mix such generators, we can also
  disintegrate the measure $\mm$, as
  follows
  \begin{equation}
    \mm
    =
    \int_{Q}
    \exp_{\sharp}
    (\hat\mm_\alpha\llcorner_{B})
    \,
    \qq(\de\alpha).
  \end{equation}
  Fix a vector $v\in \hat H$. Consider the  $C^2$ null hypersurface $H_{v}$ defined in~\eqref{eq:defHv}
  and its rigged measure $\vol_{L^v}$. Since $H_{v}$ is $C^2$, then 
  $\Ric(w,w)\geq 0$ for all light-like $w\in TH_{v}$ in the 
  differentiability points of $H$. Moreover $H$ is a locally-Lipschitz sub-manifold, then a.e.\;differentiable. It follows that 
  \[\Ric(w,w)\geq 0, \quad \text{for all } w\in TH_{v}|_{H}.\]
  By disintegrating the rigged measure $\vol_{L^v}$ along the above family of generators, we get
  \[
    \vol_{L^v}
    =
    \int_{Q}
    \vol_{L^v,\alpha}
    \qq(\de\alpha)
    .
  \]
  For $\qq$-a.e.\;$\alpha\in Q$, the conditional measure $\vol_{L^v,\alpha}$ is absolutely continuous w.r.t.\;the
  1-dimensional Lebesgue measure on the corresponding generator, with density $a_{L^v,\alpha}$. 
  By~\cite[Theorem~5.11]{CMM24a}, such densities satisfy 
  \begin{equation}\label{eq:DiffIneqaLvalpha}
    (a_{L^v,\alpha})''(t)
    +
    \frac{((a_{L^v,\alpha})'(t))^2}{n-2}
    \leq
    0
    ,
    \quad
    t=-g(W,(\exp|_{U_{v}})^{-1}(x))
    ,
    \quad
    \forall x\in H_{v}\cap H
    .
  \end{equation}
  Since locally the measure $\mm$ coincides with $\vol_{L^v}$, the uniqueness of the disintegration implies that the conditional measure $\exp_{\sharp}
    (\hat\mm_\alpha\llcorner_{B})$ has density $a_{L^v,\alpha}$, for $\qq$-a.e.\,$\alpha\in Q$.
  For $v_1,v_2\in \hat{H}$, one can check that 
  \[
  a_{L^{v_1},\alpha}(t)
  = 
  a_{L^{v_2},\alpha}(t)
  ,
  \qquad
  t=-g(W,u)
  ,
  \qquad
  u\in U_{v_1}\cap U_{v_2}
  .
  \]
  Therefore, the disintegrated measures can be glued, finding that $\exp_{\sharp}(\hat\mm_\alpha\llcorner_{B})$ is absolutely continuous, w.r.t.\;the Lebesgue
  measure, with density $a_{\alpha}$ satisfying the differential inequality~\eqref{eq:DiffIneqaLvalpha}. The conclusion now follows from \Cref{T:localization}.
  %
  %
  %
\end{proof}


\section{Equivalence relations for measures and gauges}\label{Sec:CovarianceGaugeMeas}

In the smooth setting, the measure $\mm_{L}$ and
the gauge $G_{L,S}$ depend on the choice of the
vector field $L$.
Therefore,  a priori \Cref{D:syntheticnull} might
depend on these choices as well.  One of the goals of our previous work~\cite{CMM24a} was to show the independence of the $\NC^e$ condition from such a choice. 
The goal of this section is to show such independence  also in this 
more general setting.

In this section, we consider a given topological causal space 
$(X,\ll, \leq, \mathfrak{T})$ and
all the synthetic null hypersurfaces will be subsets of $X$.
%

Over $\noending{H}$ 
we can define two equivalence relations, one for measures
and one for gauges. 
Firstly we recall the concept of transverse function that in
the smooth setting can be given by asking $f$ to be constant along 
the generators of the null hypersurface.
Since this characterization can be stated without referring to 
the smooth structure, it is natural to consider it also in our framework. 

\begin{definition}[Transverse function]
Let $H$ be a closed achronal set. 
A function $f:\noending{H}\to\R$ is transverse
if, for any causal curve $\gamma$ in $\noending{H}$, $f\circ \gamma$ is constant.
\end{definition}


In the smooth setting, if $L_1,L_2\in\vfield(\normal H)$ are two null
geodesic vector fields, then $L_1=\phi L_2$ for some transverse
function $\phi$ and
the relation between the two rigged measures is $\vol_{L_1}=\frac{1}{\phi}
\vol_{L_2}$.
Accordingly, we identify measures which are equivalent up to multiplication by a transverse function.
\begin{definition}[Equivalence relation for measures]\label{D:equivm}
Let $H$ be a closed achronal set
and $\mm_1,\mm_2\in\M^{+}(H)$ be two measures.
%
We set $\mm_1\sim\mm_2$ 
if and only if 
\[
\mm_1\llcorner_{H\setminus\noending{H}}
=
\mm_2\llcorner_{H\setminus\noending{H}} \quad \text{and}\quad 
\mm_1\llcorner_{\noending{H}}=f\mm_2\llcorner_{\noending{H}},
\]
for some
strictly positive transverse function $f$.
\end{definition}

Continuing the analogy with the smooth setting, 
the definition of
$G_{L,S}$ is  depending both on $L$ and $S$
but the only relevant aspect of the theory is the set $G$-causal curves.
We therefore introduce an equivalence relation also between gauges. 

\begin{definition}[Equivalence relation for gauges]\label{D:equivG}
  Let $H$ be a closed achronal set and 
  let $G_1,G_2:\noending{H}\to\R$ be two gauges. 
  We set
  $G_1\sim G_2$ if and only if $\Caus_{G_1}=\Caus_{G_2}$ (see 
  \Cref{D:Gcausalcurve}).
\end{definition}

Next we prove that 
the $\NC^e(N)$ condition
is well defined for equivalence classes of gauges and reference measures, i.e.,\ the $\NC^e(N)$ condition
does not
depend on the choice of the measure and the gauge inside the respective
equivalence classes.

\begin{proposition}\label{lem:integral-transverse}
  Let $H$ be a closed achronal set, let $\nu\in\OptCaus^H(\mu_0,\mu_1)$ for a
  pair of  probability measures $\mu_0,\mu_1\in \Prob(H)$ which are 
  concentrated on $\noending{H}$. 
  
  Set $\mu_{t}:=(e_{t})_{\sharp}\nu$. Then, for every transverse function   $f:\noending{H}\to\R$, it holds that
  \begin{equation}
    \label{eq:integral-transverse}
    \int_{H}
    f
    \,\de\mu_{t}
    =
    \int_{H}
    f
    \,\de\mu_{s}
    ,
    \qquad
    \text{ for all }
    t,s\in[0,1]
    .
  \end{equation}

  If moreover $f$ is non-negative and the integrals
  above are equal to $1$,
  defining $\tilde \mu_{t}:= f\mu_{t}$,
  then 
  $\tilde\nu:=(f\circ e_0)\, \nu \in\OptCaus^H(\tilde\mu_0,\tilde\mu_1)$ 
  and $\tilde\mu_{t}=(e_{t})_{\sharp}\tilde\nu$.
\end{proposition}

\begin{proof}
  Since
  $f$ is transverse, then $f\circ\gamma$ is constant for all
  $\gamma\in \Caus_{H}$, such that $\gamma_{[0,1]}\subset\noending{H}$,
  therefore, since  $\nu(\Caus_{H})=1$,
   we can compute
  \begin{equation*}
    \int_{H}
    f
    \,\de\mu_{t}
    -
    \int_{H}
    f
    \,\de\mu_{s}
    =
    \int_{C([0,1];H)}
    (
    f\circ e_{t}-
    f\circ e_{s}
    )
    \,
    \de \nu
    =0
    .
  \end{equation*}
For the second part, a direct computation gives
\begin{equation*}
    \tilde\mu_{t}=f\mu_{t}
    =
    f(e_{t})_{\sharp}(\nu)
    =
    (e_{t})_{\sharp}(f\circ e_{t} \nu)
    =
    (e_{t})_{\sharp}(f\circ e_0 \nu)
    =
    (e_{t})_{\sharp}(\tilde\nu).
    \qedhere
  \end{equation*}
\end{proof}

\begin{proposition}\label{P:convexity-invariance}
Let $H$ be a closed achronal set, $\nu\in\OptCaus^H(\mu_0,\mu_1)$ for a pair of  probability measures $\mu_0,\mu_1\in \Prob(H)$ concentrated on $\noending{H}$ and  
let $\mm_1,\mm_2\in\M^{+}({H})$  such that $\mm_1\sim\mm_2$.

Then the function $t\mapsto \Ent(\mu_{t}|\mm_1)-\Ent(\mu_{t}|\mm_2)$ is constant.
\end{proposition}

\begin{proof}
  Let  $f$ be any transverse function such that 
  $\mm_1 = f \mm_2$ and
   $\rho_{k,t}$ be the density of $\mu_{t}$ w.r.t.\ 
   the measure
  $\mm_{k}$, $k=1,2$.
  Since $\rho_{2,t}=\rho_{1,t}f$, a direct
  computation gives
  \begin{align*}
    \Ent(\mu_{t}|\mm_{2})
    &
    =
    \int_{H}
    \log(\rho_{2,t})
    \,
    \de\mu_{t}
    =
    \int_{H}
    \log(\rho_{1,t})
    \,
    \de\mu_{t}
    +
      \int_{H}
    \log f
    \,
      \de\mu_{t}
    \\
    &
      =
      \Ent(\mu_{t}|\mm_{1})
    +
      \int_{H}
    \log f
    \,
    \de\mu_{t}
      \stackrel{\text{\eqref{eq:integral-transverse}}}{=}
      \Ent(\mu_{t}|\mm_{1})
    +
      \int_{H}
    \log f
    \,
    \de\mu_0
      .
      \qedhere
  \end{align*}
\end{proof}

The next result follows straightforwardly from the propositions above.

\begin{theorem}\label{T:invariance-NC-locally-compact}
Let $N>0$ and $(X,\ll, \leq, \mathfrak{T})$ be a topological causal space.
Consider a closed achronal set $H$,
two measures $\mm_1,\mm_2\in\M^+(H)$ and two \ascar\ gauges $G_1,G_2$ 
(recall~\eqref{eq:ascoli-arzela}).

If $\mm_1\sim \mm_2$ and $G_1\sim G_2$, 
then $(H,G_1, \mm_1)$ satisfies the $\NC^e(N)$ condition
  if and only if $(H,G_2, \mm_2)$ does.
\end{theorem}

\begin{proof}
  Assume $(H,G_1,\mm_1)$ to be $\NC^e(N)$.
  By~\Cref{P:extrimal-negligible},
  $\mm_1(H\setminus \noending{H})=0$, thus
  $\mm_2(H\setminus \noending{H})=0$.
  This means that the only relevant probability measures for the
  $\NC^e(N)$ condition are the ones concentrated on $\noending{H}$.
  Then \Cref{P:convexity-invariance} implies that also $(H,G_2,\mm_2)$
  satisfies $\NC^e(N)$.
\end{proof}

\section{The geometry of synthetic null hypersurfaces}\label{S:decomposition-null-hypersurface}

This section investigates more in details the geometry of synthetic null hypersurfaces, starting with relations between gauges and causal curves. 

In this section we will often assume the closed achronal set to be weakly convex;  this assumption is justified by \Cref{R:absolutelycontinuous}.

\begin{proposition}\label{P:G-parametrization}
  Let $H$ be a closed achronal set, $G$ be a gauge and $\gamma : [0,1] \to H$ be a
  causal curve.
  Then there exists a reparametrization of $\gamma$ which is
  $G$-causal.
\end{proposition}

\begin{proof}
By \Cref{R:intervalH}, 
we can assume  that 
$[0,1]\backslash\gamma^{-1}(\noending{H})\subset\{0,1\}$. 
Also assuming
$\gamma|_{(0,1)}$ injective is not restrictive.
Define $l(t):=G(\gamma_{t})$ and notice that by construction $l$ is strictly increasing.
  %
%
Let $(a,b):=l((0,1))$
and $m:(a,b)\to[0,1]$ be the inverse of $l|_{(0,1)}$.
Define $\eta:=\gamma\circ m:(a,b)\to H$.
This function is continuous.
  At this point extend $\eta$ by declaring that $\eta_{a}:=\gamma_0$ and
  $\eta_{b}=\gamma_1$.
  To check that this function satisfies the desired properties is trivial.
\end{proof}

\begin{example}
  Consider the Minkowski space $\R^2$ ($t$ is the time variable, $x$
  is the space variable).
  Let $H$ be the future light cone.
  Let $G:H\to \R$ given by $G(x,t)=\log t$.
  This function fails hypothesis (2) in the definition of gauge.

  Let $\gamma:[0,1]\to H$ be given by $\gamma_{r}=(r,r)$.
  There is no reparametrization of $\gamma$ such that its
  post-composition with $G$ is affine.
\end{example}

It is then immediate to obtain the following

\begin{corollary}
Let $H$ be a weakly convex closed achronal set, $G$ be a gauge and
  $x,y\in {H}$.
If $x\leq y$, $x\neq y$, then there exists a $G$-causal curve connecting $x$ to $y$.
\end{corollary}

\begin{proposition}\label{P:not-causally-related}
Let $H$ be a weakly convex closed achronal set and $G$ be a gauge for $H$.
Let $x,y\in \noending{H}$ be such that $G(x)=G(y)$.

Then either $x=y$ or $x$ and $y$ are not causally related. 
\end{proposition}

\begin{proof}
Assume that $x\neq y$
and, without loss of
generality, that $x\leq y$.
Then by weakly convexity of $H$ there exists $\gamma$, a causal
curve connecting $x$ to $y$, that, up to reparametrization, we can assume to be injective.
Then $G\circ \gamma$ is strictly increasing
giving a contradiction.
\end{proof}

\begin{proposition}\label{P:Hcausal}
  Let $H$ be a weakly convex closed achronal set admitting a gauge $G$.
  Then $H$ is causal, that is $\leq$ is a partial order relation over $H$ (or, equivalently, $H$ does not contain  causal loops).
\end{proposition}

\begin{proof}
  Assume by contradiction the existence of a non-constant causal curve 
  $\gamma : [0,1] \to H$ with $\gamma_0= \gamma_1$. 
  Without loss of generality $\gamma_0 \neq \gamma_{1/2}$ and $\gamma_{1/2} \neq \gamma_1$. 
  By definition of end points, 
  $\gamma([0,1])\subset \noending{H}$
  and by \Cref{P:G-parametrization} $\gamma$ can be assumed to be $G$-causal.
  
  Since $\gamma_0=\gamma_1$ and $G\circ \gamma$ is monotone, then $G\circ\gamma$ is constant. Thus
  $\gamma$ itself is constant by \Cref{P:not-causally-related}, giving a
contradiction.
\end{proof}

\subsection{Null non-branching and first properties}

Assuming that branching phenomena do not occur within $H$ permits to deduce much more on the structure of the achronal set. 
We start giving the definition of null non-branching.

\begin{definition}
Let $H\subset X$ be a closed achronal set. We say that $H$  is \emph{forward null non-branching} if the following holds. 
For every pair of injective
causal curves $\gamma^{1},\gamma^{2} : [0,1] \to
H$, such that $\gamma^1_0=\gamma^2_0$ and $\gamma^1_{1/2}=\gamma^2_{1/2}$,
then 
either $\gamma^{1}_{[0,1]} \subset \gamma^{2}_{[0,1]}$ or
$\gamma^{2}_{[0,1]} \subset \gamma^{1}_{[0,1]}$.

We say that $H$ is \emph{backward null non-branching} if $H$ is
forward null non-branching in the causally reversed space.

We say that $H$ is \emph{null non-branching} if it is both forward
and backward null non-branching.
\end{definition}

The next two propositions guarantees that the null non-branching condition
is available for a broad class of Lorentzian manifolds.

\begin{proposition}\label{prop:Lipschitz-nonBranch}
Let $(M,g)$ be 
a strongly causal Lipschitz Lorentzian manifold.
Let $S$ be an achronal set and $H = \partial I^+(S)$ its achronal boundary.
Assume $\noending{H}$ to be $C^{1,1}_{loc}$.
Then $H$ is null  non-branching.
\end{proposition}

\begin{proof}
  Since $\noending{H}$ is $C^{1,1}$, there exists a
  Lipschitz-continuous, null vector field
  $L\in\vfield(T\noending{H})$.
  First, we claim that if $\gamma:I\to\noending{H}$ is a causal
 curve, then there exists a reparametrization $\eta:J\to\noending{H}$ of $\gamma$, such
  that $\eta$ is $C^1$ and $\dot\eta=L$.
  Indeed, by \Cref{lem:lipschitz-reparametrization} we can assume
  $\gamma$ to be Lipschitz-continuous.
  Then in the differentiability points, $\dot\eta\in T\noending{H}$
  and $\dot\eta$ is a future-directed causal vector.
  It thus holds that $\dot\eta_{r}$ is  parallel to $L_{\eta_{r}}$.
  Therefore, we can define the $L^\infty$ function
  $\phi:I\to (0,\infty)$ by the property
  \[\dot\gamma_{r}=\phi(r)L_{\gamma_{r}}.\]
  If we define
  \[
    T(s):=
    \bigg(
    \int_0^s
    \phi(r)
    \,
    \de r
    \bigg)^{-1}
    \qquad
    \text{ and }
    \qquad
    \eta_{s}:=\gamma_{T(s)}
    ,
  \]
  it is then immediate to see that $\eta$ satisfies the desired
  property.
Consider now
a pair of injective causal curves $\gamma^{1},\gamma^{2} : [0,1] \to
U$, such that $\gamma^1_0=\gamma^2_0$ and
$\gamma^1_{1/2}=\gamma^2_{1/2}$.
  By the first part, we know that there exist
  $\eta^i: J^{i}\to\noending{H}$, which are  reparametrizations of
  $\gamma^i|_{(0,1)}$, satisfying $\dot \eta^i=L$, for $i=0,1$.
  We can assume $0\in J^i$ and $\eta^i_0=\gamma^{i}_{1/2}$.
  By continuity, they can be extended to the closure $\overline{J^i}$.
  Since these two curves satisfy the same ODE, then they coincide in
  the intersection $\overline{J^1}\cap\overline{J^2}$. The thesis follows.
\end{proof}

  \begin{proposition}\label{P:non-branching-curves}
Let $(M,g)$ be 
a strongly causal $C^2$ Lorentzian manifold.
Let $\gamma: I \to M$ be a null curve (i.e., a causal curve such that its
points are not chronologically related).

Then $\gamma$ is the reparametrization of a geodesic.
\end{proposition}
\begin{proof}
  By \Cref{lem:lipschitz-reparametrization}, we may assume $\gamma$ to
  be locally-Lipschitz and that $0\in I$.
  Since the statement is local, it suffices to  show that $\gamma$
  can be reparametrized into a geodesic in a neighborhood of $0$.
  Let $p=\gamma_0$.
  Let $H$ be the light-cone in $M$, emanating from $p$, and let
  $\hat{H}$
  be the light-cone in the Minkowski spacetime. Denote by $\hat{L}$ the radial vector field along $\hat{H}$.
  Since $g$ is $C^2$,
  the exponential is a map of class $C^1$, and therefore a
  diffeomorphism between a neighborhood of $0$ in $T_{p}M$ and a
  neighborhood of $p$.
  As a consequence, $H\setminus \{p\}=\exp(\hat H\setminus\{0\})$ is a
  $C^1$ sub-manifold of $M$. It is standard fact that $L:=\exp_*\hat{L}$ is a null geodesic vector field along $H$.
  We define the gauge $G(x):=|\exp^{-1}(x)|$, where $|\cdot|$ is the Euclidean norm in $T_{p}M$; the map $G$ is of class $C^1$.
  By strong causality, $\gamma$ takes values in $H$, therefore
  $\dot\gamma_{t}\in TH$, for a.e.\ $t$.
  Since by assumption $\dot \gamma$ is causal, it follows that  $\dot\gamma$ is
  parallel to $L$.
  Denoting  $\hat \gamma:=\exp^{-1}\circ \gamma$, we infer that 
  $\dot{\hat\gamma}$ is parallel to $\hat L$, a.e..
  As a consequence there exists $\hat \eta$, a reparemetrization of
  $\hat\gamma$, such that $\hat\eta_{t}=\hat{L}_{\hat\eta_{t}}$, for a.e.\
  $t$.
  Since $\hat L$ is a geodesic vector field, then $\hat\eta$ is a
  geodesic in $T_{p}M$, therefore, $\hat\eta$ is a linear ray starting
  from $0$.
  We conclude that $\eta:=\exp\circ \hat\eta$ is a geodesic in $(M,g)$.
\end{proof}
\begin{corollary}\label{C:null-non-branching-c2}
Let $(M,g)$ be 
a strongly causal $C^2$ Lorentzian manifold.

Then every closed achronal set is null non-branching.
\end{corollary}
\begin{proof}
  Let $H\subset M$ be a closed achronal subset and fix $\gamma^0,\gamma^1$ two causal
  curves in $H$, such that $\gamma_{s}^0=\gamma_{s}^1$,
  $s=0,\frac{1}{2}$.
  By the previous proposition, we can assume that $\gamma^i$ is a
  geodesic, $i=0,1$, with $\gamma^0_0=\gamma^1_0$ and
  $\gamma^0_{s_0}=\gamma^0_{s_1}$, with $s_0,s_1\in(0,1)$.
  The vectors $\dot\gamma^0_{s_0}$ and $\dot\gamma^0_{s_1}$ are
  parallel; if not, then $\gamma^0_{s_0}$ is a focal point for
  $\gamma_0^0$, thus $\gamma_0$ stops to be a null curve after $s_0$,
  a contradiction.
  We deduce that $\gamma^0$ and $\gamma^1$ are part of the same
  maximal geodesic and coincide at $0$. The thesis follows by the Cauchy{--}Lipschitz theorem, ensuring (existence and) uniqueness of solutions of the geodesic equation in a Lorentian metric $g$ of class $C^2$.
\end{proof}

We next draw some consequences from the null non-branching condition.
\begin{proposition}\label{P:uniqueness-G-causal}
  Let $H$ be a weakly convex, null non-branching, closed achronal set, $G$ a gauge for $H$, and
  $x,y\in {H}$, with $x\leq y$ and $x\neq y$.
  Assume that either $x\in\noending{H}$ or $y\in\noending{H}$.

  Then there exists a unique $G$-causal curve connecting $x$ to $y$.
\end{proposition}
\begin{proof}
    The existence is given by \Cref{P:G-parametrization}.

  We assume $x\in\noending{H}$, the other case being symmetric.
  Let $\gamma$ and $\eta$ be two $G$-causal curves from $x$ to $y$.
  Let $z\in J^{-}(x)\cap H$ with $z\neq x$, and let $\delta$ be a
  curve connecting $z$ to $x$.

  Consider first the case $G\circ\gamma=G\circ\eta$.
  If by contradiction $\gamma_{t}\neq \eta_{t}$, thus, $\gamma([0,1])\not\subset\eta([0,1])$, nor
  $\eta([0,1])\not\subset\gamma([0,1])$.
  By concatenating $\delta$ with $\gamma$ and $\delta$ with $\eta$, we
  obtain a contradiction with the non-branching hypothesis.

  Consider now the case $G\circ\gamma\neq G\circ\eta$.
  In this case, by affinity, since $\gamma_0=\eta_0=x\in\noending{H}$,
  then $G\circ\gamma$ and $G\circ\eta$ have different derivatives.
  If $G\circ \eta$ is steeper, take $\tilde\eta$ as a restriction and
  affine reparametrization of $\eta$.
  By following the previous part we deduce $\gamma=\tilde\eta$, which
  is a contradiction.
\end{proof}

\begin{remark}\label{R:measurablity-Gamma}
\Cref{P:uniqueness-G-causal} states that there  exists a map
\[\Gamma:H\times\noending{H}\cup
\noending{H}\times H\to C_{G},\quad  
\text{such that }\Gamma(x_0,x_1)_{i}=x_{i},\; i=0,1.
\]
Note that $\Gamma$  is the right inverse of $(e_0,e_1)$, therefore the graph of $\Gamma$ is closed. It follows that  $\Gamma$ is Borel. 
\end{remark}
  %
  %
\begin{proposition}\label{P:characterization-G-causal}
  Let $H$ be a null non-branching weakly convex closed achronal set 
  and let $G$ be a gauge.
  Assume that $\gamma$ is a non-constant curve in $H$, such that $\gamma_{s}\leq \gamma_{t}, \forall s\leq t$, continuous
  at the end-points, such that $\gamma_{(0,1)}\subset\noending{H}$ 
  and $G\circ \gamma|_{(0,1)}$ is affine and non-decreasing.

  Then $\gamma$ is continuous and therefore a causal curve.
\end{proposition}
\begin{proof}
  Fix $\epsilon>0$.
  Let $\eta^\epsilon:[\epsilon,1-\epsilon]\to H$ be the $G$-causal curve given by
  \Cref{P:uniqueness-G-causal} connecting $\gamma_\epsilon$
  to $\gamma_{1-\epsilon}$.
  Fix $t\in (\epsilon,1-\epsilon)$.
  Arguing as in  the proof of
  \Cref{P:uniqueness-G-causal}, one can check that
  $\gamma_{t}=\eta^\epsilon_{t}$.
  By taking the limit as $\epsilon\to 0$, we obtain that any
  restriction of
  $\gamma|_{(0,1)}$ to a closed interval is causal.
  Since $\gamma$ is continuous at the endpoints, we conclude.
\end{proof}

We prove that the sets of initial and final points are
negligible without the tightness assumption on the gauge (see \Cref{P:extrimal-negligible}) but assuming instead the non-branching condition.

\begin{theorem}\label{T:initialpoints}
Let $(H,G,\mm)$ be a null non-branching synthetic null hypersurface 
satisfying the \NEC\ 
$\NC^{e}(N)$, for some $N> 0$.
Then 
\[
\mm(\initial)=\mm(\final)
  =
  0
.
\]
\end{theorem}

\begin{proof}
The first part of the proof repeats verbatim the one of \Cref{P:extrimal-negligible}. Hence, 
we will argue by contradiction that $\mm(\initial\setminus\final)
> 0$ 
and deduce the existence 
of continuous map map $T : K \to H$, with $K\subset \initial\setminus\final$ of positive $\mm$-measure. 
Define the absolutely continuous measure 
$\mu_{0} : =  \mm\llcorner_{K}/ \mm(K)  \in \mathcal{P}(H)$ and 
$\mu_{1} : = T_{\sharp} \mu_{0}$ 
so that 
$({\rm Id},T)_\sharp \mu_0 \in \Pi_\leq(\mu_0,\mu_1)$.

\Cref{D:syntheticnull} gives then a measure
$\nu\in\OptGeo^G(\mu_0,\mu_1)$ such that
the function $t\mapsto \U_{N}((e_{t})_{\sharp}\nu)$ is concave.
By \Cref{R:absolutelycontinuous} $(e_{t})_{\sharp}\nu \ll \mm$. 
However, since $\mu_{0}$ is concentrated on the set of initial points $\initial$, necessarily $(e_{t})_{\sharp}\nu \perp (e_{s})_{\sharp}\nu$ for all $s \neq t$, 
giving a contradiction with the $\sigma$-finiteness of $\mm$:
indeed, suppose by contradiction that there exist two causal curves $\gamma^1, \gamma^2$,
such that $\gamma^1_{t}=\gamma^2_{s}$, with $1>t> s$ but $\gamma^1_0,\gamma^2_0\in K$.
Then by the null non-branching assumption 
($\gamma^1_{t}$ is not a final point),
necessarily $\gamma^2_{(0,1)}\subset \gamma^1_{(0,1)}$.
Hence, $\gamma^2_0$ cannot be an initial point, 
giving a contradiction.
\end{proof}

\begin{remark}\label{R:invariance-NC-non-branching}
As a consequence of \Cref{T:initialpoints}, we can prove 
that, under the null non-branching assumption, the $\NC^e(N)$ condition is independent of
the representative in the equivalence class of gauge and  measure.
That is, one can prove the same conclusion of \Cref{T:invariance-NC-locally-compact}, 
under the null non-branching assumption, in place of the tightness of the gauge.
\end{remark}

\subsection{Ray decomposition of achronal sets and disintegration}
We fix the following definition: the set $\mathcal{R} : = (J \cup
J^{-1})\cap (H\times H)$ will be called the \emph{transport relation}.

\begin{proposition}\label{L:Requivalence}
  Let $H$ be a null non-branching, weakly convex closed achronal set.
  Then $\relation$ is an equivalence relation on $\noending{H}$.
\end{proposition}

\begin{proof}
The symmetric and reflexive properties are trivially verified. We are therefore only left with checking the transitive property. 

Suppose then $(x,y), (y,z) \in \mathcal{R}$. 
Without any loss of generality we can assume that all these points are
different and that $x \leq y$ and $z \leq y$ (the other cases are either trivial or equivalent to this one). 
Assume on the contrary that neither $x\leq z$, nor $z\leq x$.
Let $p\in J^{+}(y)\cap H$, such that $p\neq y$.
Let $\gamma$ and $\delta$ be two curves connecting $x$ and $z$
(respectively) to $p$, passing through $y$.
Up to reparametrizing we can assume that $\gamma_{1/2}=\delta_{1/2}=y$, a
contradition with the backward non-branching-ness.
\end{proof}

If $\alpha\in \noending{H}$, we will denote by $\noending{H}_\alpha$
the equivalence class containing $\alpha$, also called ray:
\begin{equation}
  \noending{H}_\alpha
  :=
  \{
  x\in \noending{H}
  :
  (x,\alpha)\in\relation
  \}
  .
\end{equation}
We will denote by ${H}_\alpha$ the ray with
end-points:
\begin{equation}
  {H}_\alpha
  :=
  \{
  x\in H
  :
  (x,\alpha)\in\relation
  \}.
\end{equation}
The set $H_\alpha$ is closed, because $\relation$ is closed.
We will denote by $ \overline{\noending{H}_\alpha}$ the topological
closure of $ \noending{H}_\alpha$ in $H$.
By definition of closure $\overline{\noending{H}_\alpha}\subset
H_\alpha$.

\begin{remark}
As a consequence of non-branching, each of the sets $H_\alpha\cap\initial$
and $H_\alpha\cap\final$ contains at most one element.
Therefore, $\overline{\noending{H}_\alpha}=H_\alpha$.
Indeed, by contradiction assume
$\{x\}= H_\alpha\cap \initial\setminus \overline{\noending{H}_\alpha}$.
Since $x\leq \alpha$, by weakly convexity of $H$, there exists a $G$-causal curve $\gamma$ connecting $x$ to $\alpha$.
Since $\gamma_{t}\in \noending{H}_\alpha$, $t>0$, then $x=\gamma_0$ can be approximated by points in $\noending{H}_\alpha$.
\end{remark}

Since the causal relation $\mathcal{R}$ is  induced by $\leq$, we obtain the following consequence of 
\Cref{L:Requivalence}.

\begin{proposition}\label{C:parametrization}
Let $H$ be a null non-branching, weakly convex closed achronal set,
admitting a gauge $G$.
Then each equivalence class with end-points 
of $\mathcal{R}$ inside $H$ is contained in the image 
of a causal curve. 
\end{proposition}
\begin{proof}
Fix $\alpha\in\noending{H}$.
Consider the set
\begin{equation}
\zeta : = \{ (s, z )\in  \R \times \noending{H}_\alpha \colon  G(z) = s +G(\alpha)  \}
\end{equation}
Let $I:=P_1(\zeta)\subset \R$.
We want to see that the set $\zeta$ is the graph of a $G$-causal
curve covering all $\noending{H}_\alpha$.
By \Cref{P:not-causally-related}, for any $s$ there exists at most one $z\in\noending{H}_\alpha$, such that $G(z)=s$.
We need to see that $I$ is an interval, i.e., that $I$ is connected.
Given $s,t\in I$, there exist $x,y\in \noending{H}_\alpha$, such that $(s,x),(t,y)\in\zeta$.
Thus either $x\leq y$ or $y\leq x$.
Take $\gamma$ a causal curve connecting $x$ to $y$, contained in $\noending{H}_\alpha$.
By continuity of $G\circ \gamma$, $I$ contains all the elements between $s$ and $t$
, hence it is an interval. 
Therefore, $\zeta$ is the graph of a function $\eta:I\to \noending{H}_\alpha$ and $\eta$ satisfies $G(\eta_{t})=G(\alpha)+t$.

It holds that $\eta_{s}\leq \eta_{t}$, for all $t\leq s$, therefore we can apply 
\Cref{P:characterization-G-causal} and deduce that $\eta$ is $G$-causal.
Finally, for any point $w\in \noending{H}_\alpha$, $(G(w)-G(\alpha),w)\in \zeta$, 
therefore $\eta_{I}=\noending{H}_\alpha$.

We now extend $\eta$ in order to cover the endpoints.
If $\initial\cap H_\alpha$ is empty there is nothing
to do, thus we consider $x\in\initial\cap H_\alpha$.
Let $\delta:[0,1]\to H$ be the $G$-causal
curve connecting $x$ to $\alpha$.
Let $r:=\inf_{t>0} G(\delta_{t})$.
We claim that $\inf I=r$.
If $\inf I<r$, then one can take the curve starting from
$\eta_{r-\epsilon}$, for $r-\epsilon\in I$, passing through $\delta_{1/2}$,
and reaching $\delta_1$, obtaining a contradiction; the other inequality
is trivial.
We can therefore define $\eta_{r}:= x$.
Both $G\circ \eta$ and $G\circ\delta$ are affine and increasing functions, thus 
$T:=(G\circ\delta)^{-1}\circ(G\circ\eta)$ 
is affine and increasing.
Since $G(\eta_{t})=G(\delta_{T(t)})$, for $\eta>r$, by defintion of $\zeta$ and $\eta$, we deduce $\eta_{t}=\delta_{T(t)}$, $t>r$.
As a consequence $\lim_{t\to r^+}\eta_{t}=\lim_{s\to 0^+} \delta_{s}=x$, proving 
that $\eta$ is continuous in $r$, and therefore causal.
\end{proof}

Notice that in~\Cref{C:parametrization} we have constructed a parametrization 
of each equivalence class with $G$-speed equal to one. 
We define the map $\gflow_{G}:\dom(\gflow_{G})\subset \noending{H}\times \R\to \noending{H}$, as the
map whose graph is given by
\begin{equation}
  \Graph(\gflow_{G})
  =
  \{
  (x,t,y):
  G(y)-G(x)=t
  \}
  .
\end{equation}
Following \Cref{C:parametrization}, we can extend $\gflow_{G}$ 
to cover also the possible end-points.

The next proposition characterizes the equivalence of gauges as an
affine transformation with transverse coefficients.

\begin{proposition}\label{P:characterization-equivalence}
Let $H$ be a null non-branching, weakly convex,
closed achronal set,
  admitting two gauges $G_1$ and $G_2$.
  Then $G_1\sim G_2$, if and only if $G_1=f+h G_2$, for some
  transverse functions $f,h$, with $h>0$.
\end{proposition}
\begin{proof}
  The ``if'' part is trivial, so we consider the other implication.
  Fix $G_1\sim G_2$ two equivalent gauges, and let
  \[
  h(z):=
  (G_1\circ \Psi_{G_2}(z,\,\cdot\,))'(0)
  .
  \]
  Since 
  $\gflow_{G_2}(z,\,\cdot\,)$ is $G_2$-causal, hence $G_1$-causal, 
  therefore $G_1\circ\gflow_{G_2}(z,\,\cdot\,)$ is affine;
  it follows that the derivative in the equation  above exists
  and $h$ is transverse.
  Since
  \[
  G_2(\gflow_{G_1}(z,h(z)\, t))
  =G_2(z) +t
  ,
  \]
  then $\gflow_{G_1}(z,h(z)\, t)=\gflow_{G_2}(z, t)$.
  We can thus compute
  \begin{align*}
    \big((G_1- h G_2)\circ \gflow_{G_1}(z,t)\big)'
    &
    =
    1-h(z) \big(G_2\circ \gflow_{G_1}(z,t)\big)'
    \\
    &
    =
     1-h(z) \bigg(G_2\circ \gflow_{G_2}\bigg(z,\frac{t}{h(z)}\bigg)\bigg)'
    =
   0
    ,
  \end{align*}
  therefore $G_1-h G_2=:f$ is transverse.
\end{proof}

The next proposition guarantees the existence of a $\mathcal{A}$-measurable cross-section for the
equivalence relation $\relation$.
Here $\mathcal{A}$ is the $\sigma$-algebra generated 
by the set of analytic sets, i.e., projections of Borel sets inside a Polish space.

Recall also that we use the notation $x<y$ to denote  $x\leq y, \, x\neq y$.

\begin{proposition}[Measurable selection]\label{P:cross-section}
Let $H$ be a null non-branching, causal, weakly convex closed achronal set.
Then there exists a $\mathcal{A}$-measurable quotient map $\QQ :
\noending{H} \to \noending{H}$ for the equivalence relation
$\mathcal{R}$, i.e.,
for all $x \in H$, $(x,\QQ(x)) \in \relation$  and 
\[
\QQ(x) = \QQ(y) \iff (x,y) \in \relation(x).
\]
\end{proposition}

\begin{proof}
  We construct a cross-section $Q\subset\noending{H}$, as follows.
  Let $(U_{n})_{n}$ be a countable base for the topology of $X$, such that $U_{n}$
  are precompact.
  Consider the compact sets $C_{n}:=\overline{U_{n}\cap\noending{H}}\subset H$.
  Define
  \begin{align}
    &
      \begin{aligned}
      A_{n}
      :=
        P_1((J^{-1}\setminus\Delta)\cap (C_{n}\times C_{n}))
        =
        \{x\in C_{n}\colon \exists y\in C_{n}, x<y\}
        ,
      \end{aligned}
  \end{align}
  The sets $A_{n}$ are $\sigma$-compact, as they are projections of
  intersections of $\sigma$-compact sets ($\Delta$ is the diagonal,
  whose complementary is open, thus $\sigma$-compact,
  by the properness of the topology).
  Define
  \begin{align}
    \label{eq:maximal-elements}
    &
    L_{n}
    :=
    C_{n}\cap \noending{H}
      \setminus A_{n}
      =
      \{
      x\in
      C_{n}
      \cap \noending{H}
      \colon
      \nexists y\in C_{n}
      ,
      x<y
      \}
    ,
  \end{align}
  which is a Borel set.
  Define
    \begin{align*}
    B_{n}
      :=
    P_1(\relation\cap(\noending{H}\times L_{n}))
        \cap 
        \noending{H}
    =
    \{ x\in \noending{H}\colon
       \text{the set }
       \noending{H}_{x}\cap L_{n}
       \text{ is non-empty}
       \}
    ,
    \end{align*}
  which is an analytic set.
  Finally, define
  \begin{equation}
    E_{n}:=
    B_{n}
    \setminus
    \bigg(
    \bigcup_{i=1}^{n-1}
    B_1
    \bigg)
    ,
    \qquad
    Q_{n}
    :=
    L_{n}
    \cap E_{n}
    ,
  \end{equation}
  which belong to the $\sigma$-algebra generated by analytic sets.
  Define 
\begin{equation}\label{eq:def:Qpm}
Q:=\bigcup_{n=1}^{\infty} Q_{n}\subset V.
  \end{equation}
  Next, we show that  the set $Q$ is indeed a cross-sections.

{\bf Step 1.} For all $x\in\noending{H}$, $n\in\N$,  $L_{n}\cap
      \noending{H}_{x}$ contains at most one element.
      Fix $x\in\noending{H}\cap\relation(U)$.
  Assume on the contrary that there exists $v,w\in L_{n}\cap
  \noending{H}_{x}$, $w\neq v$.
  Then either $v\leq w$ or $w\leq v$.
  If $v\leq w$, then $v\notin L_{n}$, because~\eqref{eq:maximal-elements} would be contradicted by $y=w$.
  Analogously, $w \leq v$ leads to a contradiction. 

{\bf Step 2.} For all $x\in\noending{H}$, there exists $n\in\N$, such that
  $L_{n}\cap\noending{H}_{x}$
  contain one element.
  
      Fix $x\in\noending{H}$.
  Then there exist two points $y,z\in \noending{H}_{x}$,
  such that $y< x< z$.
  Since $H$ is causal, $x\notin J^{-}(y)\cup J^{+}(z)$, 
  therefore there exist $n$ and $W$, such that $U_{n}$
  is a neighborhood of $x$, and $W$ a neighborhood of
  $J^{-}(y)\cup J^{+}(z)$, such that $W\cap U_{n}=\emptyset$.
  Thus, $C_{n}\cap W=\emptyset$.
  Let $\gamma:[0,1]\to \noending{H}$ be a causal curve connecting $x$
  to $z$.
  Let $\bar t:=\sup\{t:\gamma_{t}\in C_{n}\}<1$ and let $v=\gamma_{\bar{t}}\in C_{n}\cap \noending{H}_{x}$.

  We claim that $v\in L_{n}$; if not, there exists $w\in C_{n}$, such
  that $v<w$, thus $w\in \noending{H}_\alpha$.
  If $w\in J^{+}(z)\subset W$, then we obtain a contradiction with
  $C_{n}\cap W=\emptyset$, thus $w\leq z$; by forward
  non-branching, there exists $s\in (\bar t, 1)$, such that
  $w=\gamma_{s}$, a contradiction with the definition of $\bar t$.

  We notice that $v\neq x$, otherwise $x\in \partial U_{n}$, a
  contradiction with the fact that $U_{n}$ is neighborhood of $x$.
  %

An immediate consequence of the last step is that for all $x\in
  \noending{H}$ there exists $n$, such that $x\in B_{n}$.
  Therefore, it holds that
  $Q_{n}\cap \noending{H}_{x}= L_{n}\cap\noending{H}_{x}$, for the
  smallest $n$, such that $x\in B_{n}$.
  As a consequence, for all $x$, $Q\cap \noending{H}_{x}$ is a
  singleton.

\textbf{Step 3.} Construction and measurability of the quotient map $\QQ$.
\\  Define the map $\QQ$ as $\QQ(x)=\alpha$, if
  $(x,\alpha)\in\relation$ and $\alpha\in Q$.
  In other words, the graph of the map $\QQ$ is
  given by
  \begin{equation}
    \Graph \QQ
    =
    \relation
    \cap
    ( \noending{H}
    \times
    Q
    )
    .
\end{equation}
We next show that the map $\QQ$ is measurable.
We adopt the convention that $\times$ is computed before $\cap$.
  Fix  an open set $W\subset X$ and compute
  \begin{align*}
    \QQ^{-1}(W)
    &
      =
      P_1(
    \relation
    \cap
     \noending{H}
      \times
      (
      Q
      \cap W)
      )
      \overset{\eqref{eq:def:Qpm}}{=}
      P_1
      \bigg(
    \relation
    \cap
    \noending{H}
      \times
      \bigcup_{n=1}^{\infty}
      \big(
    L_{n}
      \cap W
    \cap E_{n}
      \big)
      \bigg)
    \\
    &
      =
      \bigcup_{n=1}^{\infty}
      P_1
      (
    \relation
    \cap
    \noending{H}
      \times
      (
    L_{n}
      \cap W
    \cap E_{n}
      )
      )
    \\
      \overset{\eqref{eq:stable-set}}&{=}
      \bigcup_{n=1}^{\infty}
      P_1
      (
    \relation
    \cap
    \noending{H}
      \times
      (
      L_{n}
      \cap W
      )
      \cap
      E_{n}\times
      E_{n}
      )
    \\
    &
      =
      \bigcup_{n=1}^{\infty}
      P_1
      (
    \relation
    \cap
    \noending{H}
      \times
      (
      L_{n}
      \cap W
      )
      )
      \cap
    E_{n}
      ,
  \end{align*}
  which belongs to the $\sigma$-algebra generated by analytic sets;
  we have used the fact
  \begin{equation}
    \label{eq:stable-set}
    \relation\cap
    \noending{H}
    \times
    (
    E_{n} \cap Y
    )
    =
    \relation\cap
    E_{n}
    \times
    E_{n}
    \cap
    \noending{H}\times Y
    ,
    \qquad
    \forall Y\subset X
    .
    \qedhere
  \end{equation}
\end{proof}

Once a measurable quotient map is available, one can invoke Disintegration theorem and obtain a strongly consistent 
disintegration formula for any element of the class $[\mm]$ with respect to $\mathcal{R}$. 
We refer to~\cite[Section 4.2]{CaMo:20} for this implication in the synthetic Lorentzian setting; for metric spaces we point to the references 
the cited in~\cite[Section 4.2]{CaMo:20} (see
in particular~\cite[Th.~3.4]{CaMoLap}). 
 We only recall that we will denote by $Q$ the quotient set associated to $\QQ$, i.e.\ $Q = \QQ(H)$.

\begin{theorem}\label{T:disintegration}
Let $(H,G,\mm)$ be a weakly convex, null non-branching synthetic null hypersurface.
Then the following disintegration formula holds:
\[
\mm\llcorner_{\noending{H}} = \int_{Q} \mm_{\alpha} \, \qq(\de \alpha),
\]
where $\qq$ is a Borel probability measure over $Q \subset \noending{H}$ such that 
$\QQ_{\sharp} \mm\llcorner_{\noending{H}}  \ll \qq$ and the map 
$Q \ni \alpha \mapsto \mm_{\alpha} \in \mathcal{M}_{+}(\noending{H})$ satisfies the following properties:
\begin{enumerate}
\item for any $\mm$-measurable set $B$, the map $\alpha \mapsto \mm_{\alpha}(B)$ is $\qq$-measurable;
\item for $\qq$-a.e. $\alpha \in Q$, $\mm_{\alpha}$ is concentrated on $\QQ^{-1}(\alpha) = H_{\alpha}$ (strong consistency);
\item for any $\mm$-measurable set $B\subset\noending{H}$ and $\qq$-measurable set $C$, the following disintegration formula holds: 
\[
\mm(B \cap \QQ^{-1}(C)) = \int_{C} \mm_{\alpha}(B) \, \qq(\de \alpha);
\]
\item  For every compact subset $\mathcal K\subset X$ there exists a constant $C_{\mathcal K}\in (0,\infty)$ such that
\[
 \mm_{\alpha}(\mathcal K) \leq C_{\mathcal K}, \quad \text{for $\qq$-a.e. $\alpha\in Q$.}
\]
\end{enumerate}
Moreover, fixed any $\qq$ as above such that $\QQ_{\sharp} \mm \ll \qq$, the disintegration is $\qq$-essentially unique.
\end{theorem}

A few comments are in order.
If $\mm_1\sim \mm_2$ and $\qq\in\Prob(Q)$ is such that
$\q\ll\QQ_{\sharp}\mm_1\llcorner_{\noending{H}}$, then also $\qq\ll\QQ_{\sharp}\mm_2\llcorner_{\noending{H}}$.
Moreover, if $\mm_2\llcorner_{\noending{H}}=f\mm_1\llcorner_{\noending{H}}$ for some transverse function, then
$\mm_{2,\alpha}=f(\alpha)\mm_{2,\alpha}$, for $\q$-a.e.\ $\alpha\in Q$.

\section{Light-like Optimal Transport  and localization}\label{Sec:OT-Loc}

In this section, we will establish existence and uniqueness results for the optimal transport problem along a synthetic null hypersurface $H$, and use these to localize the synthetic Ricci curvature lower bounds to the measures $\mm_\alpha$  concentrated on the null generators $H_\alpha$ of $H$, as in the disintegration  \Cref{T:disintegration}.

\subsection{Existence and uniqueness of a monotone light-like Optimal Transport}

\begin{definition}[Monotone set]\label{D:monotoneset}
Given $(Y,\leq)$ a partially ordered set, $A\subset Y \times Y$ is said to be \emph{monotone}
if 
\[
x_1 \leq x_2, \ x_1 \neq x_2
\Longrightarrow y_1\leq y_2, \quad \text{for all } (x_1,y_1), (x_2,y_2) \in A.
\]
\end{definition}

\begin{definition}[Monotone plans]\label{D:monotoneplan}
  Given a topological causal space $(X,\ll,\leq,\mathfrak{T})$ and two
  probability measures $\mu_0,\mu_1\in \Prob(X)$, a coupling
  $\pi \in \Pi_{\leq}(\mu_{0},\mu_{1})$ is called \emph{monotone} if
  it is concentrated on a Borel set $A \subset X\times X$ that is
  monotone in the sense of \Cref{D:monotoneset} in $(X,\leq)$.
\end{definition}

\begin{lemma}\label{lem:monotone}
  Let $H$ be a null non-branching, weakly convex, 
  closed achronal set admitting 
  a gauge function $G$.
  Let $A\subset J\cap (\noending{H}\times \noending{H})$.
  The following are equivalent.
  \begin{enumerate}
  \item
    The set $A$ is monotone in $(H,\leq)$.
  \item
    for all $\alpha\in Q$, there exists a monotone set
    $B_\alpha\subset\R\times\R$, such that
    \begin{equation}
      A\cap (\noending{H}_\alpha\times \noending{H}_\alpha)
      \subset
      \gflow_{G}(\alpha,\,\cdot\,)
      \otimes
      \gflow_{G}(\alpha,\,\cdot\,)
      (B_\alpha)
    \end{equation}
  \end{enumerate}
\end{lemma}
\begin{proof}
  (1) $\implies$ (2).
  \quad
  Take as
  $B_\alpha:=( \gflow_{G}(\alpha,\,\cdot\,) \otimes \gflow_{G}(\alpha,\,\cdot\,))^{-1}(A)$.
  The fact that $B_\alpha$ is monotone is trivial.

  (2) $\implies$ (1).
  \quad
  Fix $(x_1,y_1),(x_2,y_2)\in A$.
  If $x_1< x_2$, then all these four points belong to the
  same class $H_\alpha$.
  Then $x_{i}=\gflow_{G}(\alpha, t_{i})$ and $s_{i}=\gflow_{G}(\alpha, s_{i})$, for some
  $(t_{i},s_{i})\in B_\alpha$, $i=1,2$.
  The assumption $x_1\leq x_2$ gives $t_1\leq t_2$ and the
  monotonicity of $B_\alpha$ gives $s_1\leq s_2$, concluding the
  proof.
\end{proof}

\begin{proposition}\label{P:monotone-transport-plan}
Let $H$ be a non-branching, weakly convex,
closed achronal set, admitting a gauge function $G$.
Let
$\mu_0,\mu_1\in \Prob(\noending{H})$ such that 
$\Pi_{\leq}(\mu_0,\mu_1)\neq \emptyset$.

Then there exists a unique monotone coupling
$\pi \in \Pi_{\leq}(\mu_0,\mu_1)$.
\end{proposition}

\begin{proof}
  Let $Q\subset\noending{H}$ be a measurable cross-section of $\relation$ and let
  $\QQ:\noending{H}\to Q$ be the quotient map.
  \smallskip
  
{\bf Step 1.} Existence.
 Let $\bar\pi\in\Pi_{\leq}(\mu_0,\mu_1)$.
  Consider the
  map 
  \begin{equation}\label{eq:defQ1}
  \QQ^1:=\QQ\circ P_1: \noending{H}\times \noending{H}
  \to Q,
  \end{equation}
  where $P_1$ is the projection on
  the first variable.
  By definition of coupling, we get
  \begin{equation}
    \qq:=
    \QQ_{\sharp}(\mu_1)
    =
    \QQ^1_{\sharp}(\bar\pi)
    .
  \end{equation}
  We are therefore in position to apply the disintegration theorem,
  deducing that
  \begin{equation}
    \bar\pi
    =
    \int_{Q}
    \bar\pi_\alpha
    \,
    \qq(\de\alpha)
    ,
  \end{equation}
  where the probability measures $\bar\pi_\alpha$ are concentrated on
  the fibers of $\QQ^1$.
  Since $\pi(J)=1$, then the measure $\bar\pi_\alpha$ are concentrated
  on
  \begin{equation}
    (\QQ^1)^{-1}(\alpha)
    \cap J
    =
    \{(x,y)\in \noending{H}: (x,\alpha)\in\relation\}
    \cap J
    \subset
    \noending{H}_\alpha
    \times
    \noending{H}_\alpha
    .
  \end{equation}
  At this point, for $i=0,1$, we let
  \begin{equation*}
    \tilde{\bar{\pi}}_\alpha
    :=
    (
    \gflow(\alpha,\,\cdot\,)^{-1}
    \otimes
    \gflow(\alpha,\,\cdot\,)^{-1}
    )_{\sharp}
    \bar\pi_\alpha
    \in\Prob(\R^2)
    \quad
    \text{ and }
    \quad
    \tilde\mu_{i}
    :=
    (P_{i})_{\sharp}    \tilde{\bar{\pi}}_\alpha
    \in\Prob(\R).
  \end{equation*}
  Let $\tilde\pi_\alpha\in \Prob(\R^2)$ be the monotone rearrangement
  of $\tilde\mu_0$ into $\tilde\mu_1$, and let
  $B_\alpha:=\supp\tilde\pi_\alpha$.
By~\cite[Lemma~6.4]{CMM24a},
    we have that 
    \[B_\alpha\subset\{(x,y) \in \R^2: x\leq y\}.\]
  Define
  \begin{equation}
    \pi:=
    \int_{Q}
    (
    \gflow_{G}(\alpha,\,\cdot\,)
    \otimes
    \gflow_{G}(\alpha,\,\cdot\,)
    )_{\sharp}
    \tilde\pi_{\alpha}
    \,
    \de\alpha
    .
  \end{equation}
  Since $\gflow_{G}(\alpha,\,\cdot\,)$ is order-preserving
  and $\tilde{\bar{\pi}}_\alpha$ is concentrated on 
  $B_\alpha\subset\{(x,y)\colon y\geq x\}$,  $\pi$ is concentrated on $J$.
Finally, it is a routine calculation to see that $\pi$ is coupling between $\mu_0$ and $\mu_1$.

\smallskip

{\bf Step 2.} Uniqueness.
  Let $\pi^1$ and $\pi^2$ be two monotone transport plans.
  It holds that
  \begin{equation}
    \QQ^1_{\sharp}\pi^1
    =
    \QQ^1_{\sharp}\pi^2
    =:
    \q
    .
  \end{equation}
  We can therefore apply the disintegration theorem, finding that
  \begin{equation}
    \pi^{i}
    =
    \int_{Q}
    \pi^i_\alpha
    \,
    \q(\de\alpha).
  \end{equation}
Let
\begin{equation*}
    \tilde\pi^k_\alpha
    :=
    ((\gflow_{G}(\alpha,\,\cdot\,))^{-1}\otimes
    (\gflow_{G}(\alpha,\,\cdot\,))^{-1})_{\sharp}\pi^k_\alpha
    \in\Prob(\R^2).
\end{equation*}
Since $\tilde\pi_\alpha^1$ and $\tilde\pi_\alpha^2$ have the same
marginals,
by the uniqueness of
the monotone rearrangement,
we deduce
$\tilde\pi_\alpha^1=\tilde\pi_\alpha^2$ giving $\pi^1 = \pi^2$.
\end{proof}

\begin{proposition}\label{P:G-causal-transport-plan}
  Let $H$ be a null non-branching, weakly convex, closed achronal set, admitting a gauge function $G$.
  Let $\mu_0,\mu_1\in \Prob(\noending{H})$ be
  two probability measures, and assume there exists a causal coupling
  $\pi\in\Pi_{\leq}(\mu_0,\mu_1)$.

  Then there exists a unique $\nu\in\OptGeo^{G}(\mu_0,\mu_1)$, $G$-causal dynamical transport
  plan, such that $(e_0\otimes
  e_1)_{\sharp}\nu=\pi$.
\end{proposition}
\begin{proof}
  The plan $\pi$ gives full
  measure to $B:=J\cap (\noending{H}\times \noending{H})$.
  Consider the map $\Gamma:B\to C_{G}$ given by 
  \Cref{R:measurablity-Gamma}, such that, for each 
  $(x_0,x_1)\in B$, $\Gamma\big((x_0,x_1)\big)_{i}=x_{i}$, $i=0,1$.
  Then  $\nu:=\Gamma_{\sharp}\pi$, satisfies the claim.
\end{proof}

\begin{corollary}\label{C:monotone-G-causal}
  Let $H$ be a null non-branching, weakly convex, closed achronal set,
  admitting a gauge function $G$. 
  Let $\mu_0,\mu_1\in \Prob(\noending{H})$ be
  two probability measures, such that $\Pi_{\leq}(\mu_0,\mu_1)\neq\emptyset$.

  Then there exists $\nu\in\OptGeo^{G}(\mu_0,\mu_1)$, a unique $G$-causal, dynamical optimal
  transport plan, such that 
  $(e_0,e_1)_{\sharp}\nu$ is a  monotone coupling.
\end{corollary}
\begin{proof}
    Combine
  \Cref{P:monotone-transport-plan} with
  \Cref{P:G-causal-transport-plan}.
\end{proof}

As the following example shows, the requirement that the marginal
measures give no mass to the endpoints is necessary for the existence of a monotone coupling and 
it cannot be weakened by the requirement that only one
measure is concentrated on $\noending{H}$.

  \begin{example}
    Consider the $2$-dimensional Minkowski space $\R^2$ and let
    $H=\{(x,|x|)\}$.
    Consider the measures
    \[\mu_0=\frac{1}{2}(\delta_{(0,0)}+\delta_{(1,1)}), \quad  \mu_1=\frac{1}{2}(\delta_{(-2,2)}+\delta_{(2,2)}).\]
    The only transport plan $\pi\in\Pi_\leq(\mu_0,\mu_1)$ is
    induced by a map $T$ given by $T(0,0)=(-2,2)$ and
    $T(1,1)=(2,2)$.
    Notice that
    \[
    \supp \pi=\{((0,0),(-2,2)),((1,1),(2,2))\}.
    \]
    %
    %
    For this set, it holds that $(1,1)\in J^{+}(0,0)$, but
    $(2,2)\notin J^+(-2,2)$, thus the plan is not monotone.
\end{example}

\subsection{Localization of the \NEC}

\begin{theorem}\label{T:localization}
  Let $(H,G,\mm)$ be a null-non-branching 
  synthetic null hypersurface.
  Then the following are equivalent.
  \begin{itemize}
      \item 
  $(H,G,\mm)$ satisfies the \NEC\ $\NC^{e}(N)$.

\item
    Considering the disintegration given by \Cref{T:disintegration},
 for $\q$-a.e.\ $\alpha\in Q$, the metric measure space
  $(H_\alpha, |G(\,\cdot\,) - G(\,\cdot\,)|, \mm_\alpha)$ 
  satisfies the $\CD(0,N-1)$ condition.
  That is, using the identification $\gflow_{G}(\alpha,\,\cdot\,)$ between $H_\alpha$ 
  and a real interval $I_\alpha\subset \R$,
  the measure $\mm_\alpha$ is absolutely continuous and, denoting
  by $h_\alpha:I_\alpha\to[0,\infty)$ its density, 
  $t\mapsto h_\alpha^{1/(N-2)}(t)$
 is concave.
 \end{itemize}
\end{theorem}

\begin{proof}
First notice that by \Cref{R:absolutelycontinuous}, 
the $\NC^{e}(N)$ condition implies that $H$ is weakly convex. Hence \Cref{T:disintegration} applies.

  One can check that $(H_\alpha, |G(\,\cdot\,) - G(\,\cdot\,)|)$ is a metric space isometric to the metric 
  space  $(I_\alpha,|\cdot-\cdot|)$, via the map $\gflow_{G}(\alpha,\,\cdot\,)$.
  Therefore, it is sufficient to check
  that $(I_\alpha,|\,\cdot\,|,\tilde\mm_\alpha)$ satisfies the $\CD(0,N-1)$
  condition, where
  $\tilde\mm_\alpha:=((\Psi_{G}(\alpha,\,\cdot\,))^{-1})_{\sharp}\mm_\alpha$.

We first show that for $\q$-a.e.\ $\alpha$, the measure $\tilde\mm_\alpha$ satisfies the $\MCP(0,N-1)$ condition.
Suppose the contrary.
Then there exists $\epsilon>0$, such that the set
\begin{align*}
    B:=
    \{ \alpha \in Q \colon \; &
\exists \mu_{0,\alpha}\in \Prob(\R)
,
\exists r_\alpha\in\supp \tilde\mm_\alpha
,
\exists t_\alpha\in (0,1)
\colon
\\
&
\U_{N-1}(\mu_{t_\alpha,\alpha}|\tilde\mm_\alpha)
\leq
(1-t_\alpha)
\U_{N-1}(\mu_{0,\alpha}|\tilde\mm_\alpha)
-\epsilon
\}
\end{align*}
has positive $\q$-measure, where $\mu_{t_\alpha}$ is the Wasserstein 
interpolation at time $t_\alpha$ between $\mu_{0,\alpha}$ and $\delta_{r_\alpha}$.
Then the set
\begin{align*}
    C=
    \Big\{
    &
    (\alpha,t)\in Q\times [0,1]
    \colon
    \U_{N-1}(\mu_{t,\alpha}|\tilde\mm_\alpha)
    \leq
    (1-t)
    \U_{N-1}(\mu_{0,\alpha}|\tilde\mm_\alpha)
    -
    \frac{\epsilon}{2}
    \Big\}
\end{align*}
has positive $\q\otimes\L^1$-measure.
Indeed, by upper semicontinuity of $t\mapsto \U_{N-1}(\mu_{t,\alpha}|\tilde\mm_\alpha)$, 
for any $\alpha\in B$, the slice $C(\alpha)$ is a non-empty open set.
By Fubini's Theorem, there exists $\bar t\in (0,1)$ such that 
the inequality in the definition of $C$ is verified on a subset of $B$, still denoted by $B$, of positive $\q$-measure.
Integrating the said inequality at $\bar t$ over $B$ 
contradicts the $\NC^e(N)$ inequality between 
\[
\mu_0
:=
\frac{1}{\q(B)}
\int_{B}
\mu_{0,\alpha}
\,
\q(\de\alpha)
\qquad
\text{ and }
\qquad
\mu_1
:=
\frac{1}{\q(B)}
\int_{B}
\delta_{r_\alpha}
\,
\q(\de\alpha)
.
\]
Moreover $\OptGeo^G(\mu_0,\mu_1)$ has only one element: indeed, since $\Pi_{\leq}(\mu_0,\mu_1)$ is a singleton, then \Cref{P:G-causal-transport-plan} ensures that there exists a unique $G$-causal dynamical plan from $\mu_0$ to $\mu_1$.

A consequence of the $\MCP(0,N-1)$ condition is that the measures $\tilde\mm_\alpha$ are $\L^1$-absolutely continuous.
Therefore arguing as in~\cite[Sec.~8.2]{CMM24a}, 
one can prove that $\mm_\alpha$ satisfies the $\CD(0,N-1)$ condition.

The proof of the converse implication follows the arguments 
of~\cite[Sec.~7]{CMM24a}.
\end{proof}

\begin{remark}
    As a consequence of the proof, if $\mm$ has full support
    in $H$ then, for $\q$-a.e.\ $\alpha\in Q$, $\mm_\alpha$ 
    has full support in $H_\alpha$.
\end{remark}

\begin{remark}[A Brenier-like result]
    As a consequence of \Cref{T:localization}, one can prove a Brenier-like theorem; i.e, 
    if $(H,G,\mm)$ is a null non-branching 
  synthetic null hypersurface satisfying $\NC^e(N)$, if $\mu_0\ll\mm$ 
    and $\Pi_{\leq}(\mu_0,\mu_1)\neq\emptyset$, then the unique monotone coupling 
    given by \Cref{P:monotone-transport-plan} is induced by a map.
    The proof follows standard arguments of $L^1$-Optimal Transport (see, e.g.,~\cite[Sec.~3.6.1]{Cavalletti17}), once proved that the disintegration of $\mm$ along the transport set satisfies the property stated in the second item of  \Cref{T:localization}.
\end{remark}

\begin{remark}
The second item in \Cref{T:localization} shall be seen as a synthetic counterpart of the $\NC^1(N)$ condition introduced in~\cite[Def.~5.13]{CMM24a}  in the smooth setting. 
Thus  \Cref{T:localization} shall be seen as a synthetic counterpart of the equivalence between the $\NC^e(N)$ and $\NC^1(N)$ conditions, proved in the smooth framework in~\cite[Sec.~8]{CMM24a}.
\end{remark}

\section{Synthetic Hawking's Area Theorem}

In this section, as an application of the theory
developed above, we provide a new formulation of Hawking's Area Theorem, where the assumption that the Ricci
curvature is non-negative on null vectors is replaced by the
$\NC^e$ condition.

Firstly, we give a suitable
definition for the area of a cross-section that is compatible with the smooth one.  It will be convenient to consider also its local version; 
for this purpose,  recall the notation 
\[\relation (A)=(J^-(A)\cup J^+(A))\cap H,\]
for any subset $A\subset H$ of the  null non-branching synthetic null hypersurface  $(H,G,\mm)$.

\begin{definition}
  Let $(H,G,\mm)$ be a null non-branching synthetic null hypersurface.
  Let $S\subset {H}$ be a set.
We define the (future) $\epsilon$-enlargement of $S$ by:
\begin{equation}
    S^{+}_\epsilon:=
    \{x\in\noending{H}
    \colon\gflow_{G}(x,-r)\in S
    ,
    r\in(0,\epsilon)
    \}.
\end{equation}
The (future) Minkowski content of $S$ is defined by
\begin{align}
    &
      \mm_{G}^{+}(S)
    :=
    \limsup_{\epsilon\to 0^{+}}
    \frac{
      \mm(S^+_{\epsilon})
      }{\epsilon}
      .
\end{align}
Let $A\subset H$. The relative (future) Minkowski content of $S$ with respect to $A$ is  defined by 
\begin{equation}\label{eq:minkowski-relativo}
  \mm^{+}_{G}(S;A)
  :=
  \limsup_{\epsilon\to 0^{+}}
  \frac{
    \mm(S^+_{\epsilon}\cap \relation (A))
  }{\epsilon}
  .
\end{equation}
\end{definition}

The definition below  shall be read as a synthetic counterpart
of the smooth notion of a 
null hypersurface being future geodesically complete.

\begin{definition}
Let $H$ be a 
  closed achronal set admitting a gauge function $G$.
  We say that $H$ is \emph{future complete for $G$},  if for all  $z\in \noending{H}$
\begin{equation}\label{eq:defHFCG-nononb}
     \sup_{y\in J^{+}(z)\cap \noending{H} }G(y)=\infty.
\end{equation}
\end{definition}

Notice that if $H$ is null non-branching, than future completeness 
is equivalent to requiring that for all  $z\in \noending{H}$
\begin{equation}\label{eq:defHFCG}
     \sup_{y\in \noending{H}_{z}}G(y)=\infty 
     ;
\end{equation}
   or, equivalently, if  
  for all $z\in \noending{H}$ and
  all $t>0$ it holds that $(z,t)\in\Dom(\Psi_{G})$.

We are now in position to state and prove Hawking's area theorem.

\begin{theorem}\label{eq:HawSynt}
  Let  $(H,G,\mm)$ be a null non-branching, synthetic null hypersurface, satisfying the $\NC^e(N)$
  condition. Assume that $H$ is future complete for $G$ and that $\supp \mm=H$.  
  Let $S_1, S_2\subset \noending{H}$ be two acausal sets with $S_1\subset J^{-}(S_2)$. Then
  \begin{equation}
    \mm^+_{G}(S_1)
    \leq
    \mm^+_{G}(S_2)
    ,
  \end{equation}
  where the value $+\infty$ is admitted.
\end{theorem}
\begin{proof}
  We use \Cref{T:disintegration} and we adopt its notation.
Since $H=\supp \mm$ and $H$ is future 
complete for $G$, then the density $h_\alpha$ of $\mm_\alpha$ is positive on a unbounded interval.
Since by Theorem~\ref{T:localization}  $h_\alpha^{1/(N-2)}$ is concave, we infer that $h_\alpha$ is non-decreasing.
The assumption $S_1\subset J^{-}(S_2)$ implies that
$\QQ((S_1)^{+}_\epsilon)\subset \QQ((S_2)^{+}_\epsilon)$.
Using that  $(S_{i})_\epsilon^+ \subset \noending{H}$
\begin{align*}
\mm((S_2)_\epsilon^+)
&
= \int_{\QQ((S_2)_\epsilon^+)} \mm_\alpha((S_2)_\epsilon^{+})
\,
\q(\de\alpha)
\geq
\int_{\QQ((S_1)_\epsilon^+)} \mm_\alpha((S_2)_\epsilon^{+})
\,
\q(\de\alpha)
\\
&
\geq
\int_{\QQ((S_1)_\epsilon^+)} \mm_\alpha((S_1)_\epsilon^{+})
\,
\q(\de\alpha)
=
\mm((S_1)_\epsilon^{+})
,
\end{align*}
proving the theorem.
\end{proof}

\begin{remark}
Inspecting the proof of \Cref{eq:HawSynt}, one can check that the assumption~\eqref{eq:defHFCG} can be replaced by the following weaker (but slightly more tecnical) condition: for $\qq$-a.e. $\alpha$, the ray $H_\alpha$  (see \Cref{T:localization}) is isometric to a real interval  unbounded from above. 
\end{remark}

A priori,  the Minkowski content of $S$ could depend on the choice
of $\mm$ and $G$.
In the next proposition we show that, under suitable covariance
relations linking $G$ with $\mm$, the Minkowski content is independent from such choices.

\begin{proposition}\label{P:covariance-area}
  Let $(X,\ll,\leq,\mathfrak{T})$ be a topological causal space and let $H\subset X$ be
  a null non-branching closed achronal set.
  Let $\mm_{i}\in\M^{+}(H)$ be a measure and $G_{i}$ be a gauge on $H$, $i=1,2$.
  Let $S\subset H$ be an acausal set.
  Assume that $G_1=f+hG_2$, for some transverse
  functions $h,f$ (cf.\ \Cref{P:characterization-equivalence})
 and that $\mm_1=\frac{1}{h}\mm_2$.
  Assume that $(H,G_1,\mm_1)$ satisfies the $\NC^e(N)$
  condition (hence $(H,G_2,\mm_2)$ does, as well, see \Cref{R:invariance-NC-non-branching}).
  Assume that $H$ is future complete for $G_1$ (hence for $G_2$).
  
  If
  \[
  (\mm_1)^{+}_{G_1}(S)<\infty
  \qquad
  \text{ and }
  \qquad
  (\mm_2)^{+}_{G_2}(S)<\infty
  ,
  \]
  then
  \begin{equation}
    (\mm_1)^+_{G_1}(S)
    =
    (\mm_2)^+_{G_2}(S).
  \end{equation}
\end{proposition}
\begin{proof}
  Let $\mm_{i,\alpha}$ be the conditional measures given by the
  disintegration of $\mm_{i}$ and let $h_{i,\alpha}$, be the density of
  $\mm_{i}$, $i=1,2$, see~\Cref{T:disintegration}.
  Since by assumption $\mm_2=h\mm_1$, and $G_1=f + h G_2$,
  it is immediate to see that
  \begin{equation}
    \label{eq:transformation-h}
    h_{1,\alpha}(t)=h_{2,\alpha}(f(\alpha)+h(\alpha)\,t)
  \end{equation}
  Denote by $S^{+}_{i,\epsilon}$ the $\epsilon$ enlargement w.r.t.\ the
  gauge $G_{i}$, $i=1,2$.
  Let $A:=\QQ(S_{1,\epsilon}^{+})=\QQ(S_{2,\epsilon}^{+})$.
  For every $\alpha\in A$, let $a_{i,\alpha}\in\R$ be such that
  $(a_{i,\alpha},a_{i,\alpha}+\epsilon)$ is the maximal interval
  satisfying
  $\gflow_{G_{i}}((a_{i,\alpha},a_{i,\alpha}+\epsilon))\in S^{+}_{i,\epsilon}$,
  $i=1,2$.
  It holds that
  \begin{equation}
    \label{eq:a-affine}
    a_{1,\alpha}=f(\alpha)+h(\alpha) a_{2,\alpha}
    .
  \end{equation}
  We compute
  \begin{align*}
    \mm^{+}_{G_{i}}(S)
    &
    =
      \limsup_{\epsilon\to 0^{+}}
      \frac{\mm(S_{i,\epsilon}^{+})}{\epsilon}
      =
      \limsup_{\epsilon\to 0^{+}}
      \int_{A} \frac{\mm_\alpha(S_{i,\epsilon}^{+})}{\epsilon}
\,
      \q(\de\alpha)
    \\
    &
      =
            \limsup_{\epsilon\to 0^{+}}
\int_{A}
\int_{a_{i,\alpha}}^{a_{i,\alpha}+\epsilon}
\frac{h_{i,\alpha}(t)}{\epsilon}
\,
\de t
\,
      \q(\de\alpha)
      .
  \end{align*}
  Denote by $v_{i,\alpha}$ the function
  \begin{equation}
    v_{i,\alpha}(t)
    :=
    \int_{a_{i,\alpha}}
    h_{i,\alpha}(s)
    \,\de s
    .
  \end{equation}
  The $\NC^e(N)$ condition implies that $h_{i,\alpha}^{\frac{1}{N-2}}$
  is concave, thus the future completeness w.r.t.\ $G_{i}$ guarantees
  that $h_{i,\alpha}$ is non-decreasing, therefore $v_{i,\alpha}$ is
  convex.
  It follows that the limit $\lim_{\epsilon\to
    0^{+}}\frac{v_{i,\alpha}(\epsilon)}{\epsilon}$ exists and it is
  monotone decreasing.
  The hypothesis $\mm^{+}_{G_{i}}(S)<\infty$, guarantees that for
  $\epsilon>0$ small enough,
  $\frac{v_{i,\alpha}(\epsilon)}{\epsilon}\in L^1(\q)$.
  We are in position to apply the dominated convergence theorem,
  deducing
  \begin{equation}
    \mm^{+}_{G_{i}}(S)
    =
    \int_{A}
    h_{i,\alpha}(a_{i,\alpha})
    \,
    \q(\de\alpha)
    .
  \end{equation}
  We conclude by combining~\eqref{eq:transformation-h} and~\eqref{eq:a-affine}.
\end{proof}

\section{Penrose's singularity theorem for synthetic and continuous spacetimes}\label{S:Penrose}

In this section, we propose a way to extend Penrose's singularity
theorem to the setting of a continuous spacetime (i.e., when the
Lorentzian metric is continuous).
Recall that Penrose's singularity theorem~\cite{Penrose65} states that if a spacetime satisfyng the \NEC\ admits a non-compact Cauchy hypersurface and it contains a trapped surface, then it is null geodesic incomplete; i.e., there exists a null geodesic $\gamma$ (in the sense that it solves the geodesics-{ODE}: $\nabla_{\dot{\gamma}}\dot{\gamma}$) whose maximal domain of definition is strictly contained in $\R$.
Besides the proof, a first challenge in generalizing such a result for a continuous spacetime is to make sense of the very statement. We first start with the concept of null completeness. 

\subsection{Weak null completeness}\label{SS:WNC}

In the smooth setting, it is known that  the incomplete geodesic yielded by Penrose's singularity theorem is achronal (see~\cite{Penrose65, Penrose-DiffTopGR},~\cite[Sec.~12.4]{BEE-Book},~\cite[Th.~2.67]{Min}).
Negating the existence of such a geodesic therefore amounts to requiring that every maximally extended geodesic is either complete or fails to remain achronal (or, equivalently, fails to be maximizing).
We adopt this condition as the definition of \emph{weak null completeness}, as it admits an equivalent synthetic characterization, which we establish below.

\begin{definition}[Weak null completeness]\label{defn:wgnc}
  Let $(M,g)$ be a strongly causal Lorentzian manifold.
We say that $(M,g)$ is \emph{weakly future  null complete} if for any
  maximally-defined (on the right) null geodesic (in the differential
  sense) $\gamma:[0,b)\to M$, either $b=\infty$, or
  $\gamma_0\ll\gamma_{t}$, for some $t\in[0,b)$. The notion of \emph{weakly past null  complete} is analogous, by considering maximally defined (on the left) null geodesics. We say that $(M,g)$ is \emph{weakly  null complete} if it is both weakly future and past  null complete.
\end{definition}

It is clear that  a null complete manifold is weakly null complete as the case of $b<\infty$ never occur (and therefore being weakly null incomplete will be a stronger property of being null incomplete). 
On the other hand, we now exhibit an example of a weakly null complete Lorentzian manifold that is not null complete justifying the terminology. 

In particular, we now provide a sufficient condition for a
warped product to be weakly null complete but not
null complete.  This amounts to show that all the incomplete inextendible null geodesics are not globally maximizing. The geometric idea is, roughly, of an upside down cone-like shape where all the causal geodesic end in the tip and all the null geodesics move in spiral, more and more densely as they approach the tip.

  \begin{proposition}\label{Prop:ExNotComplete}
    Let $(F,g_F)$ be a compact Riemaniann manifold and consider the
    warped product $M:=(-\infty,0)\times_f F$, endowed with the metric
    \begin{equation}
      g= \de t^2 - f^2(t)\, g_{F}.
    \end{equation}
    Assume that $f(t)\leq O(\sqrt{-t})$ as $t\to 0^-$.

    Then $M$ is a weakly null complete, but 
    both null and time-like incomplete. More precisely, every future directed geodesic (either timelike or null) is future incomplete, and every future directed null geodesic is not globally maximizing. 
  \end{proposition}

  \begin{proof}
    In this proof, we denote by $\gamma:[0,b)\to M$ an arbitrary maximally-defined
    future-directed geodesic in $M$, with $\gamma_{s}=(t_{s},x_{s})$.
  We will write $(\dot{t}, \dot{x})= \dot\gamma := \frac{\de}{\de s}\gamma$ and $\nabla_{\dot x_{s}} \dot{x}_{s}$ for the covariant derivative in $(F,g_{F})$ of $\dot{x}$ with respect to itself, i.e., the acceleration of the curve $s\mapsto x_{s}\in F$. Also, we  denote the derivative of the warping function $f:(-\infty, 0)\to \R$ by  $f':= \frac{\de}{\de t} f$.
    The equations governing the evolution of $\gamma$
    are~\cite[Prop.~7.38,~p.~208]{O'Neill}
    \begin{align}
      &
        \label{eq:acceleration-base}
        \ddot t_{s}
        =
        -
        g_{F}(\dot x_{s},\dot x_{s})
        \,f(t_{s})
        \,
        f'(t_{s})
        ,
      \\
      &
        \label{eq:acceleration-fiber}
        \nabla_{\dot x_{s}} \dot{x}_{s}
        =
        -2
        \,
        \frac{f'(t_{s})}{f(t_{s})}
        \,
        \dot t_{s}
        \,\dot x_{s}
        =
        -2
        \,
        \frac{\de \log (f \circ t_{s})}{\de s}
        \,\dot x_{s}
        .
    \end{align}

    Notice that $|\dot\gamma|^2=
    |\dot{t}_{s}|^2-f^2(t_{s})g_{F}(\dot x_{s},\dot x_{s})$ is constant in $s$, since $\gamma$ is a geodesic.   We can rewrite~\eqref{eq:acceleration-base} as
    \begin{equation}
              \ddot t_{s}
              =
              -
        (|\dot t_{s}|^2-|\dot{\gamma}|^2)
        \,
        \frac{f'(t_{s})
        }{
          f(t_{s})
        }
        .
      \end{equation}
      If $\dot x_{s}\equiv 0$ then, from~\eqref{eq:acceleration-base}, we get that  $t$ is affine and thus $b$ must be
      finite.
      \\ If instead $\dot x_{s}\not \equiv 0$, then
      \begin{equation*}
        \frac{
          \dot t_{s}
          \,
          \ddot t_{s}
        }{
          |\dot t_{s}|^2-|\dot{\gamma}|^2
        }
        =
        \frac{1}{2}\,
        \frac{
          \de \log(|\dot t_{s}|^2-|\dot{\gamma}|^2)
        }{
          \de s
        }
        =
        -
        \frac{
          f'(t_{s})
          \,
          \dot t_{s}
        }{
          f(t_{s})
        }
        =
        -
        \frac{
          \de \log(f( t_{s}))
        }{
          \de s
        }
        .
        \end{equation*}
    By integrating, we deduce that
    \begin{equation*}
      \frac{|\dot t_{s_2}|^2-|\dot{\gamma}|^2}{|\dot t_{s_1}|^2-|\dot{\gamma}|^2}
      =
      \frac{f^2(t_{s_1})}{f^2(t_{s_2})}
      ,
      \qquad
      \forall s_1, s_2\in [0,b), s_1\leq s_2
      ,
    \end{equation*}
    therefore
    \begin{equation}
      \dot t_{s_2}
      =
      \sqrt{
        |\dot{\gamma}|^2
        +
        (|\dot t_{s_1}|^2 - |\dot{\gamma}|^2)
      \frac{f^2(t_{s_1})}{f^2(t_{s_2})}
    }
    .
  \end{equation}
  We can now take the limit on the r.h.s.\ as $s_2\to b^{-}$, deducing
  that $\lim_{s\to b^{-}} \dot t_{s}=\infty$, therefore $b$ must be
  finite.
  This proves future incompleteness.

  A consequence of the limits $\lim_{s\to b^{-}} t_{s}=0$
  and
  $\lim_{s\to b^{-}} \dot t_{s}=\infty$
  is that
  \begin{align*}
    t_{s}
    =
    \int_{b}^s\dot t_{w} \,\de w
    \geq
    s-b
    ,
    \qquad
    \text{ for $s$ large enough,}
  \end{align*}
  therefore
    \begin{equation}
      \label{eq:expansion-f}
      f(t_{s})
      \leq
      O(\sqrt{b-s})
      .
    \end{equation}

    We now show that any null geodesic starting at a given
    point $p=(t_0,x_0)$ will definitely enter in $I^+(p)$.
    We claim that for any $z\in M$,
    \begin{equation}
      \label{eq:warped-future}
      J^{+}(p)\cup (-\infty,0) \times \{z\}
      \neq
      \emptyset
      .
    \end{equation}
    Suppose on the contrary that there exists $z$ contradicting~\eqref{eq:warped-future}.
    Let $L:=(\exp^{g_{F}}_{x_0})^{-1}(z)$ and
    $\dot t_0:=f(t)\sqrt{g_{F}(L,L)}$, and consider the null geodesic
    $\gamma$ starting from $p$, with initial velocity
    $\dot\gamma_0=(\dot t_0,L)$.
    We will obtain a contradiction by showing that the projection of $\gamma$ on $F$ passes through $z$.
    Equation~\eqref{eq:acceleration-fiber} states that the geodesic
    acceleration of $x$ is parallel to its velocity, thus $x$ is a
    reparametrization of a geodesic $y$ in $F$, i.e., $x_{s}=y_{r(s)}$.
    Since $F$ is compact, then Hopf{--}Rinow theorem guarantees it is complete. Then $y_{r}$ is defined for all positive $r$
    and $y$ accumulates at some point in $F$.
    A simple computation gives
    \begin{align*}
       \nabla_{\dot{x}_{s}} \dot{x}_{s}
      =
      \nabla_{\dot r(s)\dot y_{r(s)}}
      (\dot r(s)\dot y_{r(s)})
      =
      \ddot r(s) \dot y_{r(s)}
      =
      \frac{\ddot r(s)}{\dot r(s)} \dot x_{s}
      ,
    \end{align*}
   which, combined with~\eqref{eq:acceleration-fiber}, yields
    \begin{equation*}
      \frac{\de \log \dot r}{\de s}
      =
      \frac{\ddot r}{
        \dot r}
      =
      -2
      \frac{\de \log (f \circ t)}{\de s},
    \end{equation*}
    or, equivalently:
    \begin{equation*}
      \dot r(s)=\frac{C}{f^2(t_{s})},
    \end{equation*}
    for some constant $C$. Integrating, we get  that
    \begin{equation*}
      r(s)=\int_0^s\frac{C}{f^2(t_{w})}
      \,\de w
      .
    \end{equation*}
    Plugging the expansion~\eqref{eq:expansion-f} in the equation
    above, we deduce that $r(s)\to+\infty$, as $s\to b^-$.
    This means that $x$ runs on the whole trajectory of $y$ and therefore the curve
    $x$ passes through the point $z$.

    Having proved~\eqref{eq:warped-future}, we deduce that
    $\partial I^{+}(p)$ intersect each line $(-\infty, 0)\times \{z\}$ in
    only one point,  for each fixed $z\in F$.
    Moreover, it is well-known
    (see~\cite[Cor.~14.27,~p.~415]{O'Neill}) that $\partial I^{+}(p)$
    is a $C^{0}$ manifold, thus $\partial I^{+}(p)$ can be seen as a
    graph hypersurface over $F$ (in particular, it is a Cauchy hypersurface).
    Since  $F$ is compact and $\partial I^{+}(p)$ is a graph
    surface, then there exists $t_1\in (t_0, 0)$ such that
    \[
      \partial I^{+}(p) \subset (t_0,t_1) \times F.  
    \]
    It follows that any maximal null geodesic
    starting from $p$ will definitely enter $I^{+}(p)$, thus failing to be maximizing.
  \end{proof}

The weak null completeness condition can be rephrased using gauges. 
We recall that if $M$ is a smooth Lorentzian manifold
and $H=\partial I^{+}(A)$, with $A$ a compact acausal submanifold, 
a \emph{natural} gauge function on $H$ is 
 given by pushing via the exponential a fiber-wise affine function
  in the normal bundle $\normal\initial$
  (see~\Cref{R:compatibility-penrose}).
  In particular $G$ can be extended to $0$ along
$\initial$, continuously on each generator.
 Recall also that  $G$ is called \emph{proper}
if the pre-image of every compact set is precompact. 

We will now show that if $M$ is a strongly causal Lorentzian manifold,
the weak null completeness of $M$ is equivalent to the existence of natural proper gauges, with the latter having the great advantage of being formulated without any need of an underlying manifold structure. 

Thus, \emph{the properness of the gauge shall be seen as a synthetic
    counterpart of the weak null completeness}.
The statement of the next proposition is given for weak future null completeness; 
one can state a similar theorem for past completeness.

\begin{proposition}\label{prop:wgnc}
    Let $(M,g)$ be a strongly causal, causally closed
    Lorentzian manifold.
    %
    Let $H=\partial I^{+}(A)$, 
    for some $C^2$ compact space-like achronal submanifold $A$.
    %
    \begin{enumerate}
    \item\label{WGNC-to-GProper}
      Let $G$ be any natural gauge for $H$ which is continuously extendible to $0$ at $A$
      (see \Cref{R:compatibility-penrose}).
       If $(M,g)$ is weakly future  null complete, then $G$ is proper.

    \item\label{GProper-to-WGNCH}
      Conversely, if $H$ admits a proper natural gauge $G$,
      then for every
      maximally (on the right) defined null geodesic
      $\gamma:[0,b)\to M$,
      the following dichotomy holds:
      either $b=\infty$ or $\gamma$ will leave $H$ definitely.
    \item\label{GProperCones-to-WGNCH}   Finally, if every future light-cone of $M$ admits a  proper natural gauge,
    then $(M,g)$ is weakly future null future complete.
    \end{enumerate}
\end{proposition}
  \begin{proof}
    \textbf{Proof of~\ref{WGNC-to-GProper}}.
    We recall from \Cref{R:compatibility-penrose} that 
    \begin{equation}
      \noending{H}
      \subset
      \exp(
      \{
      v\in\normal A
      \colon
      g(v,v)=0
      \text{ and }
      g(v,W)<0
      \}
      )
      ,
    \end{equation}
    where $W$ is a time-like vector field given by the
    time-orientation of $M$ (notice that here, differently form
    \Cref{R:compatibility-penrose}, we are considering the larger set
    $A\supset\initial$).
    If we define
    \begin{equation*}
      C:=
      \{v\in \normal A
      \cap \dom(\exp)
      \colon
      g(v,v)=0
      ,
      \;
      g(v,W)\leq
      0
      ,
      \text{ and }
      \exp(v)\not\in I^{+}(A)
      \}
      ,
    \end{equation*}
    it is then clear that $H=\exp(C)$.
    By definition of natural gauge, it holds that
    \begin{equation*}
        G^{-1}([a,b]) 
      \subset
      \exp(
      D)
      ,
      \quad
      \text{where}
      \quad
      D:=
      \{
      v\in C
      \colon
      -g(v,W)\in [a,b]
      \}
      \subset \normal A
      .
    \end{equation*}
    We check that $D$ is compact; by continuity of the exponential the
    thesis will follow.
    Boundedness of $D$ is immediate.
    Regarding closedness
, let $v_{n}$ be a sequence in
    $D$, converging to $v_\infty\in\normal A$.
    Assume by contradiction that $v_\infty\not\in D$.
    Then $v_\infty\not\in \dom(\exp)$, since the other conditions
    defining $D$ are closed.
    It follows that the curve $\gamma_{t}:=\exp(tv_\infty)$ is defined
    on $[0,b)$ with $b \leq 1$.
    We exploit the weak null completeness to infer that,
    for some $t<b$,
    $\gamma_0\ll\gamma_{t}\in I^{+}(\gamma_0)\subset I^{+}(A)$.
    Since $\exp(tv_{n}) \overset{n\to \infty}\longrightarrow \gamma_{t}$,
    then, for $n$ large enough,
    $\exp(tv_{n})\in I^{+}(A)$.
    Therefore, $\exp(v_{n})\in I^{+}(A)$, yielding a
    contradiction with the definition of $D$.

\smallskip

 \textbf{Proof of~\ref{GProper-to-WGNCH}}.
 Fix $\gamma:[0,b)\to M$ an inextendible geodesic.
 Up to a reparametrization, we can assume $G(\gamma_{t})=t+G(\gamma_0)$.
    Suppose by contradiction that $b<\infty$ and that $\gamma$ is
    contained in $H$.
    Let $F:=\overline{G^{-1}([0,b+G(\gamma_0)])}\subset H$.
    By the properness assumption on $G$, we have that $F$
    is compact.
    Let $t_{n}\uparrow b$.
    Up to a subsequence, $\gamma_{t_{n}}\in F$ is
    converging in $F$.
    By standard properties of the geodesic flow,
    it follows that  $\exp_{\gamma_0}(bL)$ is well-defined,
    a contradiction with
    the inextendability  of $\gamma$.
\smallskip

\textbf{Proof of~\ref{GProperCones-to-WGNCH}}.     Fix $\gamma:[0,b)\to M$ a future inextendable null
    geodesic.
    Consider $H=\partial I^{+}(\gamma_0)$ the future light-cone emanating from
    $\gamma_0$.
    Applying part~\ref{GProper-to-WGNCH}, yields  that either $b=\infty$ or
    $\gamma$ leaves $H$.
    In the latter case,  there exists $t\in (0,b)$ such that
    $\gamma_0\ll\gamma_{t}$.
\end{proof}

\subsection{Penrose's singularity theorem}\label{Ss:penrose}
Having thoroughly discussed the completeness issue, a possible way to state 
the desired result on the existence of trapped sets is the following: 
if $S$ is a compact
achronal surface, such that $(\partial I^+(S),G,\mm)$ satisfies the
$\NC^e(N)$ with a proper gauge
and $S$ is $(G,\mm)$-future converging, then $\partial I^{+}(H)$ is compact.

The $\NC^e(N)$ has already been defined.  It remains to discuss the
``$(G,\mm)$-future converging'' condition.
This is a delicate point, firstly because of the non-smooth synthetic setting, secondly because 
of the presence  of initial points of $H$ inside $S$, where the parametrization given by the
gauge is not available.
\\
We next define the ``$(G,\mm)$-future converging'' condition, followed by a discussion of its compatibility with the smooth notions.

\begin{definition}\label{def:future-converging}
Let $(H,G,\mm)$ be a synthetic null
hypersurface satisfying the $\NC^e(N)$ condition,
with $H=\partial I^+(S)$, for some achronal set $S$.

We say that $S$ is \emph{$(G,\mm)$-future converging},
if there exists $\theta<0$, such that
\begin{equation}\label{eq:mean-curvature}
   \limsup_{\epsilon \to 0^{+}}
\frac{\mm(S_\epsilon^{+}\cap \relation(A))-\epsilon\, \mm^{+}_{G}(S;A)
    }{
      \epsilon^2/2}
    \leq\theta
    \,
    \mm^{+}_{G}(S;A)
    ,
  \end{equation}  
for all $A\subset \noending{H}$.
\end{definition}  
 Recall that the classical future convergence requires the mean curvature of
  $S$ on both (incoming and outgoing) sides of $\partial I^{+}(S)$ to be strictly negative.
  If $S$ is smooth and compact, this is equivalent to assume that its mean
  curvature (in both the incoming and outgoing directions) is bounded above by some $\theta<0$.
  Equation~\eqref{eq:mean-curvature} is then an integral
  characterization of this curvature bound. The basic idea is that the mean curvature is given by first variation of area, and area is in turn given as first variation of volume. Notice indeed that~\eqref{eq:mean-curvature} corresponds to a second order Taylor expansion of the volume (cf.~\cite[Remark 5.4]{CaMo:20}, for a similar construction in the timelike framework). 

  A careful reader may wonder about the necessity of the set  $A\subset \noending{H}$ in Definition~\ref{def:future-converging} and whether one can simplify the definition by stating 
  everything just in terms of $S$. The reason for including $A\subset \noending{H}$ in Definition~\ref{def:future-converging} is that, in general, there is more than one generator for $H$ leaving at
  each point of $S$ (in the classical setting, exactly two generators) and it is desirable to have a criterion to select exaclty one.
  Restricting the analysis to $\relation(A)$, for any $A\subset \noending{H}$, allows to overcome this issue.
\smallskip

In the statement of the synthetic Penrose's theorem,
we will assume the gauge to satisfy the additional condition
  \begin{equation}
    \label{eq:infimum-of-gauge}
    \inf_{x\in H_\alpha}
    G(x)
    =
    0  
    ,
    \qquad
    \text{ for $\q$-a.e.\ } \alpha\in Q
    ,
  \end{equation}
where $Q$ is the quotient set of \Cref{T:disintegration}. In the smooth
setting, one can always find a gauge satisfying~\eqref{eq:infimum-of-gauge}, as discussed in \Cref{R:compatibility-penrose}.

Moreover we recall the classical terminology that a compact achronal set $S$ is \emph{future trapped} if $\partial I^+(S)$ is compact.

\begin{theorem}[Penrose's theorem in a synthetic setting]\label{T:penrose}
Let $(H,G,\mm)$ be a null non-branching, synthetic null 
hypersurface satisfying the $\NC^e(N)$ condition for some $N> 2$,
where $H=\partial I^+(S)$ for some compact achronal set $S$, and  $\supp\mm=H$.
Assume moreover~\eqref{eq:infimum-of-gauge}.

  If $S$ is $(G,\mm)$-future converging,
  then $\norm{G}_{L^\infty(\mm)}<\infty$.
  In particular, if the gauge $G$ is proper, then   $H$ is compact; i.e., $S$ is future trapped.
\end{theorem}

\begin{proof}
We divide the proof in two parts.

{\bf Step 1.}
By the Localization \Cref{T:localization}, there exists a disintegration of $\mm$, 
given by the measures $\mm_\alpha$ satisfying the $\CD(0,N-1)$ condition.
Let $h_\alpha:[0,b_\alpha)\to{[0,\infty)}$ be the density of $\mm_\alpha$.
For the class of synthetic null hypersurfaces 
fulfilling~\eqref{eq:infimum-of-gauge}
the relative Minkowski content can be more easily described.
In particular, it will be convenient to look for a class of sets $A \subset \noending{H}$ for which the following identity is valid: 
\begin{equation}\label{E:meancurvature}
   \int_{B} h_\alpha(0)\,\qq(d\alpha) 
   =
   \mm^{+}_{G}(S;A),  
\end{equation}
where $B = \relation(A) \cap Q$. 
To establish~\eqref{E:meancurvature}, notice that 
\[
    \limsup_{\epsilon\to 0^{+}}
    \frac{
      \mm(S^+_{\epsilon}\cap \relation (A))
    }{\epsilon}
= 
  \limsup_{\epsilon\to 0^{+}}
\frac{1}{\epsilon} 
\int_{B}\int_0^\epsilon h_\alpha(t)\, \de t.
\]
Since $t\mapsto h_\alpha(t)$ is continuous, we can consider the 
real-valued measurable map 
$Q \ni \alpha \mapsto \| h_\alpha(\cdot)\|_{L^\infty(0,1)}$. 
For each $\delta >0$ there exists a compact set $K\subset Q$
such that $\qq(K)> 1- \delta$ and, for all $\alpha \in K$, 
it holds that
$\| h_\alpha(\cdot)\|_{L^\infty(0,1)} \leq C$, 
for some positive constant. 
 A direct application of dominated convergence theorem implies that for all $A\subset \relation(K)\cap \noending{H}$,
the identity in \Cref{E:meancurvature} is valid. 

In the next step, we will use \Cref{E:meancurvature} in order to obtain an upper bound on $b_\alpha$ independent on $\delta>0$; this will imply  that $\|G\|_{L^\infty}<\infty$.

{\bf Step 2.}
The concavity of the function $t\mapsto h_\alpha(t)^{\frac{1}{N-2}}$ implies that
\begin{equation}\label{eq:mcp-penrose}
    h_\alpha(t)
    \geq
h_\alpha(0)
    \bigg(
    \frac{b_\alpha-t}
    {b_\alpha }
    \bigg)^{N-2}
    ,
    \qquad
    \forall t\in[0,b_\alpha]
    .
\end{equation}
If  $N\geq 3$, Bernoulli's inequality
(i.e., $x^{N-2} \geq 1 + (N-2)(x-1)$,  for all $x\geq 0$) yields
\begin{equation}\label{eq:expansion-convexity}
  \begin{aligned}
  \int_{0}^{\epsilon}
\bigg(
      \bigg(
    \frac{b_\alpha-t}
    {b_\alpha }
    \bigg)^{N-2}
-
1
\bigg)
  \,\de t
    \geq
    \frac{2-N}{b_\alpha}
\int_{0}^
{\varepsilon}
    t
      \,\de t
    =
    \epsilon^2
    \frac{2-N}{2b_\alpha}
    ,
    \qquad
    \forall
    \epsilon >0
    .
  \end{aligned}
\end{equation}
If $2<N< 3$, the trivial inequality $x^{N-2} \geq  x$, holding for all $x\in [0,1]$, implies that
\begin{equation}\label{eq:expansion-convexityN-less-3}
  \begin{aligned}
  \int_{0}^{\epsilon}
\bigg(
      \bigg(
    \frac{b_\alpha-t}
    {b_\alpha }
    \bigg)^{N-2}
-
1
\bigg)
  \,\de t
    \geq
    {b_\alpha}
\int_{0}^
{\varepsilon}
    t
      \,\de t
    =
    \frac{\epsilon^2}
    {2b_\alpha}
    ,
    \qquad
    \forall
    \epsilon >0
    .
  \end{aligned}
\end{equation}
In the rest of the proof, we will assume that $N\geq 3$ and we will use~\eqref{eq:expansion-convexity}; in case $2<N<3$, the argument is completely analogous, using~\eqref{eq:expansion-convexityN-less-3}.

Let $\theta<0$ be given by \Cref{def:future-converging}.
Consider any $\delta>0$
and the let $K\subset Q$ be 
given by the step 1 of the proof.
Let
$A\subset \relation(K)\cap \noending{H}$ be any measurable set where~\eqref{E:meancurvature} holds, with  $B=\relation(A)\cap Q$.
Then~\eqref{eq:mean-curvature} can written as
\begin{align*}
   \theta
  \int_{B}
  h_\alpha(0)
  \,
  \q(\de\alpha)
\overset{\eqref{eq:mean-curvature}}&{\geq}
     \limsup_{\epsilon \to 0^{+}}
  \frac{2}{\epsilon^2}
  \int_{B}
  \int_{0}^{\epsilon}
  (h_\alpha(t)-h_\alpha(0))\,\de t
  \,
    \q(\de\alpha)
  \\
  \overset{\eqref{eq:mcp-penrose}}&{\geq}
       \limsup_{\epsilon \to 0^{+}}
  \frac{2}{\epsilon^2}
  \int_{B}
  h_\alpha(0)
  \int_{0}^{\epsilon}
      \bigg(
      \bigg(
    \frac{b_\alpha-t}{b_\alpha}
    \bigg)^{N-2}
-
1
\bigg)
  \,\de t
  \,
  \q(\de\alpha) 
    \\
  \overset{\eqref{eq:expansion-convexity}}&{\geq}
  \int_{B}
  \,
  h_\alpha(0)\frac{2-N}{b_\alpha}
  \,
  \q(\de\alpha)
  .
\end{align*}
The arbitrariness of $A\subset \relation(K) \cap \noending{H}$ implies the 
arbitrariness of $B\subset K$, and thus 
\begin{equation}\label{eq:balphaleqN-2Theta}
b_\alpha \leq -\frac{N-2}{\theta}
,
\qquad
\text{ for $\q$-a.e.\ }\alpha\in K
.
\end{equation}
Since the estimate is independent of $\delta >0$, taking the limit as $\delta \to 0$ yields that the inequality in~\eqref{eq:balphaleqN-2Theta} is valid for $\qq$-a.e.\;$\alpha \in Q$.
Combining this with~\eqref{eq:infimum-of-gauge}, we deduce 
that
\[
G(x)\in\left[0,\frac{2-N}{\theta}  \right]
, \quad
\text{for $\mm$-a.e.\ $x\in \noending{H}$}.
\]
The properness of $G$ implies that $G^{-1}([0,\frac{2-N}{\theta} ])$ is precompact. 
 Therefore $\mm$ is concentrated on a precompact set, yielding that
  $H=\supp\mm$ is compact.
\end{proof}



\begin{remark}\label{R:non-constant-infimum}
Inspecting the proof of \Cref{T:penrose}, one can check that the assumption~\eqref{eq:infimum-of-gauge} can be replaced by the following weaker (but slightly more technical) condition: 
the function 
    $
    \alpha\mapsto \mathcal{G}(\alpha):=  \inf_{x\in H_\alpha} G(x)
    $ belongs to $L^\infty(\qq)$.
    In this case, the proof gives that 
    $G\leq \frac{2-N}{\theta}+\|\mathcal{G}\|_{L^\infty(\qq)}$, $\mm$-a.e.. The same argument as above yields that, if $G$ is proper, then $H$ is compact.
\end{remark}

  \begin{remark}
    In the spirit of \Cref{P:covariance-area}, we discuss how
    \Cref{T:penrose} behaves under a different choice of  gauge and
     reference measure.
    Namely, let $H=\partial I^{+}(S)$ and let $G_{i}$ be a gauge for $H$
    and $\mm_{i}$ be a measure, $i=1,2$.
    Assume that $G_1$ and $G_2$ (resp.\ $\mm_1$ and $\mm_2$) are related by the covariance relation   (cf.\ \Cref{P:characterization-equivalence})
    \[
      G_1=f+hG_2, \quad \mm_1=\frac{1}{h}\mm_2,
    \]
    for some  transverse functions $h$ and $f$, 
    \begin{equation}\label{eq:boundfhC}
    -C\leq f \leq C, \quad 1/C \leq h\leq C,
    \end{equation}
    for some constant $C\geq 1$.
 Let us discuss the invariance of the assumptions of \Cref{T:penrose}, under such transformations.
    \begin{itemize}
    \item The assumption that the reference measure has full support is trivially preserved (i.e.,
      $\supp \mm_1=H\iff\supp \mm_2=H$).
    \item Eq.~\eqref{eq:infimum-of-gauge} is not  preserved exactly; however,
      one can apply \Cref{R:non-constant-infimum}, noting that
        the condition $\norm{\mathcal{G}}_{L^{\infty}(\q)}<\infty$ is preserved.
    \item Regarding the  future converging assumption, following the argument of
      \Cref{P:covariance-area} one can prove that $S$ is
      $(G_1,\mm_1)$-future converging if and only if it is
      $(G_2,\mm_2)$-future converging.
      In this case, the constant $\theta$ is replaced by
      $\theta/C$.
    \end{itemize}
    Notice that also the thesis of \Cref{P:covariance-area} is covariant, i.e., properness of the
    gauge does not depend on the choice of the gauge but only
    on its equivalence class; indeed,
    $G_2^{-1}([-a,a])\subset G_1^{-1}([-C(a+1),C(a+1)])$.

    Finally, let us point out that the bound~\eqref{eq:boundfhC} is necessary, otherwise both the hypothesis and
    the thesis of \Cref{P:covariance-area}  would cease to be invariant.
    
  \end{remark}

      The following corollary follows by combining \Cref{T:penrose}
and the $C^0$-extension
of the classical result by Penrose, stating that there are no future
trapped sets in a spacetime admitting a non-compact Cauchy surface,
as specified in Th.~4.9 of~\cite{Ling}, after~\cite[Th.~2.67]{Min}.
We refer
to~\cite{Chr-Grant, SaC0, Ling, Minguzzi-Review2019} for the basics of causality theory and useful notions for continuous spacetimes.

\begin{corollary}[Penrose's singularity theorem in $C^0$-spacetimes]\label{Cor:PenroseC0}
Let $(M,g)$ be a spacetime endowed with a continuous Lorentzian metric.
Let $S\subset M$ be a compact achronal set. Let $H=\partial I^+(S)$ and endow $H$ with a gauge function $G$ satisfying~\eqref{eq:infimum-of-gauge} and a positive Radon measure $\mm\in \M^+(H)$ with $\supp \mm=H$.
Assume that $H$ is null non-branching%
\footnote{%
The assumption that $H$ is null non-branching is always satisfied
if 
\begin{itemize}
\item $\noending{H}$ is $C^{1,1}$ and $g$ is locally Lipschitz 
(see  \Cref{prop:Lipschitz-nonBranch}), 
\item  $g$ is $C^2$, 
without any further regularity assumption on $H$ (see \Cref{C:null-non-branching-c2}).
\end{itemize}
}
and that:
\begin{enumerate}
\item The synthetic null hypersurface $(H,G,\mm)$ satisfies the
$\NC^e(N)$ condition, for some $N>2$.
\item $(M,g)$ admits a non-compact Cauchy surface.
\item $S$ is $(G,\mm)$-future converging.
\end{enumerate}
 Then the gauge $G$ cannot be proper.
\end{corollary}

\begin{proof}
Assume by contradiction that $(M,g)$  contains a compact achronal set $S\subset M$ such that $H=\partial I^+(S)$ satisfies the conditions (1), (2), (3) and moreover the gauge $G$ is proper. Then, Theorem~\ref{T:penrose} implies that $S$ is future trapped.
Since by assumption $(M,g)$ admits a non-compact Cauchy surface, we obtain a contradiction with~\cite[Thm.\,4.9]{Ling}.
\end{proof}

\begin{remark}
The corollary above is valid also for the more general setting of closed cone structures.
In this case, one needs to assume the existence of a stable Cauchy hypersurface and apply~\cite[Th.~2.67]{Min}, in place of~\cite[Th.~4.9]{Ling}. 
\end{remark}

Finally, we state and prove a version of Penrose's singularity theorem where the incomplete geodesic is \emph{maximizing}. The maximizing property was implicit in Penrose's arguments~\cite{Penrose65,Penrose-DiffTopGR}, for a discussion see for instance~\cite[Sec.\,12.4]{BEE-Book} and~\cite[Th.\,2.67]{Min}. Nevertheless, we include it with a proof for the reader's convenience.

\begin{corollary}[Penrose's  maximizing incompleteness Theorem]\label{cor:SharpenedPenrose}
Let $(M,g)$ be a spacetime endowed with a  $C^2$ Lorentzian metric.
Assume that:
\begin{enumerate}
    \item 
    The \NEC\ holds.
    \item 
    $M$ admits a non-compact Cauchy hypersurface.
    \item 
    There exists a future-converging compact $(n-2)$-dimensional achronal 
    space-like  $C^2$ submanifold $S$.
\end{enumerate}

Then there exists an inextendible (on the right),  null geodesic $\gamma:[0,b)\to \R$, $b<\infty$. 

Moreover, $\gamma$ is maximizing, i.e., $\tau_{g}(\gamma_{s}, \gamma_{t})=0$ for all $s\leq t\in [0,b)$.
\end{corollary}

\begin{proof}  
The existence of a Cauchy hypersurface implies that $(M,g)$ is strongly causal.
We can therefore apply \Cref{R:compatibility-penrose} and endow 
$H:=\partial I^{+}(S)$ with a natural gauge $G$ and the associated rigged measure $\mm$, 
making $(H,G,\mm)$ a synhtetic null hypersurface, which, by 
\Cref{thm:compatibilityNEC-NCe}, satisfies $\NC^e(n)$.
As observed after~\eqref{eq:defGhatH},
  $G$ can be continuously extended to $0$ on the set
  $\initial$ of initial points of the null generators of $H$, thus it satisfies~\eqref{eq:infimum-of-gauge}.  We next claim that
$S$ is 
$(G,\mm)$-future-converging.
Indeed, one can check~\eqref{eq:mean-curvature} by 
performing a second-order Taylor expansion in $\epsilon$ of $\mm(A^+_\epsilon)$ and then apply the future-converging hypothesis (cf.~\cite[Remark~5.4]{CaMo:20}).
Applying \Cref{Cor:PenroseC0}, yields that $G$ is not proper.
By the first point of \Cref{prop:wgnc}, it follows that $(M,g)$  is not weakly 
future null complete; i.e, there exists an inextendible null geodesic like in the claim which is a global maximizer.
\end{proof}

\section{Stability of the \texorpdfstring{$\NC^e(N)$}{NCe{(N)}} condition}\label{Sec:StabNCe}

In this section we state and prove two stability theorems for the
\NEC. The first one is a null counterpart of Sturm's  stability of the $\CD$ condition~\cite{sturm:I, sturm:II}.

We recall that, in a metric space $(X,\sfd)$ given a sequence of
closed sets $B_{n}$, $n\in\N$ the   Kuratowski-lim-sup
 of $B_{n}$ is
\begin{equation}\label{eq:kuratowski-definition}
    \text{K-}\limsup_{n\to\infty}
    B_{n}
    =
    \{
    x\in X
    \colon 
    \exists
    \text{ a subsequence }
    x_{n_{k}}\in B_{n_{k}}
    \text{ such that }
    x_{n_{k}}\to x
    \}
    .
\end{equation}

We recall the definition of push-up property.

\begin{definition}\label{def:push-up}
  Let $(X,\leq,\ll)$ be a causal space.
  We say that it has the \emph{push-up property}, if $\forall x,y,z\in X$
  \begin{equation}
    x\leq y\ll z
    \text{ or }
    x\ll y\leq z
    \qquad
    \implies
    \qquad
    x\ll z
    .
  \end{equation}
\end{definition}

The push-up property is satisfied in 
locally-Lipschitz Lorentzian manifolds~\cite{Chr-Grant}
and, more generally, in Lorentzian pre-length spaces~\cite{KS} (see also~\cite{Ling} for the case of $C^0$-Lorentzian metrics). 

\begin{theorem}\label{thm:StabNCProb}
  Let $(X,\ll,\leq, \mathfrak{T})$ be a topological causal space 
  satisfying the push-up property and let $N>1$.
  Let $(H_{n},G_{n},\mm_{n})$, $n\in\N\cup\{\infty\}$, be a sequence of 
  synthetic null hypersurfaces, with $\mm_{n}\in\Prob(H_{n})$.
  Assume the following.
  \begin{enumerate}
  \item
    There exist  monotone
   (recall \Cref{D:monotoneplan})
    transport plans $\Lambda_{n}\in
    \Pi(\mm_\infty,\mm_{n})$,
    such that
    \begin{equation}
      \label{eq:convergence-lambda}
       \Lambda_{n}
      \weak
      (\id,\id)_{\sharp}
      \mm_\infty
      ,
      \qquad
      \text{ in duality with continuous and bounded functions}
      .
    \end{equation}
    \item
    For every compact set $K\subset X$, the set of curves
    \begin{equation}
    \label{eq:compactness-gauges}
        \{\gamma\in C([0,1];X)\colon
        \gamma \in C_{G_{n}}\text{ for some 
        $n\in\N$ and }
        \gamma_{i}\in K,\, i=0,1
        \}
        \text{ is precompact.}
    \end{equation}
  \item 
  Denoted by $C_{G_{n}}$  the set of $G_{n}$-causal curves, it 
  holds that
  \begin{equation}\label{eq:kuratowski-gauges}
      \text{\rm K-}\limsup_{n\to\infty}
      C_{G_{n}}
      \subset C_{G_\infty}
      .
  \end{equation}
    \end{enumerate}

    If the synthetic null hypersurfaces $(H_{n},G_{n},\mm_{n})$ satisfy the $\NC^e(N)$
    condition, for all $n\in \N$, then also
    $(H_\infty,G_\infty,\mm_\infty)$ satisfies the $\NC^{e}(N)$
    condition, as well.
\end{theorem}

\begin{remark}
  Assumptions~(2) and~(3) are trivially gauge-invariant, because they
  depend only on the family of $G$-causal functions (and the very
  definition of equivalence between gauges is the coincidence of
  $G$-causal functions).
\end{remark}

\begin{proof}
  Let $\mu_{i}^\infty\in\Prob(H_\infty)$, $i=0,1$, be two probability
  measures, transported one in the other by a plan
  $\hat\pi_\infty\in \Pi_{\leq}(\mu_0^\infty,\mu_1^\infty)$.
  Notice that, if $\mu_{i}^\infty$ is not absolutely continuous
  w.r.t.\ $\mm_\infty$, then the concavity inequality of $\NC^e(N)$ 
  trivializes; therefore we can assume that $\mu_{i}^\infty\ll\mm_\infty$.
\smallskip

\textbf{Step 1}. Definition of the approximations $\mu_{i}^n\in \Prob(H_{n})$, $i=1,2$.  
\\  We disintegrate $\Lambda_{n}$, in the following way
\begin{equation}
  \label{eq:disintegration-Lambda}
    \Lambda_{n}
    =
    \int_{H_\infty}
    \Lambda_{n,x}
    \,
    \mm_\infty(\de x)
    ,
  \end{equation}
  where $\Lambda_{n,x}$ is a probability measure concentrated on
  $\{x\}\times H_{n}$.
  Notice that, by monotonicity of $\Lambda_{n}$, if $x\leq y$,
  then 
  \begin{equation}\label{eq:monotonicity-Lambda}
       (P_2)_{\sharp}\Lambda_{n,x}
     \otimes
      (P_2)_{\sharp}\Lambda_{n,y}
      (J)
      =1
      .
  \end{equation}
  We  push the measures
  $\mu_{i}^\infty$ and $\hat\pi_\infty$, as follows
  \begin{equation}
    \label{eq:defn-mu-n-pi-n}
    \begin{split}
    \mu_{i}^n
    &:=
    \int_{H_\infty}
    (P_2)_{\sharp}\Lambda_{n,x}
    \,
    \mu_{i}^\infty(\de x)
    ,\\
    \hat\pi_{n}&
    :=
    \int_{H_\infty\times H_\infty}
    ((P_2)_{\sharp}\Lambda_{n,x})
    \otimes
    ((P_2)_{\sharp}\Lambda_{n,y})
    \,
    \hat\pi_\infty(\de x\,\de y)
    ,
    \end{split}
  \end{equation}
  where $P_2$ is the projection on the second variable.
  We now check that $\hat\pi_{n}$ is a transport plan between $\mu_0^n$ and
  $\mu_1^n$:
  \begin{align*}
    (P_1)_{\sharp}
    \hat\pi_{n}
    &
      =
    (P_1)_{\sharp}
    \int_{H_\infty\times H_\infty}
    ((P_2)_{\sharp}\Lambda_{n,x})
    \otimes
    ((P_2)_{\sharp}\Lambda_{n,y})
    \,
      \hat\pi_\infty(\de x\,\de y)
    \\
    &
      =
    \int_{H_\infty\times H_\infty}
    (P_1)_{\sharp}
      (
      ((P_2)_{\sharp}\Lambda_{n,x})
    \otimes
      ((P_2)_{\sharp}\Lambda_{n,y})
      )
    \,
    \hat\pi_\infty(\de x\,\de y)
    \\
    &
      =
    \int_{H_\infty\times H_\infty}
      (P_2)_{\sharp}\Lambda_{n,x}
    \,
    \hat\pi_\infty(\de x\,\de y)
      =
    \int_{H_\infty\times H_\infty}
      (P_2)_{\sharp}\Lambda_{n,x}
    \,
      \mu_0^\infty(\de x)
      =
      \mu_0^n
      ,
  \end{align*}
  and analogously one can prove the same for the other projection.
  The next step is to prove that $\hat\pi_{n}(J)=1$:
  \begin{align*}
    \hat\pi_{n}(J)
    &
    =
    \int_{H_\infty\times H_\infty}
    ((P_2)_{\sharp}\Lambda_{n,x})
    \otimes
    ((P_2)_{\sharp}\Lambda_{n,y})
    (J)
    \,
    \hat\pi_\infty(\de x\,\de y)
    \\
    &
    =
    \int_{H_\infty\times H_\infty
    \cap J}
    ((P_2)_{\sharp}\Lambda_{n,x})
    \otimes
    ((P_2)_{\sharp}\Lambda_{n,y})
    (J)
    \,
    \hat\pi_\infty(\de x\,\de y)
    \overset{\text{\eqref{eq:monotonicity-Lambda}}}{=}1
    .
  \end{align*}
\smallskip

\textbf{Step 2}. Proof that $\mu_{i}^n\weak\mu_{i}^\infty$, $i=0,1$.

  We first claim that, up to a subsequence,
\begin{equation}\label{eq:P2LambdatoDelta}
  (P_2)_{\sharp}(\Lambda_{n,x})\weak\delta_{x}, \quad \text{for $\mm_\infty$-a.e.\ $x$.}
  \end{equation}
 Indeed, using assumption~\eqref{eq:convergence-lambda},
  we compute
  \begin{align*}
    0&=
       \int_{X\times X} \sfd\wedge 1
       \;
       \de((\id,\id)_{\sharp}\mm_\infty)
       =
       \lim_{n\to\infty}
       \int_{X\times X} \sfd\wedge 1
       \,
       \de\Lambda_{n}
    \\
    &
       =
       \lim_{n\to\infty}
       \int_{\noending{H}}
       \int_{X\times X} \sfd\wedge 1
       \,
       \de\Lambda_{n,x}
       \,
      \mm_\infty(\de x).
  \end{align*}
  Therefore, the function
  \begin{equation}
    \label{eq:convergence-disintegration}
    x\mapsto
       \int_{X\times X} \sfd\wedge 1
       \,
       \de\Lambda_{n,x}
       =
       \int_{X}
       \sfd(x,y)
       \wedge 1
       \,
       ((P_2)_{\sharp}\Lambda_{n,x})(\de y)
  \end{equation}
  converges in $L^1(\mm_\infty)$ to $0$ (in the last equality we used
  the fact that $\Lambda_{n,x}$ is concentrated on $\{x\}\times H_{n}$).
  Up to a subsequence, it also converges $\mm_\infty$-a.e.\;to $0$,  proving the claim~\eqref{eq:P2LambdatoDelta}.

  Using the claim~\eqref{eq:P2LambdatoDelta}, we next show that $\mu_{i}^n\weak\mu_{i}^\infty$, $i=0,1$. Since the  weak convergence is  metrizable, it is enough to show that any subsequence of $\mu_{i}^n$ converges to $\mu_{i}^\infty$.
  To this aim,  fix $\phi\in C_{b}(X)$ and compute, using~\eqref{eq:P2LambdatoDelta} and
  dominated convergence theorem:
  \begin{align*}
    \lim_{n\to\infty}
    \int_{X}
    \phi(x)
    \,
    \de\mu_{i}^n(\de x)
    &
    =
    \lim_{n\to\infty}
    \int_{\noending{H}}
    \int_{X}
    \phi(y)
    \,
    (P_2)_{\sharp}(\Lambda_{n,x})(\de y)
    \mu_{i}^\infty(\de x)
    \\
    &
   =
    \int_{\noending{H}} \left(
    \lim_{n\to\infty}
    \int_{X}
    \phi(y)
    \,
    (P_2)_{\sharp}(\Lambda_{n,x})(\de y) \right)
      \mu_{i}^\infty(\de x)
    \\
    \overset{\eqref{eq:P2LambdatoDelta}}&{=}
    \int_{\noending{H}}
    \phi(x)\,
    \mu_{i}^\infty(\de x)
    ,
  \end{align*}
  yielding that $\mu_{i}^n\weak\mu_{i}^\infty$, $i=0,1$.

  Arguing as in  of~\cite[Lemma~4.19]{sturm:I}, since 
  $\mu_{i}^n=(\Lambda_{n})_{\sharp}\mu_{i}^\infty$
  and 
  $\mm_{n}=(\Lambda_{n})_{\sharp}\mm_\infty$, we deduce that
  \begin{equation}
    \label{eq:entropy-decreases-push-forward}
    \Ent(\mu_{i}^n|\mm_{n})
    \leq
    \Ent(\mu_{i}^\infty|\mm_\infty)
    .
  \end{equation}

\textbf{Step 3}. Weak  convergence of the  $G_{n}$-causal
dynamical optimal transport plans.
\\  Let now $\nu_{n}\in\OptGeo^{G_{n}}(\mu_0^n,\mu_1^n)$ be the $G_{n}$-causal
dynamical optimal transport plan given by the definition of
$\NC^e(N)$,
 i.e., (recall~\eqref{eq:defSN}, for the definition of $\U_{N}$)
  \begin{equation}
    \label{eq:concavity-entropy-in-stability-proof}
    \U_{N-1}((e_{t})_{\sharp}\nu_{n}|\mm_{n})
    \geq
    (1-t)
    \U_{N-1}((e_0)_{\sharp}\nu_{n}|\mm_{n})
    +
    t
    \U_{N-1}((e_1)_{\sharp}\nu_{n}|\mm_{n})
    .
  \end{equation}
  We show that the sequence $\nu_{n}$ is tight.
  Fix $\epsilon>0$.
  By tightness of $(\mu_{i}^n)_{n}$, there exists a compact set
  $K\subset X$ such that $\mu_{i}^\infty(K)\geq 1-\epsilon$,
  $i=0,1$, for all $n\in \N\cap\{\infty\}$.
  We deduce that
  $\nu_{n}(C_{K})\geq 1-2\epsilon$, where
  $C_{K}\subset C([0,1]; X)$ is the family set
  of $G_{n}$-causal curves with endpoints in $K$.
  By~\eqref{eq:compactness-gauges}, we have that $C_{K}$ is 
    compact, therefore the family $(\nu_{n})_{n}$ is tight.
    By Prokhorov's Theorem, up to taking
    a subsequence, it holds 
    that $\nu_{n}\weak\nu_\infty$, for some probability
    measure $\nu_\infty$.

    We claim that $\nu_\infty$ is
  concentrated on $G_\infty$-causal curves in $H_\infty$.
  Since $(e_0,e_1)_{\sharp}\nu_\infty\in\Pi_\leq(\mu_0^\infty,\mu_1^\infty)$, then $\nu_\infty$ is concentrated on causal curves whose end-points belong to $H_\infty$.
  Since $H_\infty$ is  a closed achronal set, 
  by the push-up property, any causal curve with endpoints in $H_\infty$ 
  lays in $H_\infty$, therefore $\nu_\infty$ is concentrated 
  on causal curves in $H_\infty$.
  Making use of Lemma~\ref{lem:kuratowski-weak} below, we can thus compute
\begin{align*}
    1=
    \limsup_{n\to\infty}
  \nu_{n}(C_{G_{n}})
  \overset{\eqref{eq:kuratowski-weak}}{\leq}
    \nu_\infty\Big(  \text{\rm K-}\limsup_{n\to\infty}
C_{G_{n}}\Big)
    \overset{\eqref{eq:kuratowski-gauges}}{\leq}
    \nu_\infty(C_{G_\infty})
    \leq
    1
    .
\end{align*}

  Using the joint lower semicontinuity of the entropy under weak convergence of probability measures~\cite[Theorem 29.20]{villani:oldandnew}, we deduce that
  \begin{align}
    \label{eq:convergence-sturm-middle-points}
    \liminf_{n\to\infty}
    \Ent((e_{t})_{\#}(\nu_{n})|\mm_{n})
    \geq
    \Ent((e_{t})_{\#}(\nu_\infty)|\mm_\infty).
  \end{align}
    We can thus compute
    \begin{equation*}
      \begin{aligned}
        \U_{N-1} &((e_{t})_{\sharp}\nu_\infty|\mm_\infty)
        \overset{\eqref{eq:convergence-sturm-middle-points}}{\geq}
        \limsup_{n\to\infty}
        \U_{N-1}((e_{t})_{\sharp}\nu_{n}|\mm_\infty)
        \\
        & \quad
          \overset{\eqref{eq:concavity-entropy-in-stability-proof}}{\geq}
          \limsup_{n\to\infty}
          \big(
          (1-t)
          \U_{N-1}((e_0)_{\sharp}\nu_{n}|\mm_{n})
          +
          t
          \U_{N-1}((e_1)_{\sharp}\nu_{n}|\mm_{n})
          \big)
        \\
        & \quad
          \overset{\eqref{eq:entropy-decreases-push-forward}}{\geq}
          (1-t)
          \U_{N-1}((e_0)_{\sharp}\nu_\infty|\mm_\infty)
          +
          t
          \U_{N-1}((e_1)_{\sharp}\nu_\infty|\mm_\infty)
          .
      \end{aligned}
    \end{equation*}
    deducing that $\nu_\infty$ enjoys concavity for the entropy, thus
    $(H_{\infty},G_\infty,\mm_\infty)$ is $\NC^e(N)$.  
\end{proof}

The next two technical lemmas (included for the
reader's convenience) recall two elementary facts linking the
Kuratowski convergence of sets and the weak convergence of probability
measures.

\begin{lemma}\label{lem:kuratovski-characterization}
  Let $(X,\sfd)$ be a metric space.
  Let $(C_{n})_{n}$ be a sequence of subsets of $X$.
  Assume that $\text{\rm K-}\limsup_{n\to\infty} C_{n}\subset C$, for
  some subset $C\subset X$.

  Then, for any compact subset $K\subset X$, the following holds: for every 
  $\epsilon>0$ there exists $k\in\N$  such that
    \begin{equation}
      C_{n}\cap K
      \subset C^{\epsilon}, \quad \text{for all } n>k.
    \end{equation}
\end{lemma}
\begin{proof}
  Assume on the contrary that there exists $\epsilon_0>0$ such that, up
  to passing to a unrelabeled subsequence, there exists a sequence
  $x_{n}\in C_{n}\cap K$, such that $\dist(x_{n},C)>\epsilon_0$.
  Up to a further  subsequence, using the compactness of
  $K$, we have that $x_{n}\to x$, for some $x\in K$.
  Therefore, by definition of Kuratowski
  convergence~\eqref{eq:kuratowski-definition}, we infer that
  $x\in\text{\rm K-}\limsup_{n\to\infty} C_{n}\subset  C$. This contradicts that
   $\dist(x, C)\geq\epsilon_0$.
\end{proof}

\begin{lemma}\label{lem:kuratowski-weak}
  Let $(X,\sfd)$ be a Polish space.
  Let $(\mu_{n})_{n}$ be a sequence of probability
  measures, weakly converging to $\mu$.
  Let $(C_{n})_{n}$ be a sequence of subsets of $X$.
  Then
  \begin{equation}
    \label{eq:kuratowski-weak}
    \limsup_{n\to\infty}
    \mu_{n}(C_{n})
    \leq
    \mu\Big(
    \text{\rm K-}\limsup_{n\to\infty} C_{n}
    \Big)
    .
  \end{equation}
\end{lemma}
\begin{proof}
  Define $C:=\text{\rm K-}\limsup_{n\to\infty} C_{n}$.
  Fix $\epsilon>0$.
  By tightness, there exists a compact set $K\subset X$, such that
  $\mu_{n}(K)\geq 1-\epsilon$, for all $n$.
  From \Cref{lem:kuratovski-characterization}, we deduce that $C_{n}\cap K\subset C^\epsilon$, for all
  $n$ large enough.
  It follows that
  \begin{align*}
    \limsup_{n\to\infty}
    \mu_{n}(C_{n})
    &
    \leq
    \limsup_{n\to\infty}
    \mu_{n}(C_{n}\cap K)
    +\epsilon
    \leq
    \limsup_{n\to\infty}
    \mu_{n}(C^{\epsilon})
    +\epsilon
    \\
    &
    \leq
    \mu(\overline{C^{\epsilon}})
    +\epsilon
    \leq
    \mu(C^{2\epsilon})
    +\epsilon
    ,
  \end{align*}
  having used the upper semicontinuity of measures of a closed set
  under weak convergence.
  By arbitrariness of $\epsilon>0$, since $C$ is closed, we conclude.
\end{proof}

On a closed achronal set $H$ endowed with a gauge $G$, we define the set 
of $(G,\epsilon)$-causal curves as
\begin{align*}
  C_{G,\epsilon}:=\bigg\{
  &
    \gamma:[0,1]\to H
    \text{ causal: }
  \gamma_{r}\in \noending{H},
  \forall r\in(0,1)
    \text{ and }
  \\
  & \; \bigg|
    \frac{G(\gamma_{t})
    -G(\gamma_{s})}{t-s}
    -
    \lim_{u\to 0^{+}}(G(\gamma_{1-u})-G(\gamma_{u}))
    \bigg|
    \leq \epsilon
    \lim_{u\to 0^{+}}(G(\gamma_{1-u})-G(\gamma_{u}))
    ,
    \\
    &
   \;\ \qquad
        \forall t,s\in(0,1)
    \colon t<s
    \bigg\}
    .
\end{align*}
%

We stress out that the limit in the definition above 
exists by monotonicity.
Moreover, the definition is invariant by an affine transformation 
of $G$ (cf.\ \Cref{P:characterization-equivalence}).  

\smallskip

The next result shall be seen as a  null counterpart
   of Lott and
 Villani's~\cite{lottvillani, villani:oldandnew} stability of the $\CD(K,N)$ condition under pointed measured Gromov{--}Hausdorff convergence of pointed metric measure spaces. More precisely, the next theorem establishes a stability result for the $\NC^e(N)$ condition under a suitable pointed measured convergence of (possibly) non-compact synthetic null hypersurfaces endowed with a $\sigma$-finite measure (note that, in the previous stability result, Theorem~\ref{thm:StabNCProb}, the reference measures $\mm_{n}$ were probabilities).  
In order to handle convergence in such a higher generality, it is convenient to consider \emph{pointed}  synthetic null hypersurfaces  $(H,G,\mm, \star)$, where $\star\in H$ is a marked point.

\begin{theorem}\label{thm:StabNCePointed}
Let $N>1$.
  Let $(X_{n},\leq_{n},\ll_{n}, \mathfrak{T}_{n})$, $n\in\N\cup\{\infty\}$
  be a sequence of topological causal spaces.
  Let $(H_{n},G_{n},\mm_{n}, \star_{n})$, $n\in\N\cup\{\infty\}$ be a pointed
  synthetic null hypersurface in $(X_{n},\leq_{n},\ll_{n},\mathfrak{T}_{n})$,
  $n\in\N\cup\{\infty\}$.
  Assume the following.
  \begin{enumerate}
  \item
    There exist Borel and monotone maps $h_{n}: H_{n}\to H_\infty$
    and $g_{n}:H_\infty\to H_{n}$ and an infinitesimal
    sequence $\epsilon_{n}\downarrow 0$, 
    such that $h_{n}(\star_{n})=\star_\infty$
    and $g_{n}(\star_\infty)=\star_{n}$,
    \begin{align}
      \label{eq:weak-hn}
  &
    (h_{n})_{\sharp}\mm_{n}\weak \mm_\infty,
      \\
      \label{eq:assumption-sigma}
  &
    (g_{n})_{\sharp}(\mm_\infty)=\sigma_{n}\mm_{n},
    \qquad
    \text{
    with \qquad 
    }
    \sigma_{n}\leq 1+ \epsilon_{n},
  \\
  &
    \label{eq:compatibility-g-h}
    h_{n}\circ g_{n}(x)
    \to x
    ,
    \qquad
    \text{ for $\mm_\infty$-a.e.\ }
    x\in H_\infty
    .
\end{align}
Moreover $h_{n}$ transforms, by post-composition, $G_{n}$-causal
      curves into $(G_\infty,\epsilon_{n})$-causal
      curves; i.e., 
      \begin{equation}\label{eq:hnGnCaus}
      h_{n}\circ\gamma\in C_{G_\infty,
      \epsilon_{n}},\quad  \text{for all }\gamma\in C_{G_{n}}.
      \end{equation}
\item For every precompact set $K\subset \noending{H}_\infty$ 
and for every $\epsilon>0$,
it holds that the set
\begin{equation}
  \begin{aligned}
    \label{eq:uniform-compactness}
    &\{
    \gamma\in C_{G_\infty,\epsilon}
    \colon
    \gamma_{i}\in K, i=0,1
    \}
    \text{ is precompact}
    .
    \end{aligned}
\end{equation}
\item It holds that\footnote{Since
    $C_{G_\infty,\epsilon}$ is a decreasing sequence of sets, the
    assumption~\eqref{eq:kuratowski-G-epsilon} is equivalent to
    require that
    $C_{G_\infty}= \bigcap_{\epsilon>0}C_{G_\infty,\epsilon}
    =\text{K-}\lim_{\epsilon\to 0} C_{G_\infty,\epsilon}$.  }
\begin{equation}\label{eq:kuratowski-G-epsilon}
    \text{\rm K-}\limsup_{\epsilon\to 0}
    C_{G_\infty,\epsilon}
    \subset
    C_{G_\infty}
    .
\end{equation}
  \end{enumerate}

    If the synthetic null hypersurfaces $(H_{n},G_{n},\mm_{n})$ satisfy the $\NC^e(N)$
    condition, for all $n\in \N$, then also
    $(H_\infty,G_\infty,\mm_\infty)$ satisfies the $\NC^{e}(N)$
    condition, as well.
\end{theorem}

\begin{remark}
    Assumptions~(2) and~(3), as well as assumption~\eqref{eq:hnGnCaus}, are trivially gauge-invariant, because they
  depend only on the family of $G$-causal curves (and the very
  definition of equivalence between gauges implies the coincidence of
  $G$-causal curves).
\end{remark}

\begin{remark} Recently, inspired by the measured Gromov{--}Hausdorff convergence in positive signature,  several notions of convergence for smooth and non-smooth spacetimes appeared in the literature, see e.g.,~\cite{CaMo:20, MiSu:2024, Muller, ByMiSu:2025, SaSo:2025, MoSae:2025}. 
It is an interesting problem, that we do not address here due to length constraint, to compare these with the convergence used in the stability results.
\end{remark}

\begin{proof}

  Let $\mu^\infty_{i}\in \Prob(H_\infty)$, $i=0,1$, be such that $\Pi_\leq(\mu_0^\infty,\mu_1^\infty)\neq\emptyset$. Notice that, if $\mu_{i}^\infty$ is not absolutely continuous
  w.r.t.\ $\mm_\infty$, then the concavity inequality of $\NC^e(N)$ 
  trivializes; therefore we can assume that
  \begin{equation}
    \label{eq:density-mu-infty}
    \mu_{i}^\infty=\rho_{i,\infty}\mm_\infty\ll\mm_\infty
    .
  \end{equation}
  Let
  \begin{equation}
    \label{eq:density-mu}
    \mu^n_{i}:=
    (g_{n})_{\sharp}\mu^\infty_{i}
    =
    \rho_{i,n}
    (g_{n})_{\sharp}
    \mm_\infty
    =
    \rho_{i,n}
    \sigma_{n}
    \mm_{n}
    ,
  \end{equation}
  for some function $\rho_{i,n}$, $n\in \N$.
  Let $\nu_{n}\in\OptGeo^{G_{n}}(\mu_0^n,\mu_1^n)$ be a $G_{n}$-causal
  dynamical transport, given by the definition of $\NC^e(N)$,
 i.e., (recall~\eqref{eq:defSN}, for the definition of $\U_{N}$)
  \begin{equation}
    \label{eq:concavity-entropy-in-stability-novel-proof}
    \U_{N-1}((e_{t})_{\sharp}\nu_{n}|\mm_{n})
    \geq
    (1-t)
    \U_{N-1}((e_0)_{\sharp}\nu_{n}|\mm_{n})
    +
    t
    \U_{N-1}((e_1)_{\sharp}\nu_{n}|\mm_{n})
    .
  \end{equation}  

  \noindent
  {\bf Step 1.} Convergence at $t=0,1$.

  \noindent
Using that $\Ent(\cdot\mid\cdot)$ decreases after push-forward on both
entries (see for instance~\cite[Th.~29.20~(ii)]{villani:oldandnew}), we can estimate
the entropy at the extremals of the interval,
   as follows (assumption~\eqref{eq:assumption-sigma} has been taken
   into account)
  \begin{align*}
    \Ent(\mu_{i}^\infty|\mm_\infty)
    &
    \geq
    \Ent(\mu_{i}^n|(g_{n})_{\sharp}\mm_\infty)
    =
    \int_{H_{n}}
    \log(\rho_{i,n})
    \,
    \de \mu_{i}^n
    \\
    &
    =
    \int_{H_{n}}
    \log(\rho_{i,n}\sigma_{n})
    \,
    \de \mu_{i}^n
    -
    \int_{H_{n}}
    \log{\sigma_{n}}
    \,
    \de \mu_{i}^{n}
    \\
    \overset{\eqref{eq:assumption-sigma}}&{\geq}
    \int_{H_{n}}
    \log(\rho_{i,n}\sigma_{n})
    \,
    \de \mu_{i}^n
    -
    \int_{H_{n}}
    \log(1+\epsilon_{n})
    \,
    \de \mu_{i}^{n}
    \\
    \overset{\eqref{eq:density-mu}}&{=}
    \Ent(\mu_{i}^n|\mm_{n})
    -
    \int_{H_{n}}
    \log(1+\epsilon_{n})
    \,
    \de \mu_{i}^{n}, \quad i=0,1.
 \end{align*}
  Since by assumption $\epsilon_{n}\downarrow 0$, we deduce that
  \begin{equation}\label{eq:limsupEntnueta}
    \liminf_{n\to\infty}
      \U_{N-1}((e_{i})_{\sharp}\nu_{n}|\mm_{n})
      \geq
      \U_{N-1}((e_{i})_{\sharp}\nu_{n}|\mm_{n})
    ,
    \qquad
    i=0,1
    .
  \end{equation}

    We next prove that $(h_{n})_{\sharp}\mu_{i}^n\weak \mu_{i}^\infty$.
    Fix $\phi\in C_{b}(H_\infty)$ and compute, using dominated
    convergence theorem
    \begin{align*}
        \int_{H_{n}}
        \phi
        \,
        \de((h_{n})_{\sharp}\mu_{i}^n)
        &
        =
         \int_{H_{n}}
        \phi
        \,
        \de((h_{n}\circ g_{n})_{\sharp}\mu_{i}^\infty)
        =
        \int_{H_{n}}
        \phi(h_{n}\circ g_{n}(x))
        \,
        \mu_{i}^\infty(\de x)
        \\
        &
          =
         \int_{H_{n}}
        \phi(h_{n}\circ g_{n}(x))
        \, \rho_{i,\infty}(x)\,
        \mm_\infty(\de x) \\
        &
          \qquad
          \overset{\eqref{eq:compatibility-g-h}}{\to}
           \int_{H_{n}}
        \phi(x) \, \rho_{i,\infty}(x)\,
        \mm_\infty(\de x)      
        .
       \end{align*}

  \noindent
  {\bf Step 2.} Construction of the dynamical transport plan.

  \noindent
Let $\hat h_{n}: C_{G_{n}}\to C_{G_\infty, \epsilon_{n}}$
denote the post-composition by $h_{n}$,
(i.e.,
$\hat{h}_{n}(\gamma)=h_{n}\circ \gamma$)
  and define
$\eta_{n}:= (\hat h_{n})_{\sharp}\nu_{n}$.
    We claim that the sequence $(\eta_{n})_{n}$ is tight.
    Up to taking a subsequence, we can assume $\epsilon_{n}$ to be
    decreasing, therefore, by hypothesis, $\eta_{n}$ is concentrated on $C_{G_\infty,\epsilon_{m}}$, if $n\geq m$.
    Since $(h_{n})_{\sharp}\mu_{i}^n\weak \mu_{i}^\infty$, then the family $\{(h_{n})_{\sharp}\mu_{i}^n\}_{i=0,1, n\in \N}$ is tight. For every $\delta>0$, let $K_\delta \subset H_\infty$ be a compact set, such that 
    \[
    (e_{i})_{\sharp}\eta_{n}(K_\delta)=(h_{n})_{\sharp}\mu_{i}^n(K)\geq
    1-\delta
    , \quad
    i=0,1,\quad  \text{for all }n\in \N.
    \]
    Define 
    \[
    C^{K_{\delta}}_{G_\infty,\epsilon_{m}}
    :=
    \{
    \gamma\in C_{G_\infty,\epsilon_{m}}
    \colon
    \gamma_{i}\in K_\delta,\; i=0,1
    \}
    .
    \]
    By hypothesis~\eqref{eq:uniform-compactness},
    $C^{K_{\delta}}_{G_\infty,\epsilon_{m}}$ is precompact and it 
    holds that $\eta_{n}(C^{K_\delta}_{G_\infty,\epsilon_{m}})\geq 1-2\delta$.
    Therefore the family $(\eta_{n})_{n}$ is tight. By Prokhorov's Theorem, up to taking a subsequence, $\eta_{n}\weak\eta_\infty$, for some measure $\eta_\infty$.
    Since $\eta_{n}$ is concentrated on $C_{G_\infty,\epsilon_{m}}$, $n>m$,
    then $\eta_\infty$ is concentrated on 
    $\overline{C_{G_\infty,\epsilon_{m}}}$,
    for all $m$. Therefore, 
    by~\eqref{eq:kuratowski-G-epsilon} we have that 
    $\eta_\infty$ is concentrated on $C_{G_\infty}$.

  \noindent
  {\bf Step 3.} Convergence at $t\in(0,1)$.

  \noindent
Using again~\cite[Thm.\;29.20 (ii)]{villani:oldandnew}, we infer that
    \begin{align*}
      \Ent((e_{t})_{\sharp}\nu_{n}|\mm_{n})
      &\geq
      \Ent((h_{n})_{\sharp}
      (e_{t})_{\sharp}\nu_{n}|
      (h_{n})_{\sharp}\mm_{n})
      \\
      &
        =
      \Ent(
      (e_{t})_{\sharp}
      (\hat h_{n})_{\sharp}
      \nu_{n}|
      (h_{n})_{\sharp}\mm_{n})
      =
      \Ent(
      (e_{t})_{\sharp}
      \eta_{n}|
        (h_{n})_{\sharp}\mm_{n})
        .
    \end{align*}
    Using the joint lower-semicontinuity of the entropy under weak
    convergence and assumption~\eqref{eq:weak-hn}, we deduce that
    \begin{equation}\label{eq:liminfEntnueta}
      \liminf_{n\to\infty}
      \Ent((e_{t})_{\sharp}\nu_{n}|\mm_{n})
      \geq
      \liminf_{n\to\infty}
      \Ent(
      (e_{t})_{\sharp}
      \eta_{n}|
      (h_{n})_{\sharp}\mm_{n})
      \geq
      \Ent(
      (e_{t})_{\sharp}
      \eta_\infty
      |
      \mm_\infty
      )
      .
    \end{equation}
    Combining,~\eqref{eq:limsupEntnueta} and~\eqref{eq:liminfEntnueta},
    we  can pass to the
  limit~\eqref{eq:concavity-entropy-in-stability-novel-proof},
  concluding the proof (see the last steps of the proof of
  \Cref{thm:StabNCProb} for the details).
  \end{proof}

\bibliographystyle{acm}
\bibliography{literature.bib}
\end{document}